\documentclass[11pt,reqno]{amsart}
\usepackage{amsmath,amsfonts,amsthm,amscd,amssymb,graphicx}
\usepackage{epstopdf,bbold}
\usepackage{subcaption}
\numberwithin{equation}{section}
\usepackage{fullpage}

\usepackage{color}
\makeatletter
\newcommand{\doublewidetilde}[1]{{%
		\mathpalette\double@widetilde{#1}%
}}
\newcommand{\double@widetilde}[2]{%
	\sbox\z@{$\m@th#1\widetilde{#2}$}%
	\ht\z@=.9\ht\z@
	\widetilde{\box\z@}%
}
\makeatother

\def\eps{\varepsilon }
\def\e{\varepsilon}

\newcommand\R{\mathbb R}

\def\eps{\varepsilon}
\def\e{\varepsilon}

\newcommand\br{\begin{remark}}
\newcommand\er{\end{remark}}
\newcommand\brs{\begin{remarks}}
\newcommand\ers{\end{remarks}}
\newcommand\bp{\begin{pmatrix}}
\newcommand\ep{\end{pmatrix}}
\newcommand{\be}{\begin{equation}}
\newcommand{\ee}{\end{equation}}
\newcommand\ba{\begin{equation}\begin{aligned}}
\newcommand\ea{\end{aligned}\end{equation}}

\newcommand{\bap}{\begin{app}}
\newcommand{\eap}{\end{app}}
\newcommand{\begs}{\begin{exams}}
\newcommand{\eegs}{\end{exams}}
\newcommand{\beg}{\begin{example}}
\newcommand{\eeg}{\end{exaplem}}
\newcommand{\bpr}{\begin{proposition}}
\newcommand{\epr}{\end{proposition}}
\newcommand{\bt}{\begin{theorem}}
\newcommand{\et}{\end{theorem}}
\newcommand{\bc}{\begin{corollary}}
\newcommand{\ec}{\end{corollary}}
\newcommand{\bl}{\begin{lemma}}
\newcommand{\el}{\end{lemma}}
\newcommand{\bd}{\begin{definition}}
\newcommand{\ed}{\end{definition}}

\newcommand{\sign}{{\text{\rm sgn }}}

\usepackage{mathrsfs}  

\newcommand{\Id}{{\rm Id }}

\newcommand{\codim}{{\rm codim\, }}

\newcommand{\erf}{{\rm erf }}

\newcommand{\Res}{{\rm Residue}}
\newcommand{\sgn}{\text{\rm sgn}}
\newtheorem{theorem}{Theorem}[section]
\newtheorem{proposition}[theorem]{Proposition}
\newtheorem{corollary}[theorem]{Corollary}
\newtheorem{lemma}[theorem]{Lemma}

\theoremstyle{remark}
\newtheorem{remark}[theorem]{Remark}
\newtheorem{remarks}[theorem]{Remarks}
\theoremstyle{definition}
\newtheorem{definition}[theorem]{Definition}

\newtheorem{example}[theorem]{Example}

\newcommand{\bbR}{{\mathbb{R}}}

\newcommand{\bbC}{{\mathbb{C}}}

\usepackage[labelformat=simple]{subcaption}

\newcommand{\lb}{\label}

\newcommand{\ran}{\text{\rm{ran}}}

\newcommand{\dom}{\text{\rm{dom}}}

\newcommand{\beq}{\begin{equation}}
\newcommand{\eeq}{\end{equation}}

\title{
Spectral decomposition and decay to grossly determined solutions for a simplified BGK model
}

\author{Alim Sukhtayev}
\address{Miami University, Oxford, OH 45056}
\email{sukhtaa@miamioh.edu}
\thanks{Research of A.S. was partially supported
	under NSF grants no. DMS-1910820.}
\author{Kevin Zumbrun}
\address{Indiana University, Bloomington, IN 47405}
\email{kzumbrun@indiana.edu} 
\thanks{Research of K.Z. was partially supported
under NSF grants no. DMS-0300487 and DMS-0801745.}

\begin{document}

\begin{abstract}
Extending work of Carty, we show that $H^1$ solutions of a simplified 1D BGK model 
decay exponentially in $L^2$ to a subclass of the class of 
grossly determined solutions as defined by Truesdell and Muncaster. 
In the process, we determine the spectrum and generalized eigenfunctions of the associated non-selfadjoint linearized
operator and derive the associated generalized Fourier transform and Parseval's identity. 
Notably, our analysis makes use of rigged space techniques originating from quantum mechanics,
as adapted by Ljance and others to the nonselfadjoint case.
\end{abstract}

\date{\today}
\maketitle

\tableofcontents

\section{Introduction}

In this paper, extending a line of inquiry initiated by Carty \cite{C16krm,C17},
we consider the spectral decomposition and decay to grossly determined solutions of a simplified BGK model
\begin{equation}\lb{bgk}
\frac{\partial f}{\partial t}(t,x,v)+v\frac{\partial f}{\partial x}(t,x,v)=-f(t,x,v)+\int_{\bbR}w(r)f(t,x,r)dr,
\end{equation}
where $w(v)=e^{-v^2}/\sqrt{\pi}$ and the unknown function
$f(t, x, v)w(v)$ represents the molecular density function of a rarefied gas,
with $f(x,t,\cdot)w(\cdot)$ corresponding to the probability distribution of velocities $v$ at point $(x,t)$.
The linear integro-partial differential equation \eqref{bgk} was derived by Cercignani in the context of the slip-flow problem \cite{Ce}
as a decoupled density equation by reduction from the full 1D BGK model, itself a simplification of Boltzmann's equation.
Whereas Boltzmann's equation or the full BGK model has $5$ conserved moments, corresponding to mass, momentum, and energy, \eqref{bgk} has only 
one, corresponding to mass: namely, the integral $\rho(x,t):=\int_{\bbR}w(r)f(x,t,r)dr$.

The notion of grossly determined solutions of a kinetic equation, introduced by Truesdell and Muncaster \cite{TM} in the context of Boltzmann's equation, 
consists of a manifold of solutions that is invariant under the time-evolution of the equation, each solution of whose evolution is determined entirely 
by the spatial distribution of its (finitely many) conserved moments.
It was conjectured in \cite{TM} that all solutions of an appropriately defined class converge time-asymptotically to some class of grossly determined solutions,
the evolution of which, depending only on the macroscopic fluid-dynamical quantities corresponding to moments of the kinetic equation,
can be considered as a nonlocal generalization of the classical compressible Navier-Stokes equations.
See \cite{C16krm} for further discussion.

In the suggestive pair of papers \cite{C16krm,C17}, Carty by a combination of
Case's method of elementary solutions \cite{Ca,Ce}, Fourier transform techniques, and direct
first-principles computation, addressed for model \eqref{bgk} the questions of existence and
converence to grossly determined solutions, obtaining a number of interesting results.
In particular, for Fourier modes $\xi\in (-\sqrt{\pi}, +\sqrt{\pi})$, he deduces the spectra
of the generator $L_\xi$ of the associated Fourier transformed evolution equation, and computes associated continuous and point eigenfunctions.
He observes that superpositions of point eigenfunctions make up a family 
of grossly determined solutions \cite[Main Theorem]{C16krm}, 
while continuous eigenfunctions, having spectra with uniformly negative real part, 
are exponentially decaying.
From these observations, he concludes \cite[Discussion, Section 6]{C17} that
{\it elementary solutions}, defined as solutions with initial data $f_0$ having 
Fourier transform $\hat f_0(\xi)$ compactly supported in 
$ (-\sqrt{\pi}, +\sqrt{\pi})$ and satisfying mild additional assumptions (detailed in
\cite[Thm. 10]{C17}),
decay to the class of grossly determined solutions given by superpositions of point eigenfunctions.

However, Carty stops short of stating a precise theorem on convergence. 
In particular, in the absence of a Parseval type identity, it is unclear in what norm the 
(transient) continuous eigenmodes might decay, other than an artificial one defined in terms of
the spectral decomposition itself.
Moreover, the issue of completeness of the spectral decomposition is left unaddressed, even in the class of (necessarily real analytic) solutions with compactly supported Fourier transform considered in
\cite{C17}.
Thus, there are a number of interesting questions left for further investigation in this work.

The purpose of the present paper is to address these remaining questions: more precisely, to place Carty's
specialized computations in a larger functional analytic framework, from which we can then determine
completeness, convergence, etc. in a systematic way.
The framework that we find useful for this problem is the theory of {\it rigged spaces} (e.g., \cite{N60,L70,R83})
introduced in quantum mechanics for the study of spectral decomposition in the presence of continuous spectrum,
and especially the work of Ljance \cite{L70} on ``completely regular'' 
(or ``small-dimensional'' in a certain prescribed sense) perturbations of multiplication operators.

Denote by $L^2_w$ the weighted $L^2$ space of functions of the velocity $v$, with associated inner product
$ (f,g)_{L^2_w}:=\int_{\R}  f(v) \overline{g(v)}w(v)dv$.  Denote by $\mathbb{1}$ the ``unit''
function $\mathbb{1}(x,v)\equiv 1$.
Then, we may rewrite \eqref{bgk} as 
\be\label{rew} 
\partial_tf= Lf,
\; \hbox{\rm where $Lf:= -(v\partial_x + 1)f + ( f, \mathbb{1})_{L^2_w}\mathbb{1}$}.
 \ee
 Evidently, $\mathbb{1}(x,\cdot)\in L^2_w$ is the Maxwellian, or equilibrium state for which the collision operator
 $Q(f):= ( f, \mathbb{1})_{L^2_w}\mathbb{1}-f$ over $L^2_w$ (i.e., the righthand side of \eqref{bgk}) vanishes.

 Taking the Fourier transform of $f$ in $x$, following Carty \cite{C16krm,C17}, we obtain
 \begin{equation*}
 \partial_t \hat f=L_\xi:= -(iv\xi +1) \hat f+ ( \hat f, \mathbb{1})_{L^2_w}\mathbb{1},\,\,\mathbb{1}(v)\equiv 1\in L^2_w(\R),
 \end{equation*}
 where $\hat f$ denotes Fourier transform in $x$ and $\xi$ is the associated Fourier frequency, reducing the problem
 to the study of \eqref{feq} and the spectral decomposition of $L_\xi$ (for ease of writing we use the same symbol $\mathbb{1}$ to denote the constant function over $x,v$ and over $v$ alone).
 Noting that $L_\xi$ decomposes into the sum of the multiplication operator $S_\xi:=-(iv\xi+1)$ and the rank-one
 operator $V:=( \cdot, \mathbb{1})_{L^2_w}\mathbb{1}$, we find ourselves, finally,
 in the setting studied by Ljance \cite{L70}.

 The spectral decomposition of multiplication operator $S_\xi$ consists entirely of essential spectrum,
 with associated generalized eigenfunctions given by delta distributions.
 The content of \cite{L70}, roughly speaking, is that the spectral decomposition of the perturbed operator
 $L_\xi=S_\xi + V$-- or, indeed, any such ``small-dimensional'' perturbation of a multiplication operator
 $S_\xi$ -- consists of the same set of essential spectra as $S_\xi$, 
 with associated generalized eigenfunctions given by generalized delta-functions (in particular, again diagonalizable), 
 together with a (possibly empty) set of isolated point spectra of finite multiplicity.
 These eigenmodes are shown to be complete in the sense that the associated forward and backward
 generalized Fourier transforms satisfy a ``generalized Parseval inequality'' for functions in appropriately restricted domains, 
 relating Hilbert inner product of two functions to their spectral expansions in terms of generalized eigenfunctions,
 or ``generalized Fourier transforms.''
 Moreover, there is presented a calculus based on analytic continuation by which eigenmodes may be
 represented, and in principle computed or estimated.

Here, applying the abstract formalism of \cite{L70} to the rank-one perturbation \eqref{rew}, we show
that the left and right continuous eigenmodes of $L_\xi$ may be computed explicitly, and associated
forward and backward generalized
Fourier transforms $\mathcal{U}_\xi$ and $\mathcal{B}_\xi$ as multiples of a Hilbert transform-- see \eqref{hilbert}.
For Fourier frequencies $\xi\neq \pm \sqrt{\pi}$, this yields estimates
$$
\hbox{\rm 
$\|\mathcal{U }_\xi g\|_{L^2_{w}}\leq C_1(\xi)\|g\|_{L^2_w}$ and $\|\mathcal{B}_\xi f\|_{L^2_{w}}\leq C_2(\xi)\|f\|_{L^2_w}$},
		$$
with constants $C_j$ depending on $\xi$.
For general $\xi$, we have the uniform estimates 
$$
\hbox{\rm 
$ \|\mathcal{U }_\xi g\|_{H^{-1}_{w}}\leq C_1\|g\|_{H^1_w}$ and $\|\mathcal{B}_\xi f\|_{L^2_{w}}\leq C_2\|f\|_{L^2_w}$;}
$$
see Proposition \ref{UB_norm}.  Similar estimates hold for discrete eigenmodes induced by the
rank one perturbation, as encoded by an associated projector $P_{\lambda^*}(\xi)$; see Corollary \ref{P_norm}. 
Together with the generalized Parseval inequality, these yield respectively completness of the spectral
decomposition with respect to $L^2_w$ and uniform bounds from $H^1_w$ to $L^2_w$ of the solution operator 
for \eqref{bgk}; see Theorems \ref{MainThm} and \eqref{main_expension_thm}.
From the latter, we obtain rigorous $H^1_w\to L^2_w$ time-exponential decay bounds on the continuous part of the 
spectral decomposition, yielding time-exponential decay to grossly determined solutions in $L^2_w$ for data 
in $H^1_w$, at the sharp rate $O(e^{-t})$; see Theorems \ref{main1}--\ref{main3}.

An auxiliary argument based on Pr\"uss' Theorem and $C_0$ semigroup estimates gives exponential decay from
$L^2_2\to L^2_w$ at a lesser rate $O(e^{(\epsilon-1)t})$, $\epsilon>0$, to a different grossly determined solution
consisting of an appropriate Fourier truncation of the grossly determined solution of Theorems
\ref{main2}--\ref{main3}; see Theorem \ref{main4}.
In passing, we establish a convenient finite-codimension parametric version of Pr\"uss' theorem, 
Corollary \ref{modpruss}, that appears of independent interest.
We note that in both settings- 
the rigged-space framework of our first set of results, and the semigroup framework of the second-
the appearance of an additional, unbounded, parameter given by the Fourier frequency $\xi$, significantly complicates the analysis by the need for uniformity of all estimates with respect to $\xi$.

These results rigorously recover and in complete the analysis begun in \cite{C16krm,C17}.
The methods used in this rank-one case would appear to apply to any finite-rank perturbation.
In particular, it should apply to the linearization of the full BGK equation, which has a rank-$5$ linearized collision operator corresponding to projection onto the tangent space of the $5$-dimensional manifold of Maxwellians, to yield a similar
result of exponential decay to grossly determined solutions. A very interesting open problem would be to determine the implications as regards decay to grossly determined solutions for the full nonlinear equation.

Another very interesting direction is the study of the full Boltzmann equation with hard potential, for
which \cite{G62} the associated linearized operator $L_\xi$ is a {\it compact} perturbation, hence arbitrarily well approximated by finite-rank ones.
This would be interesting not only from the standpoint of grossly determined solutions, but also
of explicit description of the spectral decomposition of the linearized operator, hopefully giving
detailed estimates like \cite{BM05}, or $L^\infty$ resolvent estimates as conjectured in
\cite{Z17} (see also \cite{PZ16}.

From the standpoint of general theory, our analysis provides a very interesting case study for the 
nonselfadjoint rigged space
framework of \cite{L70}, for which 
essentially all spectral computations can be explicitly carried out-
see the computations of essential and discrete spectra in Theorem
\ref{s:7} and Proposition \ref{discreteprop}-
and the first to our knowledge
in which the theory is applied to obtain time-asymptotic bounds for an interesting physical system.
Moreover, the results highlight what seems to us a fundamental direction for further development of the rigged
space approach to behavior of nonselfadjoint systems, namely, the issue of loss of 
derivatives/unbounded condition number of $\mathcal{U}_\xi$, an extreme version of nonunitarity of
eigenvases for nonselfadjoint operators in general.
Different from the selfadjoint case discussed, e.g., in \cite{A69,A96}, this means that sharp evolutionary behavior
is not immediately obtained from the spectral decomposition of the generator, but may, as here, involve substantial
cancellation between modes, an issue standardly treated by the use of {\it resolvent bounds} in place of exact
spectral decomposition.
The questions suggested here are whether (i) cancellation may (at least in some cases)
instead be detected directly from a very explicit description of the spectral expansion, thus combining
the useful aspects of detailed eigenexpansion and control of conditioning, and
(ii) resolvent estimates or some analog may be obtained from the rigged space formulation itself.
These aspects are discussed further in Sections \ref{s:semigp} and \ref{s:disc}.
\medskip

{\bf Plan of the paper.}
In Section \ref{rigged}, we recall the rigged space framework of \cite{L70}, 
which we use in Sections \ref{s:3}--\ref{s:7} to obtain a detailed spectral decomposition 
of the (linear) scalar BGK model \eqref{bgk}.
In Section \ref{s:evol}, we use the spectral decomposition to obtain existence and decay to grossly determined solutions
of solutions of \eqref{bgk}, with a loss of two derivatives in velocity $v$.
This is repaired in Section \ref{s:semigp} using a different, $C_0$ semigroup argument requiring less spectral detail, 
but giving a lower exponential rate.
We conclude the main text with discussion and perspectives in Section \ref{s:disc}.
Finally, an exact computation of discrete spectrum is given in Appendix \ref{s:discrete}.


\section{Rigged Spaces}\label{rigged}

We consider the Hilbert space $L_w^2(\bbR)$, where $w(x)=e^{-x^2}/\sqrt{\pi}$ will be denoted by $H$. Let
\begin{equation*}
\Omega=\Omega_\e=\{z\in\bbC:|\Im z|<\e\}, \,\,\e>0.
\end{equation*}
\begin{definition}
	An element $\phi\in H$ is called {\it regular} if there exists an extension $z\to\phi(z)$ from the real $x$-axis to the complex $z$-plane, which is holomorphic in $\Omega$. A regular element $\phi\in H$ is called {\it completely regular} if for each $\e_1\in[0,\e)$ the following holds
	\begin{equation*}
	\sup_{y\in[-\e_1,\e_1]}\int_{\bbR}|\phi(x+iy)|^2w(x)dx<\infty.
		\end{equation*}
		The linear spaces of regular and completely regular elements will be denoted by $\Phi_0$ and $\Phi$ respectively.
\end{definition}
The space $\Phi$ (of completely regular elements) is regarded as a sequential
Hilbert space with topology defined by the norms:
\begin{equation*}
||\phi||_y:=\Big\{\frac{1}{2}\int_{\bbR}[|\phi(x+iy)|^2+|\phi(x-iy)|^2]w(x)dx\Big\}^{1/2},\,\,y\in[0,\e).
\end{equation*}
We denote by $\Phi^*$ the space of semilinear continuous functionals defined on $\Phi$, and by $\langle\phi^*,\phi\rangle$ the
value of the functional $\phi^*\in\Phi^*$ at the point $\phi\in\Phi$. For each $\phi^*\in\Phi^*$ there exists $y\in[0,\e)$ such that
\begin{equation}\lb{n1}
|\phi^*|_{-y}=\sup_{0\neq\phi\in\Phi}\frac{|\langle\phi^*,\phi\rangle|}{||\phi||_y}<\infty.
\end{equation}
And, we have the following embeddings:
\begin{equation}\lb{rigging}
\Phi\subset H\subset\Phi^*.
\end{equation} 

We shall now find an analytic representation of the functional $\phi^*\in\Phi^*$. These
functionals will be called generalized elements of the space $H$.

In particular, if the sesquilinear form
\begin{equation*}
{}_{\Phi^*}\langle\cdot, \cdot\rangle_{\Phi} : \Phi^*\times\Phi\to \bbC
\end{equation*}
denotes the duality pairing between $\Phi^*$ and $\Phi$, then
\begin{equation}\label{inclusion}
{}_{\Phi^*}\langle f, \phi\rangle_{\Phi} =(f,\phi)_{L^2_w(\bbR)},\,\,\,f\in L^2_w(\bbR),\,\phi\in\Phi,
\end{equation}
that is, the pairing ${}_{\Phi^*}\langle\cdot, \cdot\rangle_{\Phi}$ is compatible with the inner product in $L^2_w(\bbR)$. Let $R\in\mathcal{B}(\Phi,\Phi^*)$. Since $\Phi$ is reflexive, $(\Phi^*)^*=\Phi$, one has 
\begin{equation*}
R:\Phi\to\Phi^*,\,\,\,\,\,R^*:\Phi\to\Phi^*.
\end{equation*}
Moreover, $R^*$ is defined as follows
\begin{equation}\label{defadj}
{}_{\Phi^*}\langle R\phi, \psi\rangle_{\Phi}={}_{(\Phi^*)^*}\langle \phi, R^*\psi\rangle_{\Phi^*}={}_{\Phi}\langle \phi, R^*\psi\rangle_{\Phi^*}=\overline{{}_{\Phi^*}\langle R^*\psi, \phi\rangle_{\Phi}}.
\end{equation}
From this point on, we will drop subscripts in \eqref{defadj}.
\begin{definition}
	We denote by $\Phi^*_{-\eta}$ the class of complex-valued functions $z\to\phi^*(z)$ which
	are holomorphic for $|\Im z| > \eta$ and satisfy the condition
	\begin{equation*}
	\sup_{|y|>\eta}\int_{\bbR}|\phi^*(x+iy)|^2w(x)dx<\infty.
	\end{equation*}
\end{definition}

For each function $\phi^*\in\Phi^*_{-\eta}$ the boundary values
$\phi^*(x\pm i\eta)$ exist for almost all $x\in\bbR$. In addition,
\begin{equation*}
|\phi^*(z)w^{1/2}(z)|\to0\,\,\hbox{as}\,\,|z|\to\infty
\end{equation*}
uniformly in the region $|\Im z|>\eta$.
\begin{definition}
	For arbitrary $\phi\in\Phi$ and $\phi^*(\cdot)\in\Phi^*_{-\eta}$ we define
	\begin{align}\lb{fun}
	\begin{split}
	\langle\phi^*,\phi\rangle:&=\int_{\gamma}\phi^*(z)\overline{\phi(\overline z)}w(z)dz,\\
	\langle\phi,\phi^*\rangle:&=\overline{\langle\phi^*,\phi\rangle},
	\end{split}
	\end{align}
	where 
	\begin{equation}\label{congamma}
	\int_{\gamma}:=\int_{-\infty-i\gamma}^{\infty-i\gamma}-\int_{-\infty+i\gamma}^{\infty+i\gamma}
	\end{equation}
	and $\gamma$ is an arbitrary number in $[\eta,\e)$.
\end{definition}
Clearly, \eqref{fun} defines a generalized functional $\phi^*\in\Phi^*$. We say that the function
$\phi^*(\cdot)$ in \eqref{fun} represents the functional $\phi^*$. Moreover, one can show that for each generalized element there exists a unique representing function.
\begin{remark}
	Let $\phi^*\in\Phi^*$ be represented by $\phi^*(\cdot)$. We define
	\begin{equation}\lb{n2}
	||\phi||_{-\eta}:=\Big\{\frac{1}{2}\int_{\bbR}[|\phi^*(x+i\eta)|^2+|\phi^*(x-i\eta)|^2]w(x)dx\Big\}^{1/2}.
	\end{equation}
	Using the Fourier transform, one can show that the norm \eqref{n1} is equivalent to the norm \eqref{n2}.
\end{remark}

\section{The simplified BGK model}\label{s:3}

We consider now the linear integro-partial differential equation 
\begin{equation*}
\frac{\partial f}{\partial t}(t,x,v)+v\frac{\partial f}{\partial x}(t,x,v)=-f(t,x,v)+\int_{\bbR}w(r)f(t,x,r)dr,
\end{equation*}
a simplification of the 1D Boltzmann equation for the slip-flow problem \cite{Ce},
where $f(t, x, v)w(v)$ represents the molecular density function of a rarefied gas.
%
Here, the role of  $\phi\in H$ in the previous section is played by
the unknown $f(v)\in L^2_w$, with weight $w(v)=e^{-v^2}/\sqrt{\pi}$ corresponding to a Maxwellian distribution.

\subsection{Associated Spectral Problem}
We first take the Fourier transform of equation \eqref{bgk} in the spatial variable, that is,
\begin{equation}\lb{bgk1}
\frac{\partial \hat f}{\partial t}(t,\xi,v)=-vi\xi{\hat f}(t,\xi,v)-\hat f(t,\xi,v)+\int_{\bbR}w(r)\hat f(t,\xi,r)dr.
\end{equation}
Next, we introduce the following operator associated with the right-hand side of \eqref{bgk1} and acting in the weighted $L^2$ space $L^2_w$, which is the standard one considered for Boltzmann's equation \cite{G62}.
\begin{align*}
\begin{split}
(L g)(\xi,v):&=-vi\xi g(\xi,v)-g(\xi,v)+\int_{\bbR}w(r)g(\xi,r)dr,\\
g\in\dom(L)&=\{h\in L^2(\bbR^2;w(v)d\xi dv):L g\in L^2(\bbR^2;w(v)d\xi dv)\},
\end{split}
\end{align*}
and
\begin{align*}
\begin{split}
(L_{\xi} g)(v):&=-vi\xi g(v)-g(v)+\int_{\bbR}w(r)g(r)dr,\\
g\in\dom(L_\xi)&=\{h\in L^2_w(\bbR;dv):L_\xi g\in L^2_w(\bbR; dv)\}.
\end{split}
\end{align*}
We also decompose $L_{\xi}$ into the sum $M_\xi$ and $V$, i.e.
\begin{align}\lb{L}
\begin{split}
L_{\xi} :&=M_\xi+V,\,\,\dom(L_{\xi})=\dom(M_\xi),\\
(M_\xi g)(v):&=-vi\xi g(v)-g(v),\\
g\in\dom(M_\xi)&=\{h\in L^2_w(\bbR;dv):M_\xi h\in L^2_w(\bbR; dv)\},\\
(Vg)(v):&=(g,\mathbb{1})_{L^2_w}\mathbb{1},\,\,\mathbb{1}=1 \,\,\hbox{for a.e.}\,\,v\in\bbR .
\end{split}
\end{align}
Notice that $V^2=V^*=V\in\mathcal{B}(L^2_w(\bbR;dv))$.

\subsection{Spectrum of $L_{\xi}$}

\begin{definition}
	A closed operator $M\in\mathscr{C}(X)$, where $X$ is a Banach space is said to be semi-Fredholm if $\ran(M)$ is closed and at least one of $\dim\ker(M)$ and $\codim\ran(M)$ is finite.
\end{definition}

\begin{definition}[Kato]\label{def}
	Let $\Delta$ be the set of all complex numbers $\lambda$ such that $M-\lambda I$ is semi-Fredholm. The essential spectrum of $M$ denoted by $\sigma_{ess}(M)$ is the set of all complex numbers that are in the complementary set of $\Delta$, that is,
	\begin{multline*}
	\sigma_{ess}(M):=\{\lambda\in\bbC: \hbox{either $\ran(M-\lambda I)$ is not closed or}\\ 
		\hbox{$\ran(M-\lambda I)$ is closed, but $\dim\ker(M-\lambda I)=\codim\ran(M-\lambda I)=\infty$}\}.
	\end{multline*} 
\end{definition}

\begin{remark}
	In general $\Delta$ from Definition \ref{def} is the union of a countable number of components $\Delta_n$, and $\nu_n(\lambda):=\dim\ker(M-\lambda I)$, $\mu_n(\lambda):=\codim\ran(M-\lambda I)$ are constant in each $\Delta_n$ except for an isolated set of values $\lambda_{nj}$ of $\lambda$. If $\nu_n=\mu_n=0$, $\Delta_n$ is a subset of $\rho(M)$ except for the isolated eigenvalues $\lambda_{nj}$ of $M$ with finite algebraic multiplicities. If $\nu_n>0$, $\lambda_{nj}$ behave like 'isolated eigenvalues', in the sense that their geometric multiplicities are larger than geometric multiplicities of eigenvalues that are in their immediate neighborhood.  
\end{remark}
Since $M_\xi$ in \eqref{L} is a multiplication operator, it is clear that $\sigma(M_\xi)=\sigma_{ess}(M_\xi)=\{\lambda\in\bbC:\lambda=-1+i\omega,\omega\in\bbR \}$ for $\xi\neq0$, and $\sigma(M_{0})=\sigma_{ess}(M_{0})=\{-1 \}$. Moreover, since the operator $V$ is $M_\xi$-compact (or relatively compact with respect $M_\xi$), by \cite[Theorem 5.35.]{Kato},  
\begin{align}\lb{ess}
\begin{split}
	\sigma_{ess}(L_{\xi})&=\sigma_{ess}(M_\xi)=\{\lambda\in\bbC:\lambda=-1+i\omega,\omega\in\bbR \},\,\,\xi\neq0,\\
	\sigma_{ess}(L_{0})&=\sigma_{ess}(M_{0})=\{-1 \}.
\end{split}
\end{align}

\begin{definition}[Kato]
	An operator $V$ is said to be $M$-degenerate If $V$ is $M$-bounded and $\dim\ran(V)$ is finite. One can show that a $M$-degenerate operator is $M$-compact.
\end{definition}

 Let $V$ be $M$-degenerate, then for any $\lambda\in\rho(M)$, 
 \begin{align}\lb{K}
 	\begin{split}
 	\tilde K(\lambda):&=I+VR^0(\lambda)=I+V(M-\lambda I)^{-1},\\
 	\tilde{\tilde K}(\lambda):&=I+R^0(\lambda)V=I+(M-\lambda I)^{-1}V
 	\end{split}
 \end{align} are bounded operators in $\mathcal{B}(L^2_w(\bbR))$ and 
 \begin{equation*}
 	\omega(\lambda)=\det(I+V(M-\lambda I)^{-1})=\det(I+(M-\lambda I)^{-1}V)
 \end{equation*}
is well defined. 

Hence, for $M_\xi$ and $V$ from \eqref{L} we can rewrite $\tilde K(\lambda,\xi)$ and $\tilde{\tilde K}(\lambda, \xi)$ from \eqref{K} as
\begin{align*}
\begin{split}
\tilde K(\lambda,\xi)&=I+(\cdot,(R^0(\lambda,\xi))^*\mathbb{1})_{L^2_w}\mathbb{1},\\
\tilde{\tilde K}(\lambda, \xi)&=I+(\cdot,\mathbb{1})_{L^2_w}R^0(\lambda,\xi)\mathbb{1}.
\end{split}
\end{align*} 
We also introduce the operator $K(\lambda, \xi):=I+VR^0(\lambda, \xi)V=I+V(M_\xi-\lambda I)^{-1}V\in \mathcal{B}(L^2_w(\bbR))$.
Notice that for $M_\xi$ and $V$ from \eqref{L}

\begin{align*}
\begin{split}
\omega(\lambda,\xi)=\det(K(\lambda,\xi))=\det(\tilde K(\lambda,\xi))=\det(\tilde{\tilde K}(\lambda,\xi))&=\det(1+(R^0(\lambda,\xi)\mathbb{1},\mathbb{1})_{L^2_w})\\
&=1-\int_\mathbb{R}\frac{w(v)dv}{vi\xi+1+\lambda}.
\end{split}
\end{align*}
In general, $\omega$ and $K$ are meromorphic functions of $\lambda$ in any domain of the complex plane consisting of points of $\rho(M)$ and of isolated eigenvalues of $M$ with finite algebraic multiplicities. In our case, $\omega(\cdot,\xi)$ and $K(\cdot,\xi)$ are analytic functions in any domain of the complex plane consisting of points of $\rho(M_\xi)$.
Let $\mu(\lambda)$ be a numerical meromorphic function defined in a domain $\Delta$ of the complex plane. We define the multiplicity function $v(\lambda,\mu)$ of $\mu$ by
\begin{equation*}
v(\lambda,\mu)=\begin{cases}
k, & \text{if $\lambda$ is a zero of $\mu$ of order $k$},\\
-k, & \text{if $\lambda$ is a pole of $\mu$ of order $k$},\\
0, & \text{for other $\lambda\in\Delta$}.
\end{cases}
\end{equation*}
We also define the multiplicity function $\tilde{\nu}(\lambda,L)$, where $L=M+V$ by
\begin{equation*}
\tilde{\nu}(\lambda,L)=\begin{cases}
0, & \text{if $\lambda\in\rho(L)$},\\
\dim(P), & \text{if $\lambda$ is an isolated point of $\sigma(L)$},\\
+\infty, & \text{otherwise},
\end{cases}
\end{equation*}
where $P$ is the projection associated with the isolated point of $\sigma(M)$.
The following theorem holds
\begin{theorem}
	 Fix $\xi$. Let $M_\xi$ be the operator from \eqref{L} and $\lambda\in\bbC\setminus\sigma_{ess}(M_\xi)$. Then
	\begin{equation*}
	\tilde{\nu}(\lambda,L_{\xi})=v(\lambda,\omega(\cdot,\xi)),\,\,\lambda\in\bbC\setminus\sigma_{ess}(M_\xi).
	\end{equation*}
	Moreover, if $\lambda\in\sigma(L_{\xi})\setminus\sigma_{ess}(L_{\xi})$, then $\lambda$ is a zero of the function $\omega(\cdot,\xi)$ and an isolated eigenvalue of $L_{\xi}$, and the algebraic multiplicity of $\lambda$ as an eigenvalue of $L_{\xi}$ coincides with the order of $\lambda$ as a zero of $\omega(\cdot,\xi)$. Moreover, the operator $K^{-1}(\lambda,\xi)$ exists and is bounded in $H$ for all $\lambda\notin\sigma_{ess}(L_{\xi})$ except for a finite number of $\lambda_j(\xi)$. And $\sigma_d(L_{\xi})=\{\lambda_j(\xi)\}$. Finally, if $\lambda\in\rho(L_{\xi})$, then
		\begin{equation}\lb{perR}
		R(\lambda,\xi)=R^0(\lambda,\xi)-R^0(\lambda,\xi)VK^{-1}(\lambda,\xi)VR^0(\lambda,\xi).
		\end{equation}
	
\end{theorem}
\begin{proof}
	It follows from  \cite[Theorem 6.2]{Kato} after a slight modification.
\end{proof}
Next, we describe the discrete spectrum of $L_{\xi}$.
\begin{proposition}\lb{disspec} For any $\xi\in(-\sqrt{\pi},\sqrt{\pi})$ there exists a unique $\lambda^*(\xi)\in(-1,0]$ such that $\omega(\lambda^*(\xi),\xi)=0$. Moreover, the multiplicity of such $\lambda^*(\xi)$ as a zero of $\omega(\cdot,\xi)$ is one. And if $\xi\in\mathbb{C}\setminus(-\sqrt{\pi},\sqrt{\pi})$, then $\omega(\cdot,\xi)$ does not vanish. Moreover, $\lambda^*(\cdot)$ is a continuous function of $\xi$ for $\xi\in(-\sqrt{\pi},\sqrt{\pi})$,  $-1<\lambda^*(\xi)\leq0$, $\lim_{\xi\to-\sqrt{\pi}^+}\lambda^*(\xi)=-1$ and $\lim_{\xi\to\sqrt{\pi}^-}\lambda^*(\xi)=-1$.
	\end{proposition}
\begin{proof}
	Notice that for any fixed $\xi$, $\omega(\cdot,\xi)$ is defined only for $\lambda\in\bbC\setminus\sigma_{ess}(M_\xi)$. Let $\lambda\in\bbC\setminus\sigma_{ess}(M_\xi)$ and $\lambda=a+ib$, then
	\begin{equation*}
	\int_\mathbb{R}\frac{w(v)dv}{vi\xi+1+\lambda}=\int_\mathbb{R}\frac{w(v)dv}{(v\xi+b)i+1+a}=\int_\mathbb{R}\frac{(1+a-(v\xi+b)i)w(v)dv}{(1+a)^2+(v\xi+b)^2}.
	\end{equation*}
	Notice that if $b:=\Im\lambda\neq0$, then $\int_\mathbb{R}\frac{-(v\xi+b)w(v)dv}{(1+a)^2+(v\xi+b)^2}\neq0$ which implies that $\omega(\cdot,\xi)$ does not vanish. On the other hand, when $b=0$, $\int_\mathbb{R}\frac{-(v\xi)w(v)dv}{(1+a)^2+(v\xi)^2}=0$ (the integrand is an odd function). Hence,
	
	\begin{equation*}
	\int_\mathbb{R}\frac{w(v)dv}{vi\xi+1+\lambda}=\int_\mathbb{R}\frac{(1+\lambda)w(v)dv}{(1+\lambda)^2+(v\xi)^2}.
	\end{equation*}
	Note that if $\lambda\in\mathbb{R}$, then $\lambda\in\bbC\setminus\sigma_{ess}(M_\xi)$ if and only if $\lambda\neq-1$. Next, we analyze the function $\omega(\cdot,\xi)$ for $\lambda\in(-\infty,-1)\cup(-1,\infty)$. In particular, if $\xi\neq0$, by the dominated convergence theorem,
	\begin{align}\lb{omega1}
	\begin{split}
	\lim_{\lambda\to-1^+}\omega(\lambda,\xi)&=\lim_{\lambda\to-1^+}\big(1-\int_\mathbb{R}\frac{(1+\lambda)w(v)dv}{(1+\lambda)^2+(v\xi)^2}\big)\stackrel{v=(1+\lambda)u}{=}\lim_{\lambda\to-1^+}\big(1-\frac{1}{\sqrt{\pi}}\int_\mathbb{R}\frac{e^{-(1+\lambda)^2u^2}du}{1+(u\xi)^2}\big)\\
	&=\big(1-\frac{1}{\sqrt{\pi}}\int_\mathbb{R}\frac{du}{1+(u\xi)^2}\big)=\big(1-\frac{1}{\sqrt{\pi}|\xi|}\arctan(|\xi| u)\Big|^{\infty}_{-\infty}\big)=1-\frac{\sqrt{\pi}}{|\xi|}.
	\end{split}
	\end{align}
	On the other hand, if $\xi=0$
	\begin{align*}
	\begin{split}
	\lim_{\lambda\to-1^+}\omega(\lambda,0)&=\lim_{\lambda\to-1^+}\big(1-\int_\mathbb{R}\frac{(1+\lambda)w(v)dv}{(1+\lambda)^2}\big)=\lim_{\lambda\to-1^+}(1-\frac{1}{1+\lambda})=-\infty.
	\end{split}
	\end{align*}
	Similarly, if $\xi\neq0$
	\begin{align}\lb{omega2}
	\begin{split}
	\lim_{\lambda\to-1^-}\omega(\lambda,\xi)=1+\frac{\sqrt{\pi}}{|\xi|}.
	\end{split}
	\end{align}
	And, if  $\xi=0$, $\lim_{\lambda\to-1^-}\omega(\lambda,0)=\infty$.
By applying the dominated convergence theorem, one can also show that for any $\xi\in\bbR$
\begin{align}\lb{omegainf}
	\begin{split}
		\lim_{|\lambda|\to\infty}\omega(\lambda,\xi)=1.
	\end{split}
\end{align}
	Next, we calculate the derivative of $\omega(\xi,
	\cdot)$ by applying a corollary of the dominated convergence theorem on differentiation under integral sign. If $\lambda>-1$, then
	\begin{align}\lb{deromega1}
	\begin{split}
	\frac{d}{d\lambda}\omega(\lambda,\xi)&=\frac{d}{d\lambda}\big(1-\int_\mathbb{R}\frac{(1+\lambda)w(v)dv}{(1+\lambda)^2+(v\xi)^2}\big)\stackrel{v=(1+\lambda)u}{=}\frac{d}{d\lambda}\big(1-\frac{1}{\sqrt{\pi}}\int_\mathbb{R}\frac{e^{-(1+\lambda)^2u^2}du}{1+(u\xi)^2}\big)\\
	&=-\frac{1}{\sqrt{\pi}}\int_\mathbb{R}\frac{-2(1+\lambda)u^2e^{-(1+\lambda)^2u^2}du}{1+(u\xi)^2}=\frac{2(\lambda+1)}{\sqrt{\pi}}\int_\mathbb{R}\frac{u^2e^{-(1+\lambda)^2u^2}du}{1+(u\xi)^2}>0\\&\,\,\,\hbox{(the integrand is a positive cotinuous function).}
	\end{split}
	\end{align}
	Similarly, if $\lambda<-1$, then
	\begin{align}\lb{deromega2}
	\frac{d}{d\lambda}\omega(\lambda,\xi)=\frac{-2(\lambda+1)}{\sqrt{\pi}}\int_\mathbb{R}\frac{u^2e^{-(1+\lambda)^2u^2}du}{1+(u\xi)^2}>0.
	\end{align}
	By \eqref{omega1}, \eqref{omega2}, \eqref{omegainf}, \eqref{deromega1} and \eqref{deromega2}, we see that for each fixed $\xi\in(-\infty,-\sqrt{\pi}]\cup[\sqrt{\pi},\infty)$ $\omega(\cdot,\xi)$ doesn't vanish for any $\lambda\in(-\infty,-1)\cup(-1,\infty)$ (also, see Figure \ref{fig:sqrtpi}).
	Now, we consider three different cases: $\xi\in(0,\sqrt{\pi})$, $\xi=0$ and $\xi\in(-\sqrt{\pi},0)$.\\
	Case 1. Let $\xi\in(0,\sqrt{\pi})$. From \eqref{omega1}, \eqref{omega2}, \eqref{omegainf}, \eqref{deromega1} and \eqref{deromega2} it is clear that there exists a unique $\lambda^*(\xi)\in(-1, \infty)$ such that $\omega(\lambda^*(\xi),\xi)$=0 (also, see Figure \ref{fig:1}). Next, we rewrite the function $\omega$ in the following way:
	\begin{equation*}
	\omega(\lambda,\xi)=1-\frac{1}{\xi}f\Big(\frac{1+\lambda}{\xi}\Big),
	\end{equation*}
	where $f(x)=\int_\mathbb{R}\frac{xw(v)dv}{x^2+v^2}$.  By \eqref{omega1}, \eqref{omega2} and \eqref{omegainf}, we see that
	\begin{align}\lb{limfx}
	\begin{split}
	\lim_{x\to0^+}f(x)=\sqrt{\pi},\,\,\lim_{x\to0^-}f(x)=-\sqrt{\pi},\,\,\lim_{|x|\to\infty}f(x)=0.
	\end{split}
	\end{align}
	Moreover,
	\begin{equation*}
\frac{d}{d\lambda}\omega(\lambda,\xi)=-\frac{1}{\xi^2}f'\Big(\frac{1+\lambda}{\xi}\Big)
	\end{equation*}
	Therefore, $f$ is a one-to-one differentiable function on $(-\infty,0)\cup(0,\infty)$. 
	Next, for $\xi\in(0,\sqrt{\pi})$, we find a unique $\lambda^*(\xi)\in(-1, \infty)$ such that $\omega(\lambda^*(\xi),\xi)$=0, that is, $0=\omega(\lambda^*(\xi),\xi)=1-\frac{1}{\xi}f\Big(\frac{1+\lambda^*(\xi)}{\xi}\Big)$, or,
	\begin{align*}
	\begin{split}
	f\Big(\frac{1+\lambda^*(\xi)}{\xi}\Big)=\xi,\,\,\hbox{or,}\,\,\,\lambda^*(\xi)=\xi f^{-1}(\xi)-1.
	\end{split}
	\end{align*}
	Hence,
	\begin{align*}
	\begin{split}
	\lambda^*(\xi)&=\xi f^{-1}(\xi)-1=xf(x)-1=\int_\mathbb{R}\frac{x^2w(v)dv}{x^2+v^2}-1>-1,\,\,f^{-1}(\xi)=x,\,\,x\in(0,\infty).
	\end{split}
	\end{align*}
	On the other hand, if $f^{-1}(\xi)=x,\,\,x\in(0,\infty)$, then
	\begin{align*}
	\begin{split}
	\lambda^*(\xi)&=\int_\mathbb{R}\frac{x^2w(v)dv}{x^2+v^2}-1=\int_\mathbb{R}w(v)dv-\int_\mathbb{R}\frac{v^2w(v)dv}{x^2+v^2}-1=-\int_\mathbb{R}\frac{v^2w(v)dv}{x^2+v^2}<0.
	\end{split}
	\end{align*}
	Hence, $\lambda^*(\cdot)$ is a continuous function of $\xi$ for $\xi\in(0,\sqrt{\pi})$ and $-1<\lambda^*(\xi)<0$. Next, we find $\lim_{\xi\to0^+}\lambda^*(\xi)$ and $\lim_{\xi\to\sqrt{\pi}^-}\lambda^*(\xi)$. It follows from \eqref{limfx} and the dominated convergence theorem that
	\begin{align*}
	\begin{split}
	\lim_{\xi\to0^+}\lambda^*(\xi)&=\lim_{\xi\to0^+}(\xi f^{-1}(\xi)-1)=\lim_{x\to\infty}(xf(x)-1)=-\lim_{x\to\infty}\int_\mathbb{R}\frac{v^2w(v)dv}{x^2+v^2}=0,\\
	\lim_{\xi\to\sqrt{\pi}^-}\lambda^*(\xi)&=\lim_{\xi\to\sqrt{\pi}^-}(\xi f^{-1}(\xi)-1)=\lim_{x\to0^+}(xf(x)-1)=-\lim_{x\to0^+}\int_\mathbb{R}\frac{v^2w(v)dv}{x^2+v^2}=-1.\\
	\end{split}
	\end{align*}
	Moreover, one can show that $\lambda^*(\cdot)$ is strictly decreasing function of $\xi$ for $\xi\in(0,\sqrt{\pi})$. Indeed,
	\begin{align}\label{IFT}
	\begin{split}
	(\lambda^{*})'(\xi)&=-\frac{\omega'_{\xi}(\lambda^{*}(\xi),\xi)}{\omega'_{\lambda}(\lambda^{*}(\xi),\xi)}.
	\end{split}
	\end{align}
	By \eqref{deromega1}, $\omega'_{\lambda}(\lambda^{*}(\xi),\xi)>0$ for $\xi\in(0,\sqrt{\pi})$. And, by applying a corollary of the dominated convergence theorem on differentiation under integral sign, for $\lambda>-1$ we have 
	\begin{align}\lb{deromega11}
	\begin{split}
	\omega'_{\xi}(\lambda,\xi)&=\frac{\partial}{\partial\xi}\big(1-\int_\mathbb{R}\frac{(1+\lambda)w(v)dv}{(1+\lambda)^2+(v\xi)^2}\big)\stackrel{v=(1+\lambda)u}{=}\frac{\partial}{\partial\xi}\big(1-\frac{1}{\sqrt{\pi}}\int_\mathbb{R}\frac{e^{-(1+\lambda)^2u^2}du}{1+(u\xi)^2}\big)\\
	&=-\frac{1}{\sqrt{\pi}}\int_\mathbb{R}\frac{-2\xi u^2e^{-(1+\lambda)^2u^2}du}{(1+(u\xi)^2)^2}=\frac{2\xi}{\sqrt{\pi}}\int_\mathbb{R}\frac{u^2e^{-(1+\lambda)^2u^2}du}{(1+(u\xi)^2)^2}>0\,\,\,\hbox{for}\,\,\xi>0.
	\end{split}
	\end{align}
	Therefore, by \eqref{IFT} and \eqref{deromega11}, $(\lambda^{*})'(\cdot)$ is strictly negative for $\xi\in(0,\sqrt{\pi})$.

	Case 2. Let $\xi\in(-\sqrt{\pi},0)$. As in Case 1, one can analytically show that $\lambda^*(\cdot)$ is an increasing, differentiable function of $\xi$ for $\xi\in(-\sqrt{\pi},0)$,  $-1<\lambda^*(\xi)<0$, $\lim_{\xi\to0^-}\lambda^*(\xi)=0$ and $\lim_{\xi\to-\sqrt{\pi}^+}\lambda^*(\xi)=-1$. \\
	Case 3. Let $\xi=0$. Then, 
	\begin{align}\lb{L0dis}
	\begin{split}
	\omega(\lambda,0)&=1-\int_\mathbb{R}\frac{(1+\lambda)w(v)dv}{(1+\lambda)^2}=1-\frac{1}{1+\lambda}.
	\end{split}
	\end{align}
	Therefore, $\omega(\lambda,0)$ vanishes only at $\lambda=0$ (see Figure \ref{fig:0}). 
	
	Hence, putting all three cases together, we conclude that $\lambda^*(\cdot)$ is a continuous function of $\xi$ for $\xi\in(-\sqrt{\pi},\sqrt{\pi})$,  $-1<\lambda^*(\xi)\leq0$, $\lim_{\xi\to-\sqrt{\pi}^+}\lambda^*(\xi)=-1$ and $\lim_{\xi\to\sqrt{\pi}^-}\lambda^*(\xi)=-1$ (see Figure \ref{fig:lambda}). Note that Figure \ref{fig:lambda} is numerically simulated using the implicit formula \eqref{xipos} for $\lambda^*(\cdot)$.
\begin{figure}
	\begin{subfigure}{.4\textwidth}
		\centering
		\hspace*{-2cm}\includegraphics[width=1.3\linewidth]{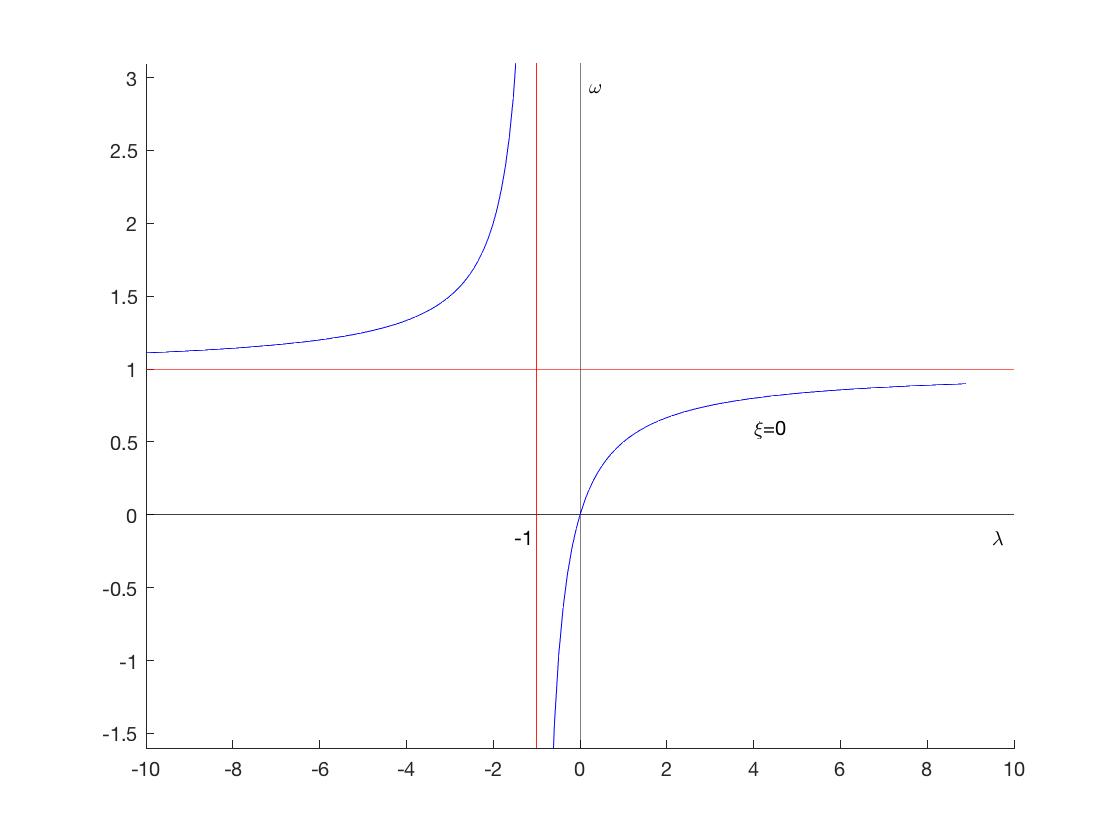}  
		\caption{$\xi=0$}
		\label{fig:0}
	\end{subfigure}
	\begin{subfigure}{.4\textwidth}
		\centering
		\includegraphics[width=1.3\linewidth]{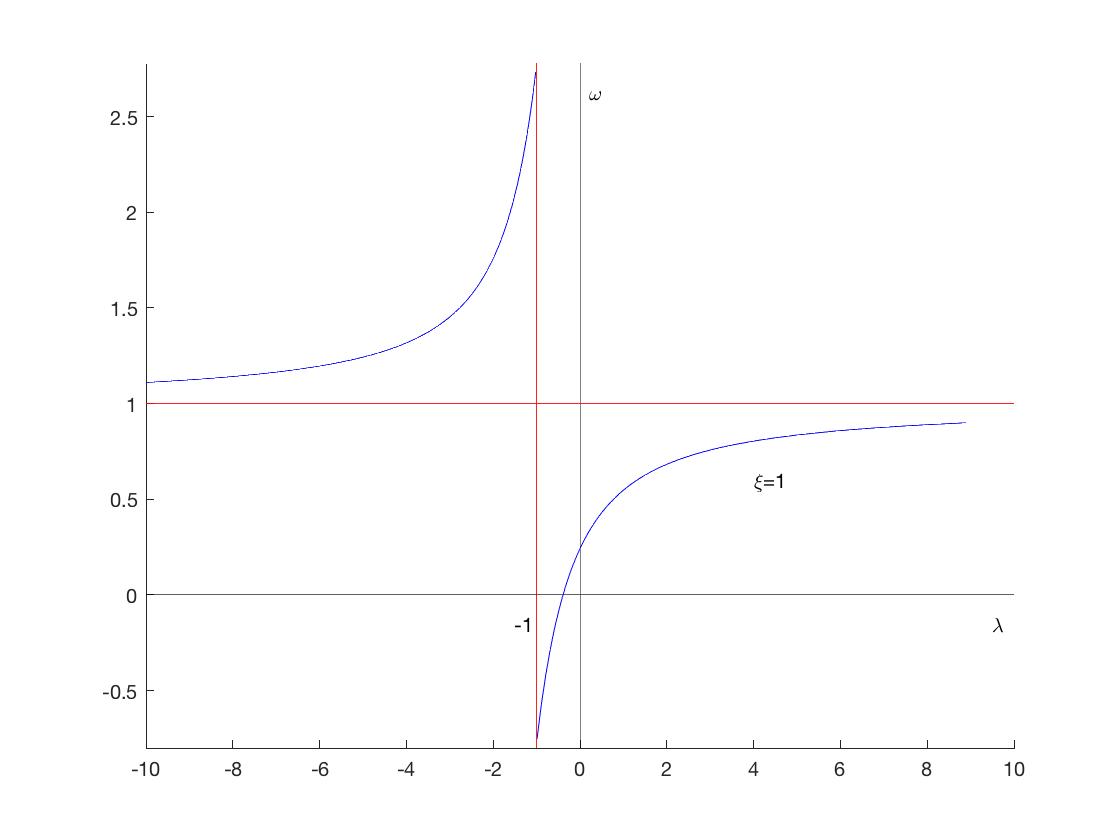}  
		\caption{$\xi=1$}
		\label{fig:1}
	\end{subfigure}
\begin{subfigure}{.4\textwidth}
	\centering
	\hspace*{-2cm}\includegraphics[width=1.3\linewidth]{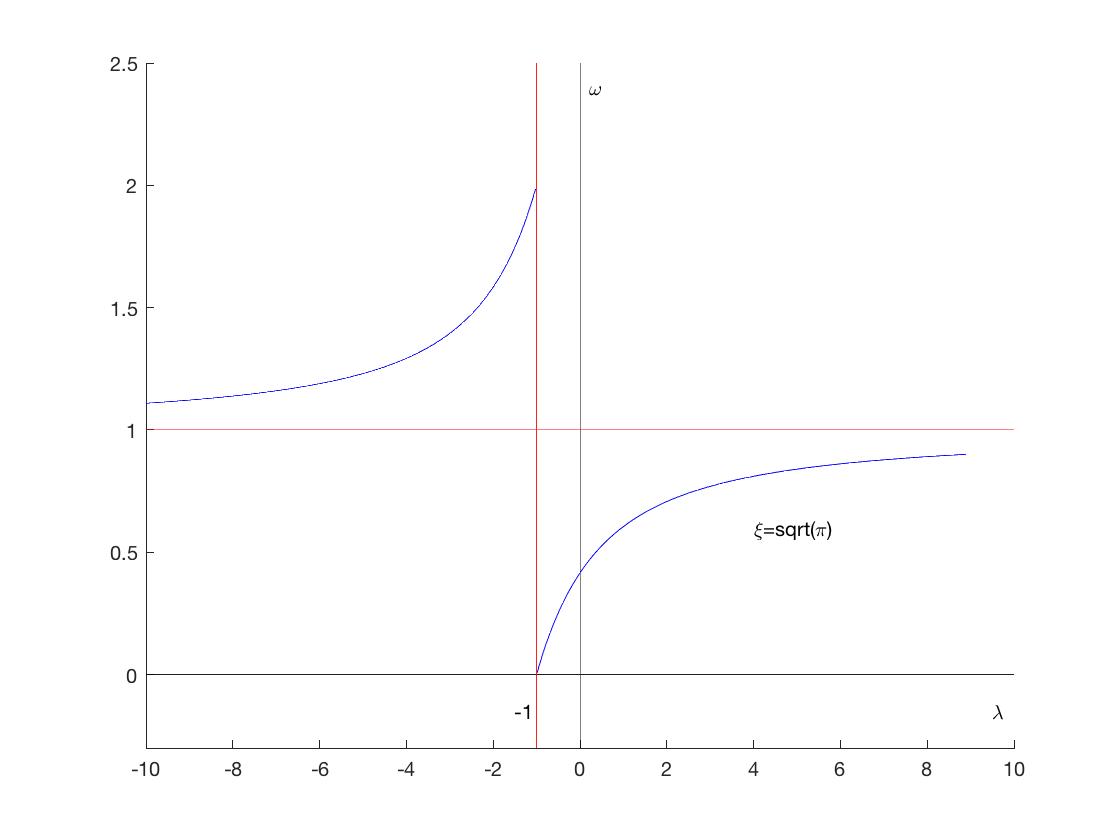}  
	\caption{$\xi=\sqrt{\pi}$}
	\label{fig:sqrtpi}
\end{subfigure}
\begin{subfigure}{.4\textwidth}
	\centering
	\includegraphics[width=1.3\linewidth]{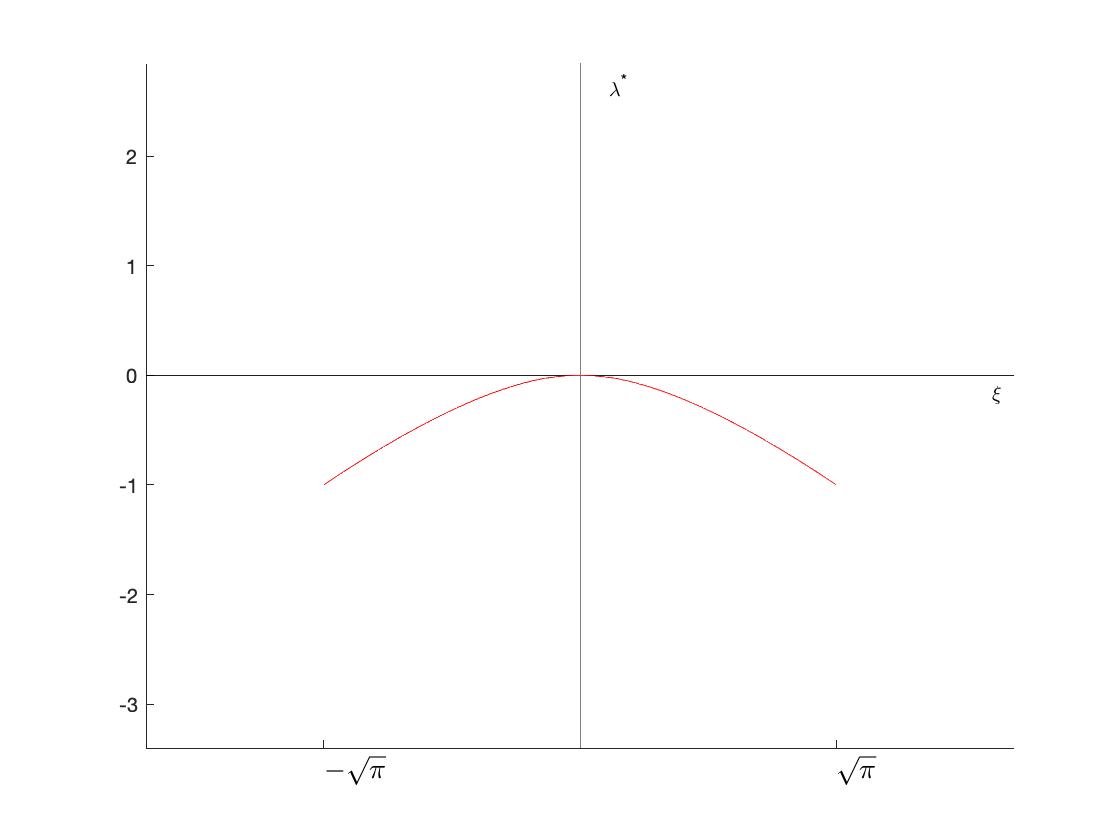}  
	\caption{The $\lambda^*$ curve}
	\label{fig:lambda}
\end{subfigure}
	\caption{The $\lambda^*$ and $\omega$ curves}
\end{figure}
\end{proof}

\begin{theorem}\lb{eigen}
	Let $\xi\in(-\sqrt{\pi},\sqrt{\pi})$. Then the discrete spectrum of $L_{\xi}$ consists of only one isolated eigenvalue $\lambda^*(\xi)$ with the corresponding eigenvector $e_1(\cdot,\xi):=\frac{1}{\cdot i\xi+1+\lambda^*(\xi)}\in L^2_w(\bbR)$, moreover, $\lambda^*(\xi)\in(-1,0]$ and $\lambda^*(\xi)$ is of algebraic multiplicity one. Similarly, $\bar e_1(\cdot,\xi):=\frac{1}{-\cdot i\xi+1+\lambda^*(\xi)}\in L^2_w(\bbR)$ is the eigenvector of $L^*_{\xi}$ corresponding to the isolated eigenvalue $\lambda^*(\xi)$.
\end{theorem}

\begin{proof}
	The description of the discrete spectrum of $L_{\xi}$ is shown in Proposition \ref{disspec}. Now, we would like to find the corresponding eigenvectors. Fix $\xi\in(-\sqrt{\pi},\sqrt{\pi})$. It follows from Proposition \ref{disspec} that there exists a unique $\lambda^*(\xi)\in(-1,0]$ such that $\omega(\lambda^*(\xi),\xi)=0$. Consider the $L^2_w$-function $\frac{1}{vi\xi+1+\lambda^*(\xi)}$. Since $\omega(\lambda^*(\xi),\xi)=0$, we have
	$
	\int_\mathbb{R}\frac{w(v)dv}{vi\xi+1+\lambda^*(\xi)}=1=(vi\xi+1+\lambda^*(\xi))\frac{1}{vi\xi+1+\lambda^*(\xi)}$,
	or,
	
	$(-vi\xi-I)\frac{1}{vi\xi+1+\lambda^*(\xi)}+\int_\mathbb{R}\frac{w(v)dv}{vi\xi+1+\lambda^*(\xi)}=1=\lambda^*(\xi)\frac{1}{vi\xi+1+\lambda^*(\xi)}.$
	Therefore,
	$L_{\xi}e_1=\lambda^*(\xi)e_1$.
\end{proof}

\begin{proposition}\label{continuity16} The maps $R^0(\cdot,\cdot)\mathbb{1}, (R^0(\cdot,\cdot))^*\mathbb{1},:\bbC\setminus\{\lambda: \Re\lambda=-1\}\times\bbR\to{L^2_w}(\bbR)$ are continuous. In particular, the following maps are continuous
	\begin{enumerate}
		\item $(\cdot,(R^0(\cdot,\cdot))^*\mathbb{1})_{L^2_w}:\bbC\setminus\{\lambda: \Re\lambda=-1\}\times\bbR\to({L^2_w}(\bbR))^*$,
		\item $(R^0(\cdot,\cdot)\mathbb{1},\mathbb{1})_{L^2_w}:\bbC\setminus\{\lambda: \Re\lambda=-1\}\times\bbR\to\bbC$,
		\item $\omega(\cdot,\cdot)=\det(1+(R^0(\cdot,\cdot)\mathbb{1},\mathbb{1})_{L^2_w}):\bbC\setminus\{\lambda: \Re\lambda=-1\}\times\bbR\to\bbC$,
		\item $R^0(\lambda^*(\cdot),\cdot)\mathbb{1}:(-\sqrt{\pi},\sqrt{\pi})\to{L^2_w}(\bbR)$,
		\item $(\cdot,(R^0(\lambda^*(\cdot),\cdot))^*\mathbb{1})_{L^2_w}:(-\sqrt{\pi},\sqrt{\pi})\to({L^2_w}(\bbR))^*$,
		\item $\omega'(\lambda^*(\cdot),\cdot):(-\sqrt{\pi},\sqrt{\pi})\to\bbR$.
	\end{enumerate} 
\end{proposition}

\begin{proof}
	Let $\lambda=a+ib$, where $a,b\in\bbR$. Then
	\begin{align*}
	\begin{split}
	\|R^0(\lambda,\xi)\mathbb{1}\|^2_{L^2_w}=\int_\mathbb{R}\frac{w(v)dv}{|vi\xi+1+\lambda|^2}=\int_\mathbb{R}\frac{w(v)dv}{(v\xi+b)^2+(1+a)^2}.
	\end{split}
	\end{align*}
	Notice that \begin{equation}\label{con1}
	\frac{w(v)}{|vi\cdot+1+\cdot|^2}:\bbC\setminus\{\lambda: \Re\lambda=-1\}\times\bbR\to\bbR\,\,\, \hbox{is continuous.}
	\end{equation} 
	Moreover, if $|a+1|\geq\delta>0$, then
	\begin{align}\label{est11}
	\begin{split}
	\frac{w(v)}{|vi\xi+1+\lambda|^2}=\frac{w(v)}{(v\xi+b)^2+(1+a)^2}\leq\frac{w(v)}{\delta^2},\,\,\hbox{and}\,\,\int_\mathbb{R}\frac{w(v)dv}{\delta^2}=\frac{1}{\delta^2}<\infty.
	\end{split}
	\end{align}
	By \eqref{con1}, \eqref{est11} and the dominated convergence theorem, we conclude that $\|R^0(\cdot,\cdot)\mathbb{1}\|^2_{L^2_w}:\bbC\setminus\{\lambda: \Re\lambda=-1\}\to \bbR$ is continuous. Similarly, one can show that  $(R^0(\cdot,\cdot))^*\mathbb{1}:\bbC\setminus\{\lambda: \Re\lambda=-1\}\times\bbR\to{L^2_w}(\bbR)$ is continuous. Items (1)-(3) immediately follow from continuity of $(R^0(\cdot,\cdot))\mathbb{1}$ and $(R^0(\cdot,\cdot))^*\mathbb{1}$. Items (4)-(5) follow from continuity of $(R^0(\cdot,\cdot))\mathbb{1}$ and $(R^0(\cdot,\cdot))^*\mathbb{1}$ and the fact that $\lambda^*(\cdot)$ is a continuous function of $\xi$ for $\xi\in(-\sqrt{\pi},\sqrt{\pi})$ (see Proposition \ref{disspec}).\\
	(6) By formula \eqref{deromega1}, we have 
	\begin{align*}
	\begin{split}
	\omega'(\lambda^*(\xi),\xi)&=\frac{2(\lambda^*(\xi)+1)}{\sqrt{\pi}}\int_\mathbb{R}\frac{u^2e^{-(1+\lambda^*(\xi))^2u^2}du}{1+(u\xi)^2}.
	\end{split}
	\end{align*}
	Notice that 
	\begin{equation}\label{con11}
	\frac{u^2e^{-(1+\lambda^*(\cdot))^2u^2}}{1+(u\cdot)^2}:(-\sqrt{\pi},\sqrt{\pi})\to\bbR\,\,\, \hbox{is continuous.}
	\end{equation} 
	Moreover, if $|1+\lambda^*(\xi)|\geq\delta>0$, then
	\begin{align}\label{est111}
	\begin{split}
	\frac{u^2e^{-(1+\lambda^*(\xi))^2u^2}}{1+(u\xi)^2}\leq u^2e^{-\delta^2u^2},\,\,\hbox{and}\,\,\int_\mathbb{R}u^2e^{-\delta^2u^2}du<\infty.
	\end{split}
	\end{align}
	By \eqref{con11}, \eqref{est111} and the dominated convergence theorem, we conclude that $\omega'(\lambda^*(\cdot),\cdot):(-\sqrt{\pi},\sqrt{\pi})\to\bbR$ is continuous.
\end{proof}

\begin{theorem}\label{unifomega}
\begin{enumerate}
	\item[1.] Fix $\varepsilon_0>0$ and let $\Lambda_{-1+\varepsilon_0}^+= \{\lambda\in\bbC: \, \Re \lambda \geq -1+\varepsilon_0 \}$. For all $\lambda$ in  $\{|\lambda|\geq R,\,\,\,R\,\hbox{sufficietly large}\}\cap \Lambda_{-1+\varepsilon_0}^+$, there exists $\gamma>0$ such that $|\omega(\lambda,\xi)|>\gamma$ uniformly for $\xi\in\bbR$. Moreover, $|\omega(\lambda,\xi)-1|\to0$ as $|\lambda|\to\infty$ within $\Lambda_{-1+\varepsilon_0}^+$, uniformly for $\xi\in\bbR$.
	\item [2.] Fix $\xi_0\in(0,\sqrt{\pi})$ and $\varepsilon>0$. Then for any $\lambda^\dagger\geq\lambda^*(\xi_0)+\varepsilon$ there exists $\gamma>0$ such that $|\omega(\lambda,\xi)|>\gamma$  for $|\xi|\geq\xi_0$ and $\lambda\in\{\lambda\in\bbC:\Re\lambda=\lambda^\dagger\}$. Moreover, $|\omega(\lambda,\xi)-1|\to0$ as $|\xi|\to\infty$ uniformly for $\lambda\in\{\lambda\in\bbC:\Re\lambda=\lambda^\dagger\}$.
	\item [3.] For any $-1<\lambda^\dagger\leq\lambda^*(\xi_0)-\varepsilon$ there exists  $\gamma>0$ such that $|\omega(\lambda,\xi)|>\gamma$  for $|\xi|\leq\xi_0$ and $\lambda\in\{\lambda\in\bbC:\Re\lambda=\lambda^\dagger\}$.
\end{enumerate}
\end{theorem}
\begin{proof}
	1. Let $\lambda=a+ib\in\Lambda_{-1+\varepsilon_0}^+$, where $a,b\in\bbR$. We consider the following cases:\\
	a) Let $a>>1$. Then
	$
	|\omega(\lambda,\xi)-1|=\big|\int_\mathbb{R}\frac{w(v)dv}{vi\xi+1+a+ib}\big|\leq\int_\mathbb{R}\frac{w(v)dv}{|1+a|}=\frac{1}{|1+a|}.
	$
	Therefore, $|\omega(\lambda,\xi)-1|\to0$ as $a\to\infty$, uniformly for $\xi\in\bbR$.\\
	b) Let $-1+\varepsilon_0\leq a\leq m$ and $|b|\geq\delta_0$, where $m, \delta_0$ are fixed and $m,\delta_0>0$. We break the integral into three parts. That is,
	\begin{align*}
	\begin{split}
	|\omega(\lambda,\xi)-1|&=\big|\int_\mathbb{R}\frac{w(v)dv}{vi\xi+1+a+ib}\big|\leq\int\limits_{|v\xi|\leq(1-\theta)|b|}\frac{w(v)dv}{|vi\xi+1+a+ib|}\\
	&+\int\limits_{(1-\theta)|b|\leq|v\xi|\leq(1+\theta)|b|}\frac{w(v)dv}{|vi\xi+1+a+ib|}+\int\limits_{|v\xi|\geq(1+\theta)|b|}\frac{w(v)dv}{|vi\xi+1+a+ib|}=I+II+III,
	\end{split}
	\end{align*}
	where $0<\theta<\delta$, $\delta$ is fixed and $\delta<1$.
	Next, we estimate each integral separately.  
	\begin{align*}
	\begin{split}
	 I:=\int\limits_{|v\xi|\leq(1-\theta)|b|}\frac{w(v)dv}{|(v\xi+b)i+1+a|}\leq\int\limits_{|v\xi|\leq(1-\theta)|b|}\frac{w(v)dv}{\theta|b|}\leq\frac{1}{\theta|b|}.
	\end{split}
	\end{align*}
	Similarly,
	\begin{align*}
	\begin{split}
	III:=\int\limits_{|v\xi|\geq(1+\theta)|b|}\frac{w(v)dv}{|(v\xi+b)i+1+a|}\leq\int\limits_{|v\xi|\geq(1+\theta)|b|}\frac{w(v)dv}{\theta|b|}\leq\frac{1}{\theta|b|}.
	\end{split}
	\end{align*}
	Now, we estimate the second integral. Since $b\neq0$, $\xi\neq0$ for the second integral. Therefore, $(1-\theta)|\frac{b}{\xi}|\leq|v|\leq(1+\theta)|\frac{b}{\xi}|$. Then
$
	w(v)=\frac{1}{\sqrt{\pi}}e^{-v^2}\leq\frac{1}{\sqrt{\pi}}e^{-(1-\theta)^2(\frac{b}{\xi})^2}.
	$
	Therefore,
	\begin{align*}
	\begin{split}
	II:&=\int\limits_{(1-\theta)|b|\leq|v\xi|\leq(1+\theta)|b|}\frac{w(v)dv}{|(v\xi+b)i+1+a|}\leq\int\limits_{(1-\theta)|b|\leq|v\xi|\leq(1+\theta)|b|}\frac{\frac{1}{\sqrt{\pi}}e^{-(1-\theta)^2(\frac{b}{\xi})^2}dv}{|a+1|}\\
	&\leq\frac{\frac{2}{\sqrt{\pi}}\theta |\frac{b}{\xi}|e^{-(1-\theta)^2(\frac{b}{\xi})^2}}{|a+1|}\leq C\frac{\theta}{|a+1|},
	\end{split}
	\end{align*}
	where the last inequality is implied by $|x|e^{-(1-\theta)^2x^2}\leq \tilde C$ uniformly for all $0<\theta<\delta$,  $\delta<1$. Therefore,  by choosing $\theta=\frac{1}{\sqrt{|b|+1}}$ and by the fact that $a$ belongs to the compact interval, we have 
	$|\omega(\lambda,\xi)-1|\leq\frac{C}{\sqrt{|\Im\lambda|+1}||\Re\lambda+1|}$
	for all $\lambda$ such that $-1+\varepsilon_0\leq \Re\lambda\leq m, |\Im\lambda|\geq\delta_0, \delta_0>0$, uniformly for $\xi\in\bbR.$
	In particular, $|\omega(\lambda,\xi)-1|\to0$ as $|\Im \lambda|\to\infty$ ($\Re\lambda$ is bounded), uniformly for $\xi\in\bbR.$\\
	Combining the results from cases a) and b), we arrive the first statement of Theorem \ref{unifomega}.\\
	2. Fix any $\lambda^\dagger\geq\lambda^*(\xi_0)+\varepsilon$.  For $|\xi|\geq\xi_0$ and $\lambda=a+ib$, where $a=\lambda^\dagger$ we have \begin{align*}
	\begin{split}
	|\omega(\lambda,\xi)-1|&=\big|\int_\mathbb{R}\frac{w(v)dv}{vi\xi+1+a+ib}\big|\leq\int\limits_{|b|\leq(1-\theta)|v\xi|}\frac{w(v)dv}{|vi\xi+1+a+ib|}\\
	&+\int\limits_{(1-\theta)|v\xi|\leq|b|\leq(1+\theta)|v\xi|}\frac{w(v)dv}{|vi\xi+1+a+ib|}+\int\limits_{|b|\geq(1+\theta)|v\xi|}\frac{w(v)dv}{|vi\xi+1+a+ib|}=I+II+III,
	\end{split}
	\end{align*}
	where $0<\theta<\delta$, $\delta$ is fixed and $\delta<1$.
	Next, we estimate each integral separately.  
	\begin{align*}
	\begin{split}
	I:=\int\limits_{|b|\leq(1-\theta)|v\xi|}\frac{w(v)dv}{|(v\xi+b)i+1+a|}\leq\int_\bbR\frac{w(v)dv}{\theta^2(v\xi)^2+(a+1)^2}.
	\end{split}
	\end{align*}
	Similarly,
	\begin{align*}
	\begin{split}
	III:=\int\limits_{|b|\geq(1+\theta)|v\xi|}\frac{w(v)dv}{|(v\xi+b)i+1+a|}\leq\int_\bbR\frac{w(v)dv}{\theta^2(v\xi)^2+(a+1)^2}.
	\end{split}
	\end{align*}
	Now, we estimate the second integral. Since $\xi\neq0$ for the second integral. Therefore, $\frac{1}{1+\theta}|\frac{b}{\xi}|\leq|v|\leq\frac{1}{1-\theta}|\frac{b}{\xi}|$. Then
	$
	w(v)=\frac{1}{\sqrt{\pi}}e^{-v^2}\leq\frac{1}{\sqrt{\pi}}e^{-(\frac{1}{1+\theta})^2(\frac{b}{\xi})^2}.
	$
	Therefore,
	\begin{align*}
	\begin{split}
	II:&=\int\limits_{(1-\theta)|v\xi|\leq|b|\leq(1+\theta)|v\xi|}\frac{w(v)dv}{|(v\xi+b)i+1+a|}\leq\int\limits_{\frac{1}{1+\theta}|\frac{b}{\xi}|\leq|v|\leq\frac{1}{1-\theta}|\frac{b}{\xi}|}\frac{\frac{1}{\sqrt{\pi}}e^{-(\frac{1}{1+\theta})^2(\frac{b}{\xi})^2}dv}{|a+1|}\\
	&\leq\frac{\frac{2}{\sqrt{\pi}}\frac{\theta}{1-\theta^2} |\frac{b}{\xi}|e^{-(\frac{1}{1+\theta})^2(\frac{b}{\xi})^2}}{|a+1|}\leq C\frac{\frac{\theta}{1-\theta^2}}{|a+1|},
	\end{split}
	\end{align*}
		where the last inequality is implied by $|x|e^{-(\frac{1}{1+\theta})^2x^2}\leq \tilde C$ uniformly for all $\theta$ such that $0<\theta<\delta$,  $\delta<1$. Therefore,  by choosing $\theta=\frac{1}{\sqrt{|\xi|+1}}$, we see that $I+II+II\to0$ and $|\omega(\lambda,\xi)-1|\to0$ as $|\xi|\to\infty$ uniformly for $\lambda\in\{\lambda\in\bbC:\Re\lambda=\lambda^\dagger\}$. In particular, for $\gamma_0>0$ there exists a $R_0>0$ such that $|\omega(\lambda,\xi)|>\gamma_0$  for $|\xi|\geq R_0$ ($R_0\geq\xi_0$) and $\lambda\in\{\lambda\in\bbC:\Re\lambda=\lambda^\dagger\}$.
		
		We will consider the case $\xi_0\leq|\xi|\leq R_0$. Choose $R_1>0$ large enough, then, by item 1, there exists a $\gamma_1>0$ such that $|\omega(\lambda,\xi)|>\gamma_1$  for $R_0\geq|\xi|\geq\xi_0$ and $\lambda\in\{\lambda\in\bbC:\Re\lambda=\lambda^\dagger,\,|\Im\lambda|\geq R_1\}$. Also, since $\omega(\cdot,\cdot)$ is a continuous function (see item (3), Theorem \ref{continuity16}) and it vanishes in the compact region $\{(\lambda,\xi): \Re\lambda\in[\lambda^*(\xi_0),\lambda^\dagger],0\leq\Im\lambda\leq R_1, R_0\geq|\xi|\geq\xi_0\}$ only at the point $(\lambda^*(\xi_0),\xi_0)$, there exists a $\gamma_2>0$ such that $|\omega(\lambda,\xi)|>\gamma_2$  for $R_0\geq|\xi|\geq\xi_0$ and $\lambda\in\{\lambda\in\bbC:\Re\lambda=\lambda^\dagger,\,|\Im\lambda|\leq R_1\}$.
		By choosing $\gamma=\min\{\gamma_i\}$, $i=1,2,3$, we arrive at the statement of item 2.\\
		3. The proof of item 3 is similar to that of item 2.
\end{proof}

\begin{proposition}\lb{L0prop}
	Let $\xi=0$. Then
	\begin{itemize}
		\item $\sigma_{ess}(L_0)=\{-1\}$, and $\sigma_d(L_0)=\{0\}$,
		\item $K(\lambda,0)=I-\frac{V}{1+\lambda},\,\,\lambda\notin\sigma_{ess}(L_0)$; $K^{-1}(\lambda,0)=I+\frac{V}{\lambda},\,\,\lambda\notin\sigma_{d}(L_0)$.
		\item $R(\lambda,0)=\frac{-1}{1+\lambda}(I+\frac{V}{\lambda})=\frac{1}{1+\lambda}(V-I)-\frac{1}{\lambda}V,\,\,\lambda\notin\sigma(L_0)$.
	\end{itemize}
\end{proposition}
\begin{proof}
	The fact that $\sigma_{ess}(L_0)=\{-1\}$, and $\sigma_d(L_0)=\{0\}$ follows from \eqref{ess} and \eqref{L0dis}.

		Let $f,g \in H={L_w^2(\bbR)}$ and assume that $K(\lambda,0)f=g$. Then
		\begin{equation*}
		g=(I+VR^0(\lambda,0)V)f=f+(f,\mathbb{1})_{L_w^2}( R^0(\lambda,0)\mathbb{1},\mathbb{1})_{L_w^2}\mathbb{1}.
		\end{equation*}
		Therefore,
		\begin{equation*}
		(g,\mathbb{1})_{L_w^2}=(f,\mathbb{1})_{L_w^2}(1-\frac{1}{1+\lambda}).
		\end{equation*}
		Hence,
		\begin{equation*}
		f=g+\frac{(f,\mathbb{1})_{L_w^2}\mathbb{1}}{1+\lambda}=g+\frac{(g,\mathbb{1})_{L_w^2}\mathbb{1}}{\lambda}.
		\end{equation*}
	The last item follows from the second item and formula \eqref{perR}.
\end{proof}

\section{The resolvent extensions of $M_\xi$ and $L_{\xi}$ through the essential spectrum}
We would like to extend the resolvent of $M_\xi$, $L_{\xi}$ and the Birman-Schwinger determinant $\omega(\cdot,\xi)$ through the essential spectrum. First, we introduce the following notation:
\begin{align*}
\begin{split}
\Pi_+&=\{z\in\bbC:\Re z>-1 \},\,\,\Pi_-=\{z\in\bbC:\Re z<-1 \},\\\,\,\Omega^T&=\{z\in\bbC:|\Re z+1|<\e\},\,\,\Omega^T_\xi=\{z\in\bbC:|\Re z+1|<\e|\xi|\},\,\,\xi\neq0,\,\,\e>0. 
\end{split}
\end{align*}

\begin{proposition}\label{extevery}
	Fix $\xi\neq0$. 
	\begin{enumerate}
	\item {(The unperturbed case)}. For arbitrary $\phi,\psi\in\Phi_0$ the function $\lambda\to (R^0(\lambda,\xi)\phi,\psi)_{L^2_w}$, where $R^0(\lambda,\xi)=(M_\xi-\lambda)^{-1}$,
	has holomorphic extensions $\lambda\to (R^0(\lambda,\xi)\phi,\psi)_\pm$ from $\Pi_{\pm}$ to $\Pi_{\pm}\cup\Omega^T_\xi$. Moreover,
	\begin{equation}\label{jump0}
	(R^0(\lambda,\xi)\phi,\psi)_+-(R^0(\lambda,\xi)\phi,\psi)_-=-\frac{2\pi} {|\xi|}\phi\Big(\frac{i(\lambda+1)}{\xi}\Big)\overline{\psi\Big(\overline{\frac{i(\lambda+1)}{\xi}}\Big)}w\Big(\frac{i(\lambda+1)}{\xi}\Big),\,\,\,\lambda\in\Omega^T_\xi.
	\end{equation} Letting
	\begin{equation*}
	\langle R^0_\pm(\lambda, \xi)\phi,\psi\rangle:=(R^0(\lambda,\xi)\phi,\psi)_\pm,\,\,\,\phi,\psi\in\Phi,
	\end{equation*}
	we see that $R^0_\pm(\lambda, \xi)\phi\in\Phi^*$ for all $\phi\in\Phi$, that is, $R^0_\pm(\lambda, \xi):\Phi\to\Phi^*$. Moreover,
		\begin{equation}\label{smallo}
		\langle R^0_\pm(\lambda, \xi)\phi,\psi\rangle=o(1),\,\,\,|\lambda|\to\infty
	\end{equation}
	uniformly in the region $\{\lambda\in\bbC:\Re\lambda+1\geq-\eta|\xi| \}$ in the case of $R_+(\lambda, \xi)$ and uniformly in the region $\{\lambda\in\bbC:\Re\lambda+1\leq\eta|\xi| \}$ in the case of $R_-(\lambda, \xi)$ for any $\eta<\varepsilon$.
	\begin{equation*}
	(R^0_\pm(\lambda, \xi))^*=R^0_\pm(\bar\lambda, -\xi),
	\end{equation*}
	where the adjoint is understood in the sense of \eqref{defadj}.
	\item The Birman-Schwinger-type function $\lambda\to K(\lambda, \xi):=I+VR^0(\lambda, \xi)V\in\mathcal{B}(L^2_w(\bbR))$ and its determinant $\lambda\to \omega(\lambda,\xi)$ have holomorphic extensions $\lambda\to K_\pm(\lambda, \xi)\in\mathcal{B}(L^2_w(\bbR))$ and $\lambda\to \omega_\pm(\lambda,\xi)$ from $\Pi_{\pm}$ to $\Pi_{\pm}\cup\Omega^T_\xi$. In particular,
	\begin{equation}\label{Kext}
K_\pm(\lambda, \xi)=I+VR^0_\pm(\lambda,\xi)V
	\end{equation}  and are operators of identity plus rank one.
	Moreover, for any $\xi\in[-\sqrt{\pi},0)\cup(0,\sqrt{\pi}]$ there exists a unique  $\lambda\in\bar\Pi_+$ (more precisely, $\lambda\in[-1,0)$) such that $0\in\sigma_d(K_+(\lambda, \xi))$, and if $\xi\in\mathbb{R}\setminus[-\sqrt{\pi},\sqrt{\pi}]$, then $0\in\rho(K_+(\lambda, \xi))$ for any $\lambda\in\bar\Pi_+$, and  $0\in\rho(K_-(\lambda, \xi))$ for any $\xi\in\bbR$ and $\lambda\in\bar\Pi_-$. 
	Also,
		\begin{equation}\label{Ksmallo}
	\|K_\pm(\lambda, \xi)-I\|_{L^2_w}=o(1),\,\,\,|\lambda|\to\infty
	\end{equation}
	uniformly in the region $\{\lambda\in\bbC:\Re\lambda+1\geq-\eta|\xi| \}$ in the case of  $K_+(\lambda, \xi)$ and uniformly in the region $\{\lambda\in\bbC:\Re\lambda+1\leq\eta|\xi| \}$ in the case of $K_-(\lambda, \xi)$ for any $\eta<\varepsilon$. Moreover, for any $\xi\neq0$ there exists a small enough $\gamma$ such $0<\gamma<\varepsilon$ and $0\in\rho(K_+(\lambda, \xi))$ for any $\lambda\in\{\lambda\in\bbC:0>\Re\lambda+1>-\gamma|\xi| \}$. 
	
	\item {(The perturbed case)}. For arbitrary $\phi,\psi\in\Phi$ the function $\lambda\to (R(\lambda,\xi)\phi,\psi)_{L^2_w}$, where $R(\lambda,\xi)=(L_{\xi}-\lambda)^{-1}$,
	has meromorphic extensions $\lambda\to (R(\lambda,\xi)\phi,\psi)_\pm$ from $\Pi_{\pm}$ to $\Pi_{\pm}\cup\Omega^T_\xi$. Letting
	\begin{equation*}
	\langle R_\pm(\lambda, \xi)\phi,\psi\rangle:=(R(\lambda,\xi)\phi,\psi)_\pm,\,\,\,\phi,\psi\in\Phi,
	\end{equation*}
	we see that $R_\pm(\lambda, \xi)\phi\in\Phi^*$ for all $\phi\in\Phi$, that is, $R_\pm(\lambda, \xi):\Phi\to\Phi^*$. In particular,
		\begin{equation}\label{purtR}
	R_\pm(\lambda, \xi)=R^0_\pm(\lambda,\xi)-R^0_\pm(\lambda,\xi)VK^{-1}_\pm(\lambda, \xi)VR^0_\pm(\lambda,\xi).
	\end{equation}
	Moreover, the poles of $(R(\cdot,\xi)\phi,\psi)_\pm$ coincide with the poles of $K^{-1}_\pm(\cdot, \xi)$. Also, \begin{equation}\label{Radj}
	(R_\pm(\lambda, \xi))^*=R_\pm(\bar\lambda, -\xi),
	\end{equation}
	where the adjoint is understood in the sense of \eqref{defadj}. 
	Finally, 	\begin{equation}\label{Rsmallo}
	\langle R_\pm(\lambda, \xi)\phi,\psi\rangle=o(1),\,\,\,|\lambda|\to\infty
	\end{equation}
	uniformly in the region $\{\lambda\in\bbC:\Re\lambda+1\geq-\eta|\xi| \}$ in the case of $R_+(\lambda, \xi)$ and uniformly in the region $\{\lambda\in\bbC:\Re\lambda+1\leq\eta|\xi| \}$ in the case of $R_-(\lambda, \xi)$ for any $\eta<\varepsilon$.
\end{enumerate}

\end{proposition}
\begin{proof}
   (1) Fix $\xi\neq0$. Let 
	\begin{equation}\lb{Cauchy}
	(R^0(\lambda,\xi)\phi,\psi)_\pm:=i\int_{\Gamma_\mp(\xi)}\frac{\phi(z_{sh})\overline{\psi(\overline{z_{sh}})}}{-i\xi z_{sh}-1-\lambda}w(z_{sh})dz,\,\,\,z_{sh}=(z+1)i,
	\end{equation}
	where the contour $\Gamma_-(\xi)(\Gamma_+(\xi))$ passes along the vertical line $z=-1$ except for a neighborhood of the point  $z=-1+(\frac{\Im\lambda}{\xi})i$, and passes around $z=-1+\frac{\lambda+1}{\xi}$ to the left(to the right) for $\xi>0$ and to the right(to the left) for $\xi<0$ without leaving the domain $\Omega^T$. It is clear that \eqref{Cauchy} defines the required analytic extension. 
	Also, note that by \eqref{Cauchy}
		\begin{equation}\lb{Rintegral}
	\langle R^0_\pm(\lambda, \xi)\phi,\psi\rangle=\int_{\gamma_\mp(\xi)}\frac{\phi(z)\overline{\psi(\overline{z})}}{-i\xi z-1-\lambda}w(z)dz,\,\,\,
	\end{equation}
	where the contour $\gamma_-(\xi)=\gamma_-$ for $\xi>0$ and $\gamma_-(\xi)=\gamma_+$ for $\xi<0$ where the contour $\gamma_-(\gamma_+)$ passes along the $x$-axis except for a neighborhood of the point  $x=-\frac{\Im\lambda}{\xi}$, and passes around $z=\frac{\lambda+1}{\xi}i$ from below(from above) without leaving the domain $\Omega$. Similarly, we define $\gamma_+(\xi)=\gamma_+$ for $\xi>0$ and $\gamma_-(\xi)=\gamma_-$ for $\xi<0$.\\
	Moreover,
		\begin{align}\label{R0adj}
		\begin{split}
			\langle R^0_\pm(\lambda, \xi)\phi,\psi\rangle&=\int_{\gamma_\mp(\xi)}\frac{\phi(z)\overline{\psi(\overline{z})}}{-i\xi z-1-\lambda}w(z)dz=\overline{\int_{\gamma_\mp(\xi)}\frac{\overline{\phi(z)}{\psi(\overline{z})}}{i\xi \bar z-1-\bar\lambda}w(z)d\bar z}\\
			&=\overline{\int_{\gamma_\pm(\xi)}\frac{\overline{\phi(\bar z)}{\psi({z})}}{i\xi z-1-\bar\lambda}w(z)d z}=\overline{\int_{\gamma_\mp(-\xi)}\frac{\overline{\phi(\bar z)}{\psi({z})}}{i\xi z-1-\bar\lambda}w(z)d z}\\
			&=\overline{\langle R^0_\pm(\bar\lambda, -\xi)\psi, \phi\rangle}=\overline{\langle (R^0_\pm(\lambda, \xi))^*\psi, \phi\rangle}.
		\end{split}
	\end{align}
	Now, let $\xi>0$.  For all $\lambda$ from the region $\{\lambda\in\bbC:\Re\lambda+1\geq-\eta\xi \}$ and $\Im z=-\gamma$, where $\eta<\gamma<\varepsilon$ we have 
	\begin{align*}
	|-i\xi z-1-\lambda|^{-1}&\leq|\gamma\xi+1+\Re\lambda|^{-1}\leq((\gamma-\eta)\xi)^{-1}.
	\end{align*}
	Moreover, for all $x$ such that $|x|\leq\frac{|\gamma\xi+1+\lambda|}{2|\xi|}$, where $x=\Re z$
	\begin{align*}
	|-i\xi z-1-\lambda|^{-1}&=|-i\xi x-\gamma\xi-1-\lambda|^{-1}\leq||\xi x|-|\gamma\xi+1+\lambda||^{-1}\leq\frac{2}{|\gamma\xi+1+\lambda|}.
	\end{align*}
	Therefore,
		\begin{align*}
	\begin{split}
	|\langle R^0_+(\lambda, \xi)\phi,\psi\rangle|&=|\int_{\gamma_-}\frac{\phi(z)\overline{\psi(\overline{z})}}{-i\xi z-1-\lambda}w(z)dz|=|\int_{-\infty-\gamma i}^{\infty-\gamma i}\frac{\phi(z)\overline{\psi(\overline{z})}}{-i\xi z-1-\lambda}w(z)dz|\\
	&\leq((\gamma-\eta)\xi)^{-1}\int_{\{x-\gamma i:|x|>\frac{|\gamma\xi+1+\lambda|}{2|\xi|}\}}|{\phi(z)\overline{\psi(\overline{z})}}w(z)|dz\\
	&+\frac{2}{|\gamma\xi+1+\lambda|}\int_{\{x-\gamma i:|x|\leq\frac{|\gamma\xi+1+\lambda|}{2|\xi|}\}}|{\phi(z)\overline{\psi(\overline{z})}}w(z)|dz,
	\end{split}
	\end{align*}
	giving the $o(1)$ result in item (1). All the other cases for the $o(1)$ result can be treated similarly.\\
	(2) The formula
	\begin{equation}\lb{Kext}
	(K_\pm(\lambda, \xi) u,v)_{L^2_w}:=(u,v)_{L^2_w}+\langle R_\pm^0(\lambda,\xi)Vu,Vv\rangle,\,\,\,u,v\in L^2_w(\bbR)
	\end{equation}
	defines the corresponding analytic continuations. Notice that $Vu, Vv\in\Phi$.\\
	Next, we introduce the extension of $V$, that is, for any $\phi^*\in\Phi^*$ and $u\in L^2_w(\bbR)$ we define $V: \Phi^*\to L^2_w(\bbR)$ as follows
	\begin{equation*}
	(V\phi^*,u)_{L^2_w}:=\langle\phi^*,Vu\rangle.
	\end{equation*}
	Therefore, 
	\begin{equation}\lb{V}
	V\phi^*=\langle\phi^*,\mathbb{1}\rangle\mathbb{1}.
	\end{equation}
	Also, notice that the operator $V\in\mathcal{B}(L^2_w(\bbR),\Phi)$ and $V\in\mathcal{B}(\Phi^*,L^2_w(\bbR))$. Moreover, the adjoint operator to $V\in\mathcal{B}(L^2_w(\bbR),\Phi)$ is equal to $V\in\mathcal{B}(\Phi^*,L^2_w(\bbR))$. Indeed, for $\phi\in\Phi^*$ and $f\in L^2_w(\bbR)$ we have 
	\begin{equation}\label{V*}
	\langle\phi^*,Vf\rangle=\overline{(f,\mathbb{1})}_{L^2_w}\langle\phi^*,\mathbb{1}\rangle=(\mathbb{1},f)_{L^2_w}\langle\phi^*,\mathbb{1}\rangle=(V\phi^*,f)_{L^2_w}=(V^*\phi^*,f)_{L^2_w}.
	\end{equation}
	Hence, it follows from \eqref{Kext} and \eqref{V*} that 
	\begin{equation*}
	K_\pm(\lambda, \xi)=I+VR^0_\pm(\lambda,\xi)V
	\end{equation*} and is an operator of identity plus rank one. Also, notice that
	\begin{equation}\label{Kadj}
	(K_\pm(\lambda, \xi))^*=I+VR^0_\pm(\bar\lambda,-\xi)V.
	\end{equation}

	Moreover, for any fixed $\xi\neq0$ and any $\lambda\in\Pi_+\cup\Omega^T_\xi$, we extend $\omega(\cdot,\xi)$ by defining $\omega_\pm(\cdot,\xi)$ to be the determinants of $K_\pm(\cdot, \xi)$:
	\begin{equation}\lb{detpm}
\omega_\pm(\lambda,\xi):=\det(K_\pm(\lambda,\xi))=\det(1+\langle R_\pm^0(\lambda,\xi)\mathbb{1},\mathbb{1}\rangle)=1-\int_\mathbb{\gamma_\mp(\xi)}\frac{w(z)dz}{zi\xi+1+\lambda}.
	\end{equation}
	Then the claim about invertibility/noninvertibility of $K_+(\lambda, \xi)$ ($K_-(\lambda, \xi)$) for $\xi\in\mathbb{R}$ and $\lambda\in\Pi_+$ ($\lambda\in\Pi_-$) follows from Proposition \ref{disspec}.\\
	Now, we assume that $\lambda\in\bar\Pi_+\setminus\Pi_+$. Let $\lambda=-1+ib$, then we consider two cases: $\xi>0$ and $\xi<0$.\\
	Case 1. Let $\xi>0$. Then
	\begin{align}\lb{Iman}
	\begin{split}
	\int_\mathbb{\gamma_\mp(\xi)}\frac{w(z)dz}{zi\xi+1+\lambda}&=\frac{1}{i\xi}\int_\mathbb{\gamma_\mp}\frac{w(z)dz}{z-\frac{1+\lambda}{\xi}i}=\frac{1}{i\xi}\lim_{\eta\to0^\pm}\int_\mathbb{\mathbb{R}}\frac{w(z)dz}{z-(-b/\xi+\eta i)}\\
	&=\frac{1}{i\xi}\big(\pm\pi iw(-b/\xi)+P.v.\int_\bbR\frac{w(v)dv}{v+b/\xi}\big).
	\end{split}
	\end{align}
	If $b\neq0$, then the imaginary part of the integral from \eqref{Iman} doesn't vanish. And if $b=0$, then
	\begin{equation*}
	\int_\mathbb{\gamma_\mp}\frac{w(z)dz}{zi\xi+1+\lambda}=\pm\sqrt{\pi}/\xi.
	\end{equation*}
	Therefore, if $\lambda\in\bar\Pi_+\setminus\Pi_+$ and $\xi>0$, then $\det(K_-(\lambda,\xi))$ doesn't vanish and $\det(K_+(\lambda,\xi))$ vanishes if and only if $\lambda=-1$ and $\xi=\sqrt{\pi}$.\\
	Case 2. Let $\xi<0$.
	Similarly, one can show that if $b\neq0$, then the imaginary part of the integral $\det(K_\pm(\lambda,\xi))$ doesn't vanish and if $b=0$, then
	\begin{equation}\lb{Iman1}
\det(K_\pm(\lambda,\xi))=1-(\mp\sqrt{\pi}/\xi).
	\end{equation}
	Therefore, if $\lambda\in\bar\Pi_+\setminus\Pi_+$ and $\xi<0$, then $\det(K_-(\lambda,\xi))$ doesn't vanish and $\det(K_+(\lambda,\xi))$ vanishes if and only if $\lambda=-1$ and $\xi=-\sqrt{\pi}$.\\
	Formula \eqref{Ksmallo} follows directly from \eqref{smallo} and \eqref{Kext}.\\
	The existence of $\gamma $ such that $0<\gamma<\varepsilon$ and $K_+(\lambda,\xi)$ is invertible in the region  $\{\lambda\in\bbC:0>\Re\lambda+1>-\gamma|\xi| \}$ follows from formula \eqref{Ksmallo} and the fact that the determinant of $K_+(\lambda,\xi)$ is a holomorphic function in the region $\{\lambda\in\bbC:0>\Re\lambda+1>-\varepsilon|\xi| \}$. \\
	(3) 
	Using formula \eqref{perR}, we can extend $R(\lambda,\xi)$ as follows
	\begin{equation}\label{Rextpm}
	(R(\lambda,\xi)\phi,\psi)_\pm=\langle R^0_\pm(\lambda,\xi)\phi,\psi\rangle-(K^{-1}_\pm(\lambda, \xi)VR^0_\pm(\lambda,\xi)\phi,VR^0_\pm(\bar\lambda,-\xi)\psi)_{L^2_w}.
	\end{equation}
	Also, formula \eqref{Radj} follows from \eqref{R0adj}, \eqref{Kadj}  and \eqref{Rextpm}.\\
	Finally, formula \eqref{Rsmallo} follows from \eqref{smallo} and \eqref{Ksmallo}.
\end{proof}
\begin{remark}
	Notice that the region $\Omega^T_\xi$ becomes an empty set for $\xi=0$. That is why we assume that $\xi\neq0$ when we work with the rigged spaces.  And, we will treat the case $\xi=0$ separately.
\end{remark}
\begin{remark}\label{dualrmk}
	Here, following the rigged space formalism of \cite{L70}, we have extended the resolvent through the essential
	spectrum of $L_\xi$ by restriction to analytic test functions.
	This is in some sense dual to the strategy followed in scattering theory \cite{LP89} and stability
	of traveling waves \cite{S76,GZ98,ZH98} of extending the resolvent by restriction to
	spatially-exponentially decaying test functions.
	However, different from the analyses of \cite{S76,GZ98,ZH98}, we (and Ljance \cite{L70})
	do not make use of the extension of the resolvent past the essential spectrum to
	obtain estimates, but only the continuous extensions up to the boundary from either side;
	see Section \ref{s:disc}, \eqref{idres}-\eqref{modilt}.
\end{remark}

\section{The generalized eigenfunctions}
Fix $\xi\neq0$ and let 
\begin{align}\lb{delta0}
\delta^0_\lambda(z):=\frac{1}{2\pi i}\frac{1}{ z-\frac{i(\lambda+1)}{\xi}}w^{-1/2}(z),\,\,\lambda\in\Omega^T_\xi.
\end{align}
Note that $\delta^0_\lambda(\cdot)\in\Phi^*_{-\eta}$ for $\eta>\frac{|\Re\lambda+1|}{|\xi|}$. The functional as an element of the dual space $\Phi^*$ corresponding to $\delta^0_\lambda(\cdot)$ is called a generalized functional and
\begin{equation*}
\langle\delta^0_\lambda,\phi\rangle=\overline{\phi\Big(\overline{\frac{i(\lambda+1)}{\xi}}\Big)}w^{1/2}\Big(\frac{i(\lambda+1)}{\xi}\Big), \,\,\phi\in\Phi,\,\,\lambda\in\Omega^T_\xi.
\end{equation*}
In particular, 
\begin{equation}\lb{value}
\langle\delta^0_{-1-i\lambda\xi},\phi\rangle=\overline{\phi(\overline{\lambda})}w^{1/2}(\lambda), \,\,\phi\in\Phi,\,\,\lambda\in\Omega.
\end{equation}
\begin{remark}
	Note that $\delta^0_{-1-i\lambda\xi}$ is similar to the standard Dirac delta functional with the weight $w^{1/2}$. And the difference between two functionals is that the space of the test functions for $\delta^0_{-1-i\lambda\xi}$ is $\Phi$ instead of the space of infinitely differentiable functions with compact support. 
\end{remark}
We now describe an extension of the operator $M_\xi$ from $H$ to $\Phi^*$. 
First, we introduce the operator $T$
\begin{align*}
\begin{split}
T :&=M_\xi+I,\,\,\dom(T)=\dom(M_\xi)\subset L_w^{2}(\bbR).
\end{split}
\end{align*}

\begin{definition}
Let  $\dom(T|\Phi)=\dom(M_\xi|\Phi):=\{\phi:\phi\in\dom(T)\cap\Phi,\, T\phi\in\Phi\}$.
We denote by $\dom(T|\Phi^*)$ the set of those $\phi^*\in\Phi^*$ for which there exists a $\psi^*\in\Phi^*$ such that 
\begin{equation}\lb{defT}
\langle\psi^*,\phi\rangle=-\langle\phi^*,T\phi\rangle
\end{equation}
for all $\phi\in\dom(T|\Phi)$. For a given $\phi^*\in\Phi^*$ we set $T\phi^*:=\psi^*$.
\end{definition}
\begin{lemma}\lb{Tl}
	Let $\phi^*\in\Phi^*$. Then $\phi^*\in\dom(T|\Phi^*)$ if and only if
	\begin{enumerate}
		\item the limit $l(\phi^*):=-\lim_{z\to\infty}i\xi z\phi^*(z)w^{1/2}(z)$ exists, where $\phi^*(\cdot)$ is an analytic representation of $\phi^*$.
		\item the function $z\to -i\xi z\phi^*(z)-l(\phi^*)w^{-1/2}(z)$ belongs to some space $\Phi^*_{-\eta}$, $\eta\in[0,\eps)$.
	\end{enumerate}
\end{lemma}

\begin{proof}
	Assume that (1) and (2) hold. Since for $\phi\in\dom(T|\Phi)$ and $\gamma\in(0,\varepsilon)$ the integral $\int_{\gamma}\overline{\phi(\overline z)}w^{1/2}(z)dz$ is equal to $0$, $\psi^*(z)=-i\xi z\phi^*(z)-l(\phi^*)w^{-1/2}(z)$ satisfies \eqref{defT} for all $\phi\in\dom(T|\Phi)$. Therefore, $\phi^*\in\dom(T|\Phi^*)$.
	
	Now, let $\phi^*\in\dom(T|\Phi^*)$ and $\psi^*(\cdot)$ be an analytic representation of $\psi^*=T\phi^*$. Then, by \eqref{defT}, for some $\gamma\in(0,\varepsilon)$ we have 
	\begin{equation*}
	\int_{\gamma}[-i\xi z\phi^*(z)-\psi^*(z)]\overline{\phi(\overline z)}w(z)dz=0, \,\,\phi\in\dom(T|\Phi).
	\end{equation*}
	Now, pick $\phi\in\dom(T|\Phi)$ with analytic representation $\phi(z)=\frac{1}{2\pi i}(z-\bar\zeta)^{-2}w^{-1/2}(z)$ where $|\Re\zeta|\geq\varepsilon$. Then the necessary part of the lemma follows from the fact that $[(-i\xi \zeta\phi^*(\zeta)-\psi^*(\zeta))w^{1/2}(\zeta)]'=0$ from which (1) and (2) follow.
\end{proof}
Now, we are ready to extend $M_\xi$ from $H$ to $\Phi^*$.
\begin{align}\lb{S}
\begin{split}
 M_\xi&:\dom(M_\xi|\Phi^*)\subset\Phi^*\to\Phi^*,\,\, \dom(M_\xi|\Phi^*)=\dom(T|\Phi^*),\\
 M_\xi\phi^*&=(T-I)\phi^*.
\end{split}
\end{align}
Now, we are ready to describe the generalized eigenfunctions of the operator $M_\xi$.
\begin{lemma}\label{eigM}
Let $\phi^*\in\dom(M_\xi|\Phi^*)$ and 
\begin{equation}\lb{eig}
(M_\xi-\lambda)\phi^*=0.
\end{equation}
If $|\Re\lambda+1|<\e|\xi|$, then $\phi^*=C\delta^0_\lambda$, and if $|\Re\lambda+1|\geq\e|\xi|$, then $\phi^*=0$.
\end{lemma}
\begin{proof}
	By Lemma \ref{Tl} the equation \eqref{eig} is equivalent to the following equation in terms of analytic representations
	 \begin{equation}\lb{ef}
	 -i\xi z\phi^*(z)-\phi^*(z)=l(\phi^*)w^{-1/2}(z)+\lambda\phi^*(z).
	 \end{equation}
	 Thus $\phi^*=C\delta^0_\lambda$ for $|\Re\lambda+1|<\e|\xi|$. And if $|\Re\lambda+1|\geq\e|\xi|$, then $\phi^*(\cdot)$ from \eqref{ef} belongs to $\Phi^*_{-\eta}$ if and only if $l(\phi^*)=0$ which implies that $\phi^*=0$.
\end{proof}
\begin{remark}
	Note that \begin{align*}
	\delta^{a0}_\lambda(z):=\frac{1}{2\pi i}\frac{1}{ z+\frac{i(\lambda+1)}{\xi}}w^{-1/2}(z),\,\,\lambda\in\Omega^T_\xi
	\end{align*}
	are the generalized eigenfunctions of the operator $M^*_\xi$, and 
	\begin{equation}\lb{aaction}
	\langle\delta^{a0}_\lambda,\phi\rangle=\overline{\phi\Big(\overline{\frac{-i(\lambda+1)}{\xi}}\Big)}w^{1/2}\Big(\frac{-i(\lambda+1)}{\xi}\Big), \,\,\phi\in\Phi,\,\,\lambda\in\Omega^T_\xi.
	\end{equation}
\end{remark}
In particular, 
\begin{equation}\lb{valuea}
\langle\delta^{a0}_{-1+i\bar\lambda\xi},\phi\rangle=\overline{\phi({\lambda})}w^{1/2}(\bar\lambda), \,\,\phi\in\Phi,\,\,\lambda\in\Omega.
\end{equation}
Next, we introduce the following transforms
\begin{align}\label{hilbert}
\begin{split}
\dom(\mathcal{S})&={L_w^2(\bbR)},\\
(\mathcal{S}f)(\lambda):&=\int_{\mathbb{R}}\frac{f(z)w^{1/2}(z)}{z-\lambda}dz,\,\,\Im\lambda\neq0,\\
\dom(\hat{\mathcal{S}}_\eta)&={L^2(\bbR)},\\
(\hat{\mathcal{S}}_\eta f)(\lambda):&=\int_{\mathbb{R}}\frac{f(z)}{z-(\lambda+i\eta)}dz,\,\,\lambda\in\mathbb{R},\,\,\eta\neq0\in\bbR,\\
\dom(\hat{\mathcal{S}}_\pm)&={L^2(\bbR)},\\
(\hat{\mathcal{S}}_\pm f)(\lambda):&=\lim_{\eta\to0^\pm}\int_{\mathbb{R}}\frac{f(z)}{z-(\lambda+i\eta)}dz,\,\,\lambda\in\mathbb{R},\\
\dom({\mathcal{S}}_\pm)&={L_w^2(\bbR)},\\
{\mathcal{S}}_\pm f:&=\hat{\mathcal{S}}_\pm (fw^{1/2})\in{L^2(\bbR)}\subset{L_w^2(\bbR)},
\end{split}
\end{align}
where the limit exists for almost all $\lambda\in\mathbb{R}$. Moreover, $\mathcal{S}f$ is a holomorphic function in the upper and lower half-planes, transforms $\hat{\mathcal{S}}_\eta$ and $\hat{\mathcal{S}}_\pm$ could be treated as bounded operator from ${L^2(\bbR)}$ to ${L^2(\bbR)}$ and $\hat{\mathcal{S}}_\pm=\lim_{\eta\to0^\pm}\hat{\mathcal{S}}_\eta$ in the sense of strong convergence of operators in ${L^2(\bbR)}$. Note that if $f$ is holomorphic in a neighborhood of $\lambda$ and $(z-\lambda)^{-1}f(z)$ is integrable on $(-\infty,\lambda-\epsilon)\cup(\lambda+\epsilon,\infty)$, $\epsilon>0$, then 
\begin{align*}
\lim_{\eta\to0^\pm}\int_{\mathbb{R}}\frac{f(z)}{z-(\lambda+i\eta)}dz&=\int_{\gamma_\mp}\frac{f(z)}{z-\lambda}dz.
\end{align*}
Therefore, for $\phi\in\Phi$ we extend ${\mathcal{S}}_\pm$ as follows:
\begin{equation*}
({\mathcal{S}}_\pm \phi)(\lambda):=(\hat{\mathcal{S}}_\pm (\phi w^{1/2}))(\lambda)=\int_{\gamma_\mp}\frac{\phi(z)w^{1/2}(z)}{z-\lambda}dz,\,\,\phi\in\Phi,\,\,\lambda\in\Omega.
\end{equation*}
\begin{remark}\lb{Hs}
	Notice that 
	\begin{itemize}
		\item $
		\hat{\mathcal{S}}_\eta=\left\{
		\begin{array}{ll}
		2\pi i\mathcal{F}^{-1}\chi_+(\cdot)e^{-\eta (\cdot)}\mathcal{F},\,\,\,\eta>0,\\
		-2\pi i\mathcal{F}^{-1}\chi_-(\cdot)e^{-\eta (\cdot)}\mathcal{F},\,\,\,\eta<0,
		\end{array}
		\right.
		$
		where $\mathcal{F}$ is the Fourier transform and $\chi_\pm$ are the characteristic functions of the semi-axes $\bbR_\pm$. 
		\item $\hat{\mathcal{S}}_\eta f'=(\hat{\mathcal{S}}_\eta f)'$ and $\hat{\mathcal{S}}_\eta f\in H^1(\bbR)$ for $f\in H^1(\bbR)$.
		\item $\hat{\mathcal{S}}_\pm=\lim_{\eta\to0^\pm}\hat{\mathcal{S}}_\eta=\pm2\pi i\mathcal{F}^{-1}\chi_\pm(\cdot)\mathcal{F}$ in the sense of strong operator convergence.
		\item $\hat{\mathcal{S}}_\pm f'=(\hat{\mathcal{S}}_\pm f)'$ and $\hat{\mathcal{S}}_\pm f\in H^1(\bbR)$	for $f\in H^1(\bbR)$.
		\item $({\mathcal{S}}_\pm (fw^{1/2}))'=(\hat{\mathcal{S}}_\pm (fw))'={\mathcal{S}}_\pm (f'w^{1/2})-2{\mathcal{S}}_\pm (vfw^{1/2})$ and ${\mathcal{S}}_\pm (fw^{1/2})\in H_w^1(\bbR)$	for $f\in H_w^1(\bbR)$.
	\end{itemize}
\end{remark}
\begin{proposition}\label{omegareal}
	Fix $\xi\neq0$. Then the functions $\omega_\pm(-1-\cdot i\xi,\xi)$ defined in \eqref{detpm} have the following properties:
	\begin{enumerate}
		\item  $\omega_\pm(-1-\cdot i\xi,\xi)$ are analytic functions and $\lim_{|\lambda|\to\infty}\omega_\pm(-1-\lambda i\xi,\xi)=1$.
		\item if $\xi\neq\pm\sqrt{\pi}$, then $\omega_\pm(-1-\cdot i\xi,\xi)$ don't vanish for all $\lambda\in\bbR$.
		\item if $\xi=\pm\sqrt{\pi}$, then $\omega_-(-1-\lambda i\xi,\xi)$ doesn't vanish for all $\lambda\in\bbR$, $\omega_+(-1-\lambda i\xi,\xi)$ vanishes when $\lambda=0$, $\omega'_+(-1,\pm\sqrt{\pi})\neq0$, and $\omega_+(-1-\lambda i\xi,\xi)$ doesn't vanish for all non-zero real $\lambda$.
	\end{enumerate}
\end{proposition}
\begin{proof}

		(1) The functions $\omega_\pm(-1-\cdot i\xi,\xi)$ defined in \eqref{detpm} could be analytically continued as 
		\begin{equation*}
		\omega_\pm(-1-i\lambda\xi,\xi)=1-\frac{1}{i\xi}\int_\mathbb{\gamma_\mp(\xi)}\frac{w(z)dz}{z-\lambda}.
		\end{equation*}
		(2) and (3) follow from formulas \eqref{Iman} and \eqref{Iman1} and from the fact that 
\begin{align*}
\begin{split}
\omega'_+(-1,\sqrt{\pi})=-\frac{1}{\pi}(\hat{\mathcal{S}}_+ w)'(0)=-\frac{1}{\pi}(\hat{\mathcal{S}}_+ w')(0)=-\frac{1}{\pi}\big(\pi iw'(0)+P.v.\int_\bbR\frac{w'(v)dv}{v}\big)=\frac{2}{\pi}.
\end{split}
\end{align*}
	Similarly, one can show that $\omega'_+(-1,-\sqrt{\pi})=\frac{2}{\pi}$.
\end{proof}

\begin{proposition}\label{MR=RM} Let $\xi\neq0$. Then
	\begin{align}\lb{residen}
	\begin{split}
	[R^0_\pm(\lambda,\xi)\mathbb{1}](z)&=-\frac{1}{2\pi\xi}\frac{w^{-1/2}(z)}{ z-\tilde\lambda}[{({\mathcal{S}}\mathbb{1})(z)}-{({\mathcal{S}}_{\pm\sign{\xi}} \mathbb{1})(\tilde\lambda)}],\\
	&\,\,\,\,\,\,\,\hbox{where}\,\,\,\,\tilde\lambda=\frac{1}{\xi}(1+\lambda)i,\,\,\lambda\in\sigma_{ess}(M_\xi).
	\end{split}
	\end{align}
	Moreover, if $\mathfrak{i}$ is the natural inclusion map $\mathfrak{i}: H\to\Phi^*$ (see \eqref{inclusion}) and $f\in H$, then one can show that the analytic representation of $\mathfrak{i}f$ is of the form:
	\begin{align}\label{increpres}
	\begin{split}
	(\mathfrak{i}f)(z)&=-\frac{1}{2\pi i}w^{-1/2}(z)(\mathcal{S}f)(z),\,\,\,\Im z\neq0.
	\end{split}
	\end{align}
	Also, we have 
	\begin{align}\lb{residen1}
	\begin{split}
	(M_\xi-\lambda) R^0_\pm(\lambda,\xi)\phi&=\mathfrak{i}\phi, \,\,\phi\in\Phi,\\
	R^0_\pm(\lambda,\xi)(M_\xi-\lambda) \phi&=\mathfrak{i}\phi, \,\,\phi\in\dom(M_\xi|\Phi).
	\end{split}
	\end{align}
\end{proposition}
\begin{proof} Let $\tilde\gamma$ be a positive number such that the contour from $\int_{\tilde\gamma}$ (see \eqref{congamma}) is outside of both $\gamma_-(\xi)$ and $\gamma_+(\xi)$. Then
		\begin{align}\label{R0repres}
		\begin{split}
		\langle R^0_\pm(\lambda, \xi)\phi,\psi\rangle&=\int_{\gamma_\mp(\xi)}\frac{\phi(z)\overline{\psi(\overline{z})}}{-i\xi z-1-\lambda}w(z)dz\stackrel{\tilde\lambda=\frac{1}{\xi}(1+\lambda)i}{=}-\frac{1}{i\xi}\int_{\gamma_\mp(\xi)}\frac{\phi(z)\overline{\psi(\overline{z})}}{ z-\tilde\lambda}w(z)dz\\
		&=\frac{1}{2\pi\xi}\int_{\gamma_\mp(\xi)}\frac{\phi(z)}{ z-\tilde\lambda}\int_{\tilde\gamma}\frac{\overline{\psi(\overline{y})}w^{1/2}(y)}{ y-z}dyw^{1/2}(z)dz\\
		&=\frac{1}{2\pi\xi}\int_{\tilde\gamma}\int_{\gamma_\mp(\xi)}\frac{\phi(z)w^{1/2}(z)}{ (z-\tilde\lambda)(y-z)}dz\overline{\psi(\overline{y})}w^{1/2}(y)dy\\
		&=\frac{1}{2\pi\xi}\int_{\tilde\gamma}\frac{-1}{y-\tilde\lambda}\Big[\int_{\gamma_\mp(\xi)}\frac{\phi(z)w^{1/2}(z)}{ z-y}dz-\int_{\gamma_\mp(\xi)}\frac{\phi(z)w^{1/2}(z)}{ z-\tilde\lambda}dz\Big]\overline{\psi(\overline{y})}w^{1/2}(y)dy.
		\end{split}	
	\end{align}
	Similarly, one can show that
	\begin{equation*}
	\langle \mathfrak{i}f,\psi\rangle=(f,\psi)_{L^2_w(\bbR)}=\frac{1}{2\pi i}\int_{\tilde\gamma}\Big[\int_{\gamma_\mp(\xi)}\frac{\phi(z)w^{1/2}(z)}{ y-z}dz\Big]\overline{\psi(\overline{y})}w^{1/2}(y)dy.
	\end{equation*}
	Next, the first formula in \eqref{residen1} follows from the extension formula of $M_\xi$ \eqref{S}, \eqref{increpres} and \eqref{R0repres}. Finally, the first formula in \eqref{residen1} follows from the definitions of $\dom(M_\xi|\Phi)$ and $R^0_\pm(\lambda, \xi)$, and \eqref{increpres}.
\end{proof}
\begin{definition}
The operator $L_{\xi}$ extended to the space $\Phi^*$ is defined as the sum of the operators
$M_\xi$ and $V$ extended to this space. The domain of the extended
operator is $\dom(M_\xi|\Phi^*)$.
\begin{align*}
\begin{split}
L_{\xi}&:\dom(L_{\xi}|\Phi^*)\subset\Phi^*\to\Phi^*,\,\, \dom(L_{\xi}|\Phi^*)=\dom(M_\xi|\Phi^*),\\
L_{\xi}\phi^*&=(M_\xi+\mathfrak{i}V)\phi^*.
\end{split}
\end{align*}
\end{definition}
We now determine the generalized eigenfunctions of the extended operator $L_{\xi}$.
\begin{proposition}\lb{delta}
Fix $\xi\ne0$. Let $\lambda\in\Omega^T_\xi$ and let $K_+(\lambda,\xi)$ ($K_-(\lambda,\xi)$) be invertible. If
\begin{equation}\lb{eigeq}
(L_{\xi}-\lambda)\phi^*=0, \,\,\phi^*\in\dom(L_{\xi}|\Phi^*),
\end{equation}
then $\phi^*=C^+\delta^+_\lambda$ ($\phi^*=C^-\delta^-_\lambda$), where 
\begin{equation}\lb{deltaformula}
\delta^\pm_\lambda=(1-R^0_\pm(\lambda,\xi)V K_\pm^{-1}(\lambda,\xi)V)\delta^0_\lambda.
\end{equation}
\end{proposition}
\begin{proof}
	We rewrite equation \eqref{eigeq} as $(M_\xi-\lambda)\phi^*+\mathfrak{i}V\phi^*=0.$
	By applying Proposition \ref{MR=RM}, we rewrite the equation \eqref{eigeq} as
	\begin{equation}\lb{mreq}
	(M_\xi-\lambda)\phi^*+(M_\xi-\lambda) R^0_\pm(\lambda,\xi)V\phi^*=0.
	\end{equation}
	By Lemma \ref{eigM}, the general solution of this equation is of the form
	\begin{equation}\lb{geneq}
	\phi^*+R^0_\pm(\lambda,\xi)V\phi^*=C^\pm\delta^0_\lambda.
	\end{equation}
	After applying the operator $V$ to both sides of \eqref{geneq} and using the fact that $V^2=V$, we obtain
	\begin{equation*}
	(I+VR^0_\pm(\lambda,\xi)V)V\phi^*=C^\pm V\delta^0_\lambda.
	\end{equation*}
	Using formula \eqref{Kext}, we arrive at
	\begin{equation}\lb{solphi}
     K_\pm(\lambda,\xi)V\phi^*=C^\pm V\delta^0_\lambda.
	\end{equation}
	After substituting the solution of \eqref{solphi} into \eqref{geneq}, we arrive at
	\begin{equation*}
	\phi^*=C^\pm\delta^\pm_\lambda=C^\pm(1-R^0_\pm(\lambda,\xi)V K_\pm^{-1}(\lambda,\xi)V)\delta^0_\lambda.
	\end{equation*}
\end{proof}
\begin{remark}\lb{da+-}
	One can also show that if $\lambda\in\Omega^T_\xi$ and $K_\pm^{}(\lambda,-\xi)$ are invertible, then
	\begin{equation*}
	\delta^{a\pm}_{\lambda}=(1-R^0_\pm(\lambda,-\xi)V K_\pm^{-1}(\lambda,-\xi)V)\delta^{a0}_\lambda
	\end{equation*}
	are the generalized eigenfunctions of the extended operator $L^*_{\xi}$.
\end{remark}

In order to prove the jump formulas for the resolvents, we need the following auxiliary result: 
\begin{proposition}\label{Kinvert} Fix $\xi\ne0$. Let $\lambda\in\Omega^T_\xi$ and let $K_\pm^{}(\lambda,\xi)$ be invertible. Then
	\begin{align}\lb{inverse}
	\begin{split}
	K_\pm(\lambda,\xi)^{-1}&=I-\frac{1}{\omega_\pm(\lambda,\xi)}VR_\pm^0(\lambda,\xi)V,\\
	K_\pm(\lambda,\xi)^{-1}\mathbb{1}&=\frac{1}{\omega_\pm(\lambda,\xi)}\mathbb{1}.
	\end{split}
	\end{align}
\end{proposition}
\begin{proof} 
	Let $f,g \in H={L_w^2(\bbR)}$ and assume that $K_\pm(\lambda,\xi)f=g$. Then 
	\begin{equation*}
	g=(I+VR_\pm^0(\lambda,\xi)V)f=f+VR_\pm^0(\lambda,\xi)Vf=f+(f,\mathbb{1})_{L_w^2}\langle R_\pm^0(\lambda,\xi)\mathbb{1},\mathbb{1}\rangle\mathbb{1}.
	\end{equation*}
	Therefore,
	\begin{equation*}
	(g,\mathbb{1})_{L_w^2}=(f,\mathbb{1})_{L_w^2}(1+\langle R_\pm^0(\lambda,\xi)\mathbb{1},\mathbb{1}\rangle)=\omega_\pm(\lambda,\xi)(f,\mathbb{1})_{L_w^2}.
	\end{equation*}
	Hence,
	\begin{equation*}
	f=g-(f,\mathbb{1})_{L_w^2}\langle R_\pm^0(\lambda,\xi)\mathbb{1},\mathbb{1}\rangle\mathbb{1}=g-\frac{1}{\omega_\pm(\lambda,\xi)}VR_\pm^0(\lambda,\xi)Vg.
	\end{equation*}
	Moreover, if $g=\mathbb{1}$, then
	\begin{equation*}
	f=\mathbb{1}-\frac{1}{\omega_\pm(\lambda,\xi)}VR_\pm^0(\lambda,\xi)V\mathbb{1}=\Big(1-\frac{\langle R_\pm^0(\lambda,\xi)\mathbb{1},\mathbb{1}\rangle}{\omega_\pm(\lambda,\xi)}\Big)\mathbb{1}=\frac{1}{\omega_\pm(\lambda,\xi)}\mathbb{1}.
	\end{equation*}
\end{proof}

\begin{proposition}[Jump formulas]\lb{Jump}
	Fix $\xi\ne0$. Let $\lambda\in\Omega^T_\xi$. Then
	\begin{enumerate}
		\item 	\begin{align*}
		\begin{split}
		(R_+^0(\lambda,\xi)-R_-^0(\lambda,\xi))\phi&=-\frac{2\pi} {|\xi|}\phi\Big(\frac{i(\lambda+1)}{\xi}\Big)w^{1/2}\Big(\frac{i(\lambda+1)}{\xi}\Big)\delta^0_\lambda\\
		&=-\frac{2\pi} {|\xi|}\langle\phi,\delta^{a0}_{\bar\lambda}\rangle\delta^0_\lambda ,\,\,\,\phi\in\Phi,
		\end{split}
		\end{align*}
		where $\delta^{a0}_{\bar\lambda}$ and $\delta^{0}_{\lambda}$ are generalized eigenfunctions of the operators $M_\xi^*$ and $M_{\xi}$, respectively.
	\item If $K_+(\lambda,\xi)$ and $K_-^{}(\bar\lambda,-\xi)$ are invertible, then
	\begin{align}\label{jumps}
	\begin{split}
	(R_+(\lambda,\xi)-R_-(\lambda,\xi))\phi&=-\frac{2\pi} {|\xi|}\langle\phi,\delta^{a-}_{\bar\lambda}\rangle\delta^+_\lambda ,\,\,\,\phi\in\Phi,
	\end{split}
	\end{align}
	where $\delta^{a-}_{\bar\lambda}$ and $\delta^{+}_{\lambda}$ are generalized eigenfunctions of the operators $L_\xi^*$ and $L_{\xi}$, respectively.
	\end{enumerate}

\end{proposition}
\begin{proof}
	The first line in \eqref{jumps} follows directly from \eqref{jump0}.\\
	Similar to \eqref{residen1}, we know that $(L_\xi-\lambda) R_\pm(\lambda,\xi)\phi=\mathfrak{i}\phi$ for any $\phi\in\Phi$. Therefore, if $\phi^*=(R_+(\lambda,\xi)-R_-(\lambda,\xi))\phi$, then $(L_\xi-\lambda)\phi^*=0$. Then by Proposition \ref{delta}, 
	\begin{equation*}
	\phi^*=(R_+(\lambda,\xi)-R_-(\lambda,\xi))\phi=C^+\delta^+_\lambda.
	\end{equation*}
	Moreover, if follows from \eqref{mreq} that 
	\begin{equation*}
	C^+\delta^0_\lambda(z)=\frac{l(\phi^*+R^0_+(\lambda,\xi)V\phi^*)w^{-1/2}(z)}{-i\xi z-\lambda-1}=-\frac{2\pi}{\xi} l(\phi^*+R^0_+(\lambda,\xi)V\phi^*)\delta^0_\lambda(z).
	\end{equation*}

Therefore, $C^+=-\frac{2\pi}{\xi} l(\phi^*+R^0_+(\lambda,\xi)V\phi^*)$.
Next, we compute $l(\phi^*)$.
\begin{align*}
\begin{split}
l(\phi^*)=-\lim_{z\to\infty}(i\xi z\phi^*(z)w^{1/2}(z))=-i\xi\lim_{z\to\infty} (z[(R_+(\lambda,\xi)-R_-(\lambda,\xi))\phi] (z)w^{1/2}(z)).
\end{split}
\end{align*}
In particular, 
\begin{align*}
\begin{split}
&\lim_{z\to\infty} (z[(R_\pm(\lambda,\xi)\phi] (z)w^{1/2}(z))\overset{\eqref{purtR}}{=}\lim_{z\to\infty} (z[(R^0_\pm(\lambda,\xi)-R^0_\pm(\lambda,\xi)VK^{-1}_\pm(\lambda, \xi)VR^0_\pm(\lambda,\xi))\phi](z)w^{1/2}(z))\\
&\overset{\eqref{inverse}}{=}\lim_{z\to\infty} (z[(R^0_\pm(\lambda,\xi)-R^0_\pm(\lambda,\xi)VR^0_\pm(\lambda,\xi)+\frac{1}{\omega_\pm(\lambda,\xi)}R^0_\pm(\lambda,\xi)VR_\pm^0(\lambda,\xi)VR^0_\pm(\lambda,\xi))\phi](z)w^{1/2}(z))\\
&\overset{\eqref{R0repres}}{=}\frac{1}{2\pi \xi}{({\mathcal{S}}_{\pm\sign{\xi}} \mathbb{\phi})(\tilde\lambda)}+\frac{1}{2\pi \xi}[-\langle R_\pm^0(\lambda,\xi)\phi,\mathbb{1}\rangle+\frac{\langle R_\pm^0(\lambda,\xi)\phi,\mathbb{1}\rangle\langle R_\pm^0(\lambda,\xi)\mathbb{1},\mathbb{1}\rangle}{\omega_\pm(\lambda,\xi)}]{({\mathcal{S}}_{\pm\sign{\xi}} \mathbb{1})(\tilde\lambda)}\\
&=\frac{1}{2\pi \xi}{({\mathcal{S}}_{\pm\sign{\xi}} \mathbb{\phi})(\tilde\lambda)}-\frac{1}{2\pi \xi}\frac{\langle R_\pm^0(\lambda,\xi)\phi,\mathbb{1}\rangle}{\omega_\pm(\lambda,\xi)}{({\mathcal{S}}_{\pm\sign{\xi}} \mathbb{1})(\tilde\lambda)}.
\end{split}
\end{align*}
Hence,
\begin{align*}
\begin{split}
l(\phi^*)&=\frac{-i}{2\pi}\Big(({\mathcal{S}}_{\sign{\xi}} \mathbb{\phi})(\tilde\lambda)-({\mathcal{S}}_{-\sign{\xi}} \mathbb{\phi})(\tilde\lambda)-\frac{\langle R_+^0(\lambda,\xi)\phi,\mathbb{1}\rangle}{\omega_+(\lambda,\xi)}({\mathcal{S}}_{\sign{\xi}} \mathbb{1})(\tilde\lambda)+\frac{\langle R_-^0(\lambda,\xi)\phi,\mathbb{1}\rangle}{\omega_-(\lambda,\xi)}({\mathcal{S}}_{-\sign{\xi}} \mathbb{1})(\tilde\lambda)\Big)\\
&\overset{\eqref{aaction}}{=}\frac{-i}{2\pi}\Big(2\pi i\sign(\xi)\langle\phi,\delta^{a0}_{\bar\lambda}\rangle-\frac{\langle R_+^0(\lambda,\xi)\phi,\mathbb{1}\rangle}{\omega_+(\lambda,\xi)}({\mathcal{S}}_{\sign{\xi}} \mathbb{1})(\tilde\lambda)+\frac{\langle R_-^0(\lambda,\xi)\phi,\mathbb{1}\rangle}{\omega_-(\lambda,\xi)}({\mathcal{S}}_{-\sign{\xi}} \mathbb{1})(\tilde\lambda)\Big).
\end{split}
\end{align*}
Now, we compute $l(R^0_+(\lambda,\xi)V\phi^*)$. By \eqref{R0repres}, we have $l(R^0_+(\lambda,\xi)V\phi^*)=\frac{-i}{2\pi}\langle \phi^*,\mathbb{1}\rangle({\mathcal{S}}_{\sign{\xi}} \mathbb{1})(\tilde\lambda)$. In particular,
\begin{align*}
\begin{split}
\langle \phi^*,\mathbb{1}\rangle&=\langle (R_+(\lambda,\xi)-R_-(\lambda,\xi))\phi,\mathbb{1}\rangle\overset{\eqref{Radj}}{=}\langle \phi,(R_+(\bar\lambda,-\xi)-R_-(\bar\lambda,-\xi))\mathbb{1}\rangle\\
&\overset{\eqref{purtR}, \eqref{inverse}}{=}\langle \phi,[1-\langle R_+^0(\bar\lambda,-\xi)\mathbb{1},\mathbb{1}\rangle+\frac{\langle R_+^0(\bar\lambda,-\xi)\mathbb{1},\mathbb{1}\rangle^2}{\omega_+(\bar\lambda,-\xi)}]R^0_+(\bar\lambda,-\xi)\mathbb{1}\\
&-[1-\langle R_-^0(\bar\lambda,-\xi)\mathbb{1},\mathbb{1}\rangle+\frac{\langle R_-^0(\bar\lambda,-\xi)\mathbb{1},\mathbb{1}\rangle^2}{\omega_-(\bar\lambda,-\xi)}]R^0_-(\bar\lambda,-\xi)\mathbb{1}\rangle\\
&=\langle \phi,\frac{1}{\omega_+(\bar\lambda,-\xi)}R^0_+(\bar\lambda,-\xi)\mathbb{1}-\frac{1}{\omega_-(\bar\lambda,-\xi)}R^0_-(\bar\lambda,-\xi)\mathbb{1}\rangle.
\end{split}
\end{align*}
Therefore, we now compute $C^+=-\frac{2\pi}{\xi} l(\phi^*+R^0_+(\lambda,\xi)V\phi^*)$.
\begin{align*}
\begin{split}
C^+&=\frac{i}{\xi}\Big(2\pi i\sign(\xi)\langle\phi,\delta^{a0}_{\bar\lambda}\rangle-\frac{\langle R_+^0(\lambda,\xi)\phi,\mathbb{1}\rangle}{\omega_+(\lambda,\xi)}({\mathcal{S}}_{\sign{\xi}} \mathbb{1})(\tilde\lambda)+\frac{\langle R_-^0(\lambda,\xi)\phi,\mathbb{1}\rangle}{\omega_-(\lambda,\xi)}({\mathcal{S}}_{-\sign{\xi}} \mathbb{1})(\tilde\lambda)\\
&+\langle(\frac{1}{\omega_+(\lambda,\xi)}R^0_+(\lambda,\xi)-\frac{1}{\omega_-(\lambda,\xi)}R^0_-(\lambda,\xi))\phi,\mathbb{1}\rangle({\mathcal{S}}_{\sign{\xi}} \mathbb{1})(\tilde\lambda)\Big)\\
&=\frac{i}{\xi}\Big(2\pi i\sign(\xi)\langle\phi,\delta^{a0}_{\bar\lambda}\rangle-\frac{\langle R_-^0(\lambda,\xi)\phi,\mathbb{1}\rangle}{\omega_-(\lambda,\xi)}(({\mathcal{S}}_{\sign{\xi}} \mathbb{1})(\tilde\lambda)-({\mathcal{S}}_{-\sign{\xi}} \mathbb{1})(\tilde\lambda))\Big)\\
&=-\frac{2\pi}{|\xi|}\langle\phi,\delta^{a0}_{\bar\lambda}\rangle\big(1-\frac{\langle R_-^0(\lambda,\xi)\phi,\mathbb{1}\rangle}{\omega_-(\lambda,\xi)}\big)=-\frac{2\pi}{|\xi|}\langle\phi,\big(1-\frac{\langle R_-^0(\bar\lambda,-\xi)\phi,\mathbb{1}\rangle}{\omega_-(\bar\lambda,-\xi)}\big)\delta^{a0}_{\bar\lambda}\rangle\\
&\overset{\eqref{inverse}, \,Remark \,\ref{da+-}}{=}-\frac{2\pi}{|\xi|}\langle\phi,\delta^{a-}_{\bar\lambda}\rangle.
\end{split}
\end{align*}
\end{proof}
\section{The generalized Fourier transforms}

Let $\xi\neq0$, $\lambda\in\Omega$ and $\phi\in\Phi$. We introduce the following transforms (the generalized Fourier transforms) $\mathcal{U}_\xi$ and $\mathcal{B}_\xi$:
\begin{align*}
\begin{split}
{(\mathcal{U}_\xi \phi)(\lambda)}&=\frac{1}{w^{1/2}(\bar\lambda)}\langle\phi,\delta^+_{-1-i\lambda\xi}\rangle,\\
(\mathcal{B}_\xi \phi)(\lambda)&=\frac{1}{w^{1/2}(\lambda)}\langle\phi,\delta^{a-}_{-1+i\bar\lambda\xi}\rangle.
\end{split}
\end{align*}

Now, we can prove the following proposition:
	\begin{proposition} \lb{FT}
	Let $\xi\neq0$, $\lambda\in\Omega$ and $\phi\in\Phi$. Then
	\begin{align*}
	\begin{split}
   {(\mathcal{U}_\xi \phi)(\lambda)}&=\phi(\bar\lambda)-\frac{1}{i\xi\overline{\omega_+(-1-i\lambda\xi,\xi)}}({\mathcal{S}}_{-\sign\xi}(\phi w^{1/2}))(\bar\lambda),\\
	(\mathcal{B}_\xi \phi)(\lambda)&=\phi(\lambda)-\frac{1}{-i\xi\overline{\omega_-(-1+i\bar\lambda\xi,-\xi)}}({\mathcal{S}}_{\sign\xi}(\phi w^{1/2}))(\lambda).
	\end{split}
	\end{align*}
\end{proposition}
\begin{proof}
	It follows from Proposition \ref{delta} that 
	\begin{align*}
	\begin{split}
	\langle\phi,\delta^+_{-1-i\lambda\xi}\rangle&=\langle\phi,(1-R^0_+({-1-i\lambda\xi},\xi)V K_+^{-1}({-1-i\lambda\xi},\xi)V)\delta^0_{-1-i\lambda\xi}\rangle=\langle\phi,\delta^0_{-1-i\lambda\xi}\rangle\\
	&-\langle\phi,R^0_+({-1-i\lambda\xi},\xi)V K_+^{-1}({-1-i\lambda\xi},\xi)V\delta^0_{-1-i\lambda\xi}\rangle
	\end{split}
	\end{align*}
	It follows from \eqref{V} and \eqref{value} that
	\begin{equation}\lb{vdelta}
	V\delta^0_{-1-i\lambda\xi}=\langle\delta^0_{-1-i\lambda\xi},\mathbb{1}\rangle\mathbb{1}=w^{1/2}(\lambda)\mathbb{1}.
	\end{equation}
	Using formulas \eqref{inverse} and \eqref{vdelta}, we arrive at
	\begin{equation}\lb{functional}
	R^0_+({-1-i\lambda\xi},\xi)V K_+^{-1}({-1-i\lambda\xi},\xi)V\delta^0_{-1-i\lambda\xi}=\frac{w^{1/2}(\lambda)}{\omega_+(-1-i\lambda\xi,\xi)}R^0_+({-1-i\lambda\xi},\xi)\mathbb{1}.
	\end{equation}
	It follows from formulas \eqref{Rintegral}, \eqref{value} and \eqref{functional} that
	\begin{align}\lb{deltaact}
	\begin{split}
	\langle\phi,\delta^+_{-1-i\lambda\xi}\rangle=&\phi(\bar\lambda)w^{1/2}(\bar\lambda)-\frac{w^{1/2}(\bar\lambda)}{\overline{\omega_+(-1-i\lambda\xi,\xi)}}\overline{\int_{\gamma_-(\xi)}\frac{\overline{\phi(\overline{z})}}{-i\xi z+i\lambda\xi}w(z)dz}\\
	=&\phi(\bar\lambda)w^{1/2}(\bar\lambda)-\frac{w^{1/2}(\bar\lambda)}{i\xi\overline{\omega_+(-1-i\lambda\xi,\xi)}}\overline{\int_{\gamma_-(\xi)}\frac{\overline{\phi(\overline{z})}}{ z-\lambda}w(z)dz}\\
	=&\phi(\bar\lambda)w^{1/2}(\bar\lambda)-\frac{w^{1/2}(\bar\lambda)}{i\xi\overline{\omega_+(-1-i\lambda\xi,\xi)}}{\int_{\gamma_+(\xi)}\frac{{\phi({z})}}{ z-\bar\lambda}w(z)dz}
	\end{split}
	\end{align}
	Similarly, 
	\begin{align*}
	\begin{split}
	\langle\phi,\delta^{a-}_{-1+i\bar\lambda\xi}\rangle&=\langle\phi,(1-R^0_-({-1+i\bar\lambda\xi},-\xi)V K_-^{-1}({-1+i\bar\lambda\xi},-\xi)V)\delta^{a0}_{-1+i\bar\lambda\xi}\rangle=\langle\phi,\delta^{a0}_{-1+i\bar\lambda\xi}\rangle\\
	&-\langle\phi,R^0_-({-1+i\bar\lambda\xi},-\xi)V K_-^{-1}({-1+i\bar\lambda\xi},-\xi)V\delta^{a0}_{-1+i\bar\lambda\xi}\rangle
	\end{split}
	\end{align*}
	It follows from \eqref{valuea} that
	\begin{equation}\lb{avdelta}
	V\delta^{a0}_{-1+i\bar\lambda\xi}=\langle\delta^{a0}_{-1+i\bar\lambda\xi},\mathbb{1}\rangle\mathbb{1}=w^{1/2}(\bar\lambda)\mathbb{1}.
	\end{equation}
	Using formulas \eqref{inverse}  and \eqref{avdelta}, we arrive at
	\begin{equation*}
	R^0_-({-1+i\lambda\xi},-\xi)V K_-^{-1}({-1+i\bar\lambda\xi},-\xi)V\delta^{a0}_{-1+i\lambda\xi}=\frac{w^{1/2}(\bar\lambda)}{\omega_-(-1+i\bar\lambda\xi,-\xi)}R^0_-({-1+i\bar\lambda\xi},-\xi)\mathbb{1}.
	\end{equation*}
	Similar to \eqref{deltaact}, we have
	\begin{align*}
	\begin{split}
	\langle\phi,\delta^{a-}_{-1+i\bar\lambda\xi}\rangle=&\phi(\lambda)w^{1/2}(\lambda)-\frac{w^{1/2}(\lambda)}{-i\xi\overline{\omega_-(-1+i\bar\lambda\xi,-\xi)}}{\int_{\gamma_-(\xi)}\frac{{\phi({z})}}{ z-\lambda}w(z)dz}.
	\end{split}
	\end{align*}
\end{proof}
Based on Corollary \ref{Hs} and Remark \ref{omegareal}, we are ready to extend the generalized Fourier transforms.
\begin{itemize}
	\item For $\xi\neq0,\pm\sqrt{\pi}$\begin{align}\lb{transf}
	\begin{split}
	\mathcal{U}_\xi:&\, H^s_w(\bbR; dv) \to H^s_w(\bbR; d\lambda),\\
	(\mathcal{U}_\xi f)(\lambda)&=f(\lambda)-\frac{1}{i\xi\overline{\omega_+(-1-i\lambda\xi,\xi)}}(\mathcal{S}_{-\sign\xi}(fw^{1/2}))(\lambda),\,\,\,\hbox{for}\,\,f\in\dom(\mathcal{U}_\xi)=H^s_w(\bbR; dv),\\
	\mathcal{B}_\xi:&\, H^s_w(\bbR; dv) \to H^s_w(\bbR; d\lambda),\\
	(\mathcal{B}_\xi f)(\lambda)&=f(\lambda)-\frac{1}{-i\xi\overline{\omega_-(-1+i\lambda\xi,-\xi)}}({\mathcal{S}}_{\sign\xi}(fw^{1/2}))(\lambda),\,\,\,\hbox{for}\,\,f\in\dom(\mathcal{B}_\xi)=H^s_w(\bbR; dv).
	\end{split}
	\end{align}
	\item For $\xi=\pm\sqrt{\pi}$\begin{align}\lb{transf1}
	\begin{split}
	\lambda\mathcal{U}_\xi:&\, H^s_w(\bbR; dv) \to H^s_w(\bbR; d\lambda),\\
	(\lambda\mathcal{U}_\xi f)(\lambda)&=\lambda f(\lambda)-\frac{\lambda}{i\xi\overline{\omega_+(-1-i\lambda\xi,\xi)}}(\mathcal{S}_{-\sign\xi}(fw^{1/2}))(\lambda),\,\,\,\hbox{for}\,\,f\in\dom(\lambda\mathcal{U}_\xi)=H^s_w(\bbR; dv),\\
	\mathcal{B}_\xi:&\, H^s_w(\bbR; dv) \to H^s_w(\bbR; d\lambda),\\
	(\mathcal{B}_\xi f)(\lambda)&=f(\lambda)-\frac{1}{-i\xi\overline{\omega_-(-1+i\lambda\xi,-\xi)}}({\mathcal{S}}_{\sign\xi}(fw^{1/2}))(\lambda),\,\,\,\hbox{for}\,\,f\in\dom(\mathcal{B}_\xi)=H^s_w(\bbR; dv).
	\end{split}
	\end{align}
\end{itemize}

\begin{proposition}\label{L_action}
	Let $\xi\neq0,\pm\sqrt{\pi}$, $\lambda\in\bbR$. Then
	\begin{align}
	\begin{split}
	(\mathcal{B}_\xi L_\xi f)(\lambda)=&(-i\xi\lambda-1)(\mathcal{B}_\xi f)(\lambda),\,\,\,f\in\dom(L_\xi),\\
	(\mathcal{U}_\xi L^*_\xi f)(\lambda)=&(i\xi\lambda-1)(\mathcal{U}_\xi f)(\lambda),\,\,\,\,f\in\dom(L^*_\xi).
	\end{split}
	\end{align}
\end{proposition}
\begin{proof} Let $f\in\dom(L_\xi)$. Then
	\begin{align*}
	\begin{split}
	(\mathcal{B}_\xi L_\xi f)(\lambda)=&(L_\xi f)(\lambda)-\frac{1}{-i\xi\overline{\omega_-(-1+i\lambda\xi,-\xi)}}({\mathcal{S}}_{\sign\xi}(w^{1/2}L_\xi f))(\lambda)\\
	=&(-i\xi\lambda-1)f(\lambda)+(f,\mathbb{1})_{L^2_w}-\frac{1}{-i\xi\overline{\omega_-(-1+i\lambda\xi,-\xi)}}\int_{\gamma_-(\xi)}\frac{{(L_\xi f)({z})}}{ z-\lambda}w(z)dz\\
	=&(-i\xi\lambda-1)f(\lambda)+(f,\mathbb{1})_{L^2_w}-\frac{1}{-i\xi\overline{\omega_-(-1+i\lambda\xi,-\xi)}}\\
	\times&\int_{\gamma_-(\xi)}\frac{[-i\xi(z-\lambda)+(-i\xi\lambda-1)]f(z)+(f,\mathbb{1})_{L^2_w}}{ z-\lambda}w(z)dz\\
	=&(-i\xi\lambda-1)(\mathcal{B}_\xi f)(\lambda)+(f,\mathbb{1})_{L^2_w}-\frac{-i\xi(f,\mathbb{1})_{L^2_w}+(f,\mathbb{1})_{L^2_w}(\hat{\mathcal{S}}_{\sgn\xi}w)(\lambda)}{-i\xi\overline{\omega_-(-1+i\lambda\xi,-\xi)}}\\
	=&(-i\xi\lambda-1)(\mathcal{B}_\xi f)(\lambda)+(f,\mathbb{1})_{L^2_w}-(f,\mathbb{1})_{L^2_w}=(-i\xi\lambda-1)(\mathcal{B}_\xi f)(\lambda).
	\end{split}
	\end{align*}
	Similarly, one can show that $(\mathcal{U}_\xi L^*_\xi f)(\lambda)=(i\xi\lambda-1)(\mathcal{U}_\xi f)(\lambda)$ for $f\in\dom(L^*_\xi)$.
\end{proof}

Notice that for each fixed $\lambda\in\bbR$, $\delta^+_{-1-i\lambda\xi}(\cdot)$ and $\delta^{a-}_{-1+i\lambda\xi}(\cdot)$ can be treated as holomorphic functions outside of the certain strip, that is,
\begin{lemma} Let $\lambda\in\bbR$. Then
	the holomorphic representations of $\delta^+_{-1-i\lambda\xi}$ and $\delta^{a-}_{-1+i\lambda\xi}$ are of the following form:
	\begin{align*}
	\begin{split}
	\delta^+_{-1-i\lambda\xi}(z)&=\frac{1}{2\pi i}\frac{1}{ z-\lambda}w^{-1/2}(z)+\frac{w^{1/2}(\lambda)w^{-1/2}(z)}{2\pi\xi{\omega_+(-1-i\lambda\xi,\xi)}}\frac{1}{ z-\lambda}[{({\mathcal{S}}\mathbb{1})(z)}-{({\mathcal{S}}_+\mathbb{1})(\lambda)}],\\
	\delta^{a-}_{-1+i\lambda\xi}(z)&=\frac{1}{2\pi i}\frac{1}{ z-\lambda}w^{-1/2}(z)-\frac{w^{1/2}(\lambda)w^{-1/2}(z)}{2\pi\xi{\omega_-(-1+i\lambda\xi,-\xi)}}\frac{1}{ z-\lambda}[{({\mathcal{S}}\mathbb{1})(z)}-{({\mathcal{S}}_-\mathbb{1})(\lambda)}].
	\end{split}
	\end{align*}
\end{lemma}
\begin{proof}
	According to \eqref{deltaformula}, $\delta^\pm_\lambda=(1-R^0_\pm(\lambda,\xi)V K_\pm^{-1}(\lambda,\xi)V)\delta^0_\lambda.$ Therefore,
	\begin{equation*}
		\delta^+_{-1-i\lambda\xi}=\delta^0_{-1-i\lambda\xi}-R^0_+({-1-i\lambda\xi},\xi)V K_+^{-1}({-1-i\lambda\xi},\xi)V\delta^0_{-1-i\lambda\xi}.
	\end{equation*}
	It follows from \eqref{delta0} that the holomorphic representation of $\delta^0_{-1-i\lambda\xi}$ is
	\begin{equation*}
	\delta^0_{-1-i\lambda\xi}(z)=\frac{1}{2\pi i}\frac{1}{ z-\lambda}w^{-1/2}(z).
	\end{equation*}
	Next, according to \eqref{functional},
	\begin{equation*}
	R^0_+({-1-i\lambda\xi},\xi)V K_+^{-1}({-1-i\lambda\xi},\xi)V\delta^0_{-1-i\lambda\xi}=\frac{w^{1/2}(\lambda)}{\omega_+(-1-i\lambda\xi,\xi)}R^0_+({-1-i\lambda\xi},\xi)\mathbb{1}.
	\end{equation*}
	Finally, it is clear that the the holomorphic representations of $R^0_+({-1-i\lambda\xi},\xi)\mathbb{1}$ (cf. \eqref{residen}) is
	\begin{equation*}
	[R^0_+({-1-i\lambda\xi},\xi)\mathbb{1}](z)=-\frac{1}{2\pi\xi}\frac{w^{-1/2}(z)}{ z-\lambda}[{({\mathcal{S}}\mathbb{1})(z)}-{({\mathcal{S}}_+\mathbb{1})(\lambda)}].
	\end{equation*}
Similarly, one can derive a formula for the holomorphic representation of $\delta^{a-}_{-1+i\lambda\xi}$.
\end{proof}

\section{Generalized eigenfunction expansion}\label{s:7}

\begin{proposition}\lb{ext}
	For any $f\in  H^1_w(\bbR)$ the following inequality holds
	\begin{equation}\lb{H1}
	\int_{\pm}m(\lambda,\sigma)f(\lambda)w(\lambda)d\lambda\leq C\|f\|_{H_w^1(\bbR)},
	\end{equation}
	where $\sigma$ is a point on the vertical interval $I=\{a+i\eta \,| \,\hbox{$a\in\bbR$ is fixed},\,\,\eta\in[-\mu,\mu],\hbox{$\mu\in\bbR$ is fixed}\}$, for each fixed $\sigma\neq a$ $m$ is a holomorphic function in the neighborhood of real $\lambda$-axis, for $\sigma=a$ $m$ is a holomorphic function in the neighborhood of real $\lambda$-axis except for a single pole $a$ of order $1$,  $m$ is  uniformly bounded at infinity with respect to $\sigma$, $(\lambda-\sigma)m(\lambda,\sigma)$ is bounded over $\{\hbox{a neighborhood of $a$}\}\times I$, and for each fixed $\sigma$
	\begin{equation}\lb{defpm}
	\int_{\pm}m(\lambda,\sigma)f(\lambda)w(\lambda)d\lambda:=\lim_{b\to0^\mp}\int_{\mathbb{R}}\frac{m(\lambda)f(\lambda)w(\lambda)}{\lambda-(\sigma+ib)}d\lambda,
	\end{equation}
	and, finally, $C$ in \eqref{H1} is independent of $\sigma$.
\end{proposition}
\begin{proof}
	Let $e_a$ is a cut-off function corresponding to $a$ and defined on the real line, that is, $e_a$ is infinitely differentiable such that $e_a=0$ outside of a neighborhood of $a$ and $e_a=1$ in some (smaller) neighborhood of $a$. Introduce $E^1_\sigma$ and $E^2_\sigma$ 
	\begin{align*}
	\begin{split}
	E^1_\sigma(\lambda)&=(\lambda-\sigma)m(\lambda,\sigma)e_a(\lambda),\\
	E^2_\sigma(\lambda)&=m(\lambda,\sigma)(1-e_a(\lambda)).
	\end{split}
	\end{align*}
	Then,
	\begin{equation}\lb{H1L1}
\int_{\pm}m(\lambda,\sigma)f(\lambda)w(\lambda)d\lambda=\int_{\pm}\frac{E^1_\sigma(\lambda)f(\lambda)w(\lambda)}{\lambda-\sigma}d\lambda+\int_\bbR E^2_\sigma(\lambda)f(\lambda)w(\lambda)d\lambda
	\end{equation}
	Also, the following point-wise estimate holds for any function $g\in H^1(\bbR)$
	\begin{equation*}
	|(\hat{\mathcal{S}}_\eta g)(a)|^2\leq\|\hat{\mathcal{S}}_\eta g\|^2_{L^2(\bbR)}+\|(\hat{\mathcal{S}}_\eta g)'\|^2_{L^2(\bbR)}.
	\end{equation*}
	Or,
	\begin{equation*}
	|(\hat{\mathcal{S}}_\eta g)(a)|^2\leq\|\hat{\mathcal{S}}_\eta g\|^2_{L^2(\bbR)}+\|\hat{\mathcal{S}}_\eta g'\|^2_{L^2(\bbR)}.
	\end{equation*}
	Since $
	\hat{\mathcal{S}}_\eta=\left\{
	\begin{array}{ll}
	2\pi i\mathcal{F}^{-1}\chi_+(\cdot)e^{-\eta (\cdot)}\mathcal{F},\,\,\,\eta>0\\
	-2\pi i\mathcal{F}^{-1}\chi_-(\cdot)e^{-\eta (\cdot)}\mathcal{F},\,\,\,\eta<0
	\end{array}
	\right.,
	$ the limit of $\hat{\mathcal{S}}_\eta g(a)$ as $\eta\to0^\pm$ exists, and moreover,
		\begin{equation*}
	\Big|\int_{\pm}\frac{g(\lambda)}{\lambda-\sigma}d\lambda\Big|\leq C\|g\|_{H^1(\bbR)},
	\end{equation*}
	where $C$ is independent of $\sigma$.
	Therefore, formula \eqref{H1} follows from \eqref{H1L1} and the fact that $fw\in H^1(\bbR)$ and $\|f\|_{L^1_w(\bbR)}\leq\|f\|_{L^2_w(\bbR)}\|\mathbb{1}\|_{L^2_w(\bbR)}=\|f\|_{L^2_w(\bbR)}$ (that is, $L^2_w(\bbR)\subset L^1_w(\bbR)$).
\end{proof}

\begin{remark}
	Let $f\in L^1_w(\bbR)$ have an extension which is holomorphic in the neighborhood of the pole $\sigma$ of the function $m$ from Proposition  \ref{ext}. Then,
	\begin{equation}\lb{H1ext}
	\int_{\gamma_\pm}m(z,\sigma)f(z)w(z)dz=\int_{\pm}m(\lambda,\sigma)f(\lambda)w(\lambda)d\lambda.
	\end{equation}
\end{remark}

Let the functionals $\delta^{a\pm}_\lambda$ be the generalized eigenfunctions of the operator $L^*_{\xi}$, extended to $\Phi^*$.
\begin{theorem}\label{main_exp}
	Fix $\xi\neq0$. If $\phi\in\dom(S|\Phi)$ and $\psi\in\Phi$, then
	\begin{equation}\lb{gee}
	(\phi,\psi)_{L_w^2}=\int_{-\infty-i\frac{\gamma}{\xi}}^{\infty-i\frac{\gamma}{\xi}}\langle\phi,\delta^{a-}_{-1+i\bar\lambda\xi}\rangle\langle\delta^+_{-1-i\lambda\xi},\psi\rangle d\lambda+\langle P_{\lambda^*}(\xi)\phi,\psi\rangle,
	\end{equation}
	where $\gamma\in(0,\e)$ is chosen so that for all $\lambda$ such that $-\gamma\xi<(\sign\xi)\Im\lambda<0$ the operators $K_\pm(-1-i\lambda\xi,\xi)$ are invertible, and $P_{\lambda^*}(\xi)$ is the following operator
	\begin{align*}
	\begin{split}
	P_{\lambda^*}(\xi)&=-\Res_{\lambda=\lambda^*(\xi)}R(\lambda,\xi)\,\,\hbox{for}\,\,\xi\in(-\sqrt{\pi},0)\cup(0,\sqrt{\pi}),\\
	P_{-1^\pm}&=-\Res_{\lambda=-1}R^+(-1,\pm\sqrt{\pi}),\\
	P_{\lambda^*}(\xi)&=0, \,\,\hbox{for}\,\,\xi\notin[-\sqrt{\pi},\sqrt{\pi}],
	\end{split}
	\end{align*} where   $\lambda^*(\xi)\in[-1,0]$ such that $0\in\sigma_d(K_+(-1-i\lambda^*(\xi)\xi, \xi))$. In particular,

	\begin{itemize}
		\item if $\xi\in(-\sqrt{\pi},0)\cup(0,\sqrt{\pi})$, then for any $f,g \in H={L_w^2(\bbR)}$ the following generalized eigenfunction expansion holds
		\begin{equation}\lb{ext1}
		(f,g)_{L_w^2}=\int_{\mathbb{R}}\langle f,\delta^{a-}_{-1+i\lambda\xi}\rangle\langle\delta^+_{-1-i\lambda\xi},g\rangle d\lambda+(P_{\lambda^*}(\xi)f,g)_{L_w^2},
		\end{equation}
		where $\lambda^*(\xi)\in(-1,0)$ is an isolated eigenvalue of $L_\xi$ and $P_{\lambda^*}(\xi)$ is the Riesz projection corresponding to the simple eigenvalue $\lambda^*(\xi)$ of $L_{\xi}$.
		\item if $\xi=\pm\sqrt{\pi}$, then for any $f, g\in H^1_{w}(\bbR)$
		\begin{equation*}
		(f,g)_{L_w^2}=\int_{\mp}\langle f,\delta^{a-}_{-1+i\bar\lambda\xi}\rangle\langle\delta^+_{-1-i\lambda\xi},g\rangle d\lambda+\langle P_{-1^\pm}f,g\rangle.
		\end{equation*}
			\item if $\xi\notin[-\sqrt{\pi},\sqrt{\pi}]$, then for any $f,g \in H={L_w^2(\bbR)}$
		\begin{equation}\lb{ext4}
		(f,g)_{L_w^2}=\int_{\mathbb{R}}\langle f,\delta^{a-}_{-1+i\lambda\xi}\rangle\langle\delta^+_{-1-i\lambda\xi},g\rangle d\lambda.
		\end{equation}
	\end{itemize}

\end{theorem}
\begin{proof}
	We know (see \eqref{Rsmallo}) that for $\phi, \psi\in\Phi$
	\begin{equation*}
	\langle R_\pm(\lambda, \xi)\phi,\psi\rangle=o(1),\,\,\,|\lambda|\to\infty
	\end{equation*}
	uniformly in the region $\{\lambda\in\bbC:\Re\lambda+1\geq-\eta|\xi| \}$ in the case of $R_+(\lambda, \xi)$ and uniformly in the region $\{\lambda\in\bbC:\Re\lambda+1\leq\eta|\xi| \}$ in the case of $R_-(\lambda, \xi)$ for any $\eta<\varepsilon$.\\
	Let $\lambda$ be in the resolvent set of $L_{\xi}$. Then
	\begin{equation}\lb{integration}
	\frac{1}{\lambda}(\phi,\psi)_{L_w^2}=-(R(\lambda,\xi)\phi,\psi)_{L_w^2}+\frac{1}{\lambda}(R(\lambda,\xi)L_{\xi}\phi,\psi)_{L_w^2},\,\, \phi\in\dom(L_\xi|\Phi),\,\,\psi\in\Phi.
	\end{equation}
	Notice that for $\phi, \psi\in\Phi$
	\begin{equation*}
	(R(\lambda, \xi)\phi,\psi)=\langle R_\pm(\lambda, \xi)\phi,\psi\rangle=o(1),\,\,\,|\lambda|\to\infty
	\end{equation*}
	uniformly in the region $\{\lambda\in\bbC:\Re\lambda+1\geq N \}$ in the case of $R_+(\lambda, \xi)$ and uniformly in the region $\{\lambda\in\bbC:\Re\lambda+1\leq-N \}$ in the case of $R_-(\lambda, \xi)$ for any $N>0$. Hence,
	\begin{equation*}
	\lim_{N\to\infty}\int_{N}\frac{1}{\lambda}(R(\lambda,\xi)L_{\xi}\phi,\psi)_{L_w^2}d\lambda=0,
	\end{equation*}
	where $\int_{N}:=\int_{-\infty i-1+N}^{\infty i-1+N}-\int_{-\infty i-1-N}^{\infty i-1-N}$. Also,
	\begin{equation*}
	\lim_{N\to\infty}\int_{N}\frac{1}{\lambda}(\phi,\psi)_{L_w^2}d\lambda=\lim_{N\to\infty}\oint_{|\lambda|=N}\frac{1}{\lambda}(\phi,\psi)_{L_w^2}d\lambda=2\pi i(\phi,\psi)_{L_w^2}.
	\end{equation*}
	Therefore,
	\begin{equation}\lb{intergration1}
	(\phi,\psi)_{L_w^2}=-\frac{1}{2\pi i}\lim_{N\to\infty}\int_{N}(R(\lambda,\xi)\phi,\psi)_{L_w^2}d\lambda.
	\end{equation}

	Or,
	\begin{align*}
	\begin{split}
	(\phi,\psi)_{L_w^2}&=-\frac{1}{2\pi i }\Big(\int_{-\infty i-1+\gamma}^{\infty i-1+\gamma}-\int_{-\infty i-1-\gamma}^{\infty i-1-\gamma}\Big)(R(\lambda,\xi)\phi,\psi)_{L_w^2}d\lambda\\
	&-\Res_{|\Re\lambda+1|>\gamma}(R(\lambda,\xi)\phi,\psi)_{L_w^2},\,\,\,\,\,\gamma>0.
	\end{split}
	\end{align*}
	Next, we choose $\gamma$ as indicated in the statement of the theorem. Then
	\begin{align*}
	\begin{split}
	&(\phi,\psi)_{L_w^2}=\frac{1}{2\pi i }\int_{-\infty i-1-\gamma}^{\infty i-1-\gamma}((R^{-}(\lambda,\xi)-(R^{+}(\lambda,\xi))\phi,\psi)_{L_w^2}d\lambda\\
	&-\Res_{|\Re\lambda+1|=0}(R^+(\lambda,\xi)\phi,\psi)_{L_w^2}-\Res_{|\Re\lambda+1|\neq0}(R(\lambda,\xi)\phi,\psi)_{L_w^2}.
	\end{split}
	\end{align*}
	
	After the change of variables $\lambda\to-1-i\lambda\xi$, we arrive at
		\begin{align*}
		\begin{split}
		&(\phi,\psi)_{L_w^2}=-\frac{\xi}{2\pi }\int_{-(\sign\xi)\infty-i\frac{\gamma}{\xi}}^{(\sign\xi)\infty-i\frac{\gamma}{\xi}}((R^+(-1-i\lambda\xi,\xi)-(R^-(-1-i\lambda\xi,\xi))\phi,\psi)_{L_w^2}d\lambda\\
		&-\Res_{|\Re\lambda+1|=0}(R^+(\lambda,\xi)\phi,\psi)_{L_w^2}-\Res_{|\Re\lambda+1|\neq0}(R(\lambda,\xi)\phi,\psi)_{L_w^2}.
		\end{split}
	\end{align*}
	By applying Proposition \ref{Jump}, we arrive at formula \eqref{gee}.\\
	Formulas  \eqref{ext1}-\eqref{ext4} follow from Propositions \ref{FT} and \ref{ext}.
\end{proof}

\begin{corollary}\label{xi=0}
	Let $\xi=0$. Then we can still use formula \eqref{integration} for $L_0$, and after integration as in \eqref{intergration1}, we arrive at
	\begin{equation*}
	(f,g)_{L_w^2}=((I-V)f,g)_{L_w^2}+(Vf,g)_{L_w^2},
	\end{equation*}
\end{corollary}

Also, note that the Riesz projection  $P_0=-\Res_{\lambda=0}R(0,0)$ corresponding to the simple eigenvalue $\lambda^*(0)=0$ of $L_{0}$ is equal to $V$. 

\begin{theorem}\label{MainThm} 
	Fix $\xi\neq0$. Then
	\begin{itemize}
		\item if $\xi\in(-\sqrt{\pi},0)\cup(0,\sqrt{\pi})$, then for any $f,g \in H={L_w^2(\bbR)}$ the following generalized eigenfunction expansion holds
		\begin{equation}\lb{Fext1}
		(f,g)_{L_w^2}=\int_{\mathbb{R}} (\mathcal{B}_\xi f)(\lambda)\overline{(\mathcal{U}_\xi g)(\lambda)}{w(\lambda)}d\lambda+(P_{\lambda^*}(\xi)f,g)_{L_w^2},
		\end{equation}
		where $\lambda^*(\xi)\in(-1,0)$ is an isolated eigenvalue of $L_\xi$ and $P_{\lambda^*}(\xi)$ is the Riesz projection corresponding to the simple eigenvalue $\lambda^*(\xi)$ of $L_{\xi}$.
		\item if $\xi=\pm\sqrt{\pi}$, then for any $f, g\in H^1_{w}(\bbR)$
		\begin{equation}
		(f,g)_{L_w^2}=\int_{\mp}(\mathcal{B}_\xi f)(\lambda)\overline{(\mathcal{U}_\xi g)(\lambda)}{w(\lambda)}d\lambda+\langle P_{-1^\pm}f,g\rangle.
		\end{equation}
		\item if $\xi\notin[-\sqrt{\pi},\sqrt{\pi}]$, then for any $f,g \in H={L_w^2(\bbR)}$
		\begin{equation}\lb{Fext4}
		(f,g)_{L_w^2}=\int_{\mathbb{R}}(\mathcal{B}_\xi f)(\lambda)\overline{(\mathcal{U}_\xi g)(\lambda)}{w(\lambda)}d\lambda.
		\end{equation}
	\end{itemize}
\end{theorem}

\begin{proof}
	Formulas  \eqref{Fext1}-\eqref{Fext4} follow directly from Propositions \ref{FT} and \ref{ext} and Theorem \ref{main_exp}.
\end{proof}

\begin{proposition}\label{UB_norm}
	Let $\xi\neq0$. Then
	\begin{align}\lb{unbound}
	\begin{split}
	&\big|\int_{\mathbb{R}}(\mathcal{B}_\xi f)(\lambda)\overline{(\mathcal{U}_\xi g)(\lambda)}{w(\lambda)}d\lambda\big|\leq C(\xi) \|f\|_{L^2_w}\|g\|_{L^2_w}\,\,\,\hbox{for}\,\,f,g\in L_w^2(\bbR)\,\hbox{and}\,\,\xi\neq0,\pm\sqrt{\pi},\\
	&\big|\int_{\mp}(\mathcal{B}_\xi f)(\lambda)\overline{(\mathcal{U}_\xi g)(\lambda)}{w(\lambda)}d\lambda\big|\leq C \|f\|_{H_w^1}\|g\|_{H_w^1}\,\,\,\hbox{for}\,\,f,g\in H_w^1(\bbR)\,\hbox{and}\,\,\xi\neq0,\\
	&\|\mathcal{U }_\xi g\|_{H^{-1}_{w}}\leq C_1\|g\|_{H^1_w},\,\,\,\|\mathcal{B}_\xi f\|_{L^2_{w}}\leq C_2\|f\|_{L^2_w}\,\, \hbox{for}\,\,g\in H_w^1(\bbR),\,f\in L_w^2(\bbR)\,\hbox{and}\,\,\xi\neq0.
	\end{split}
	\end{align}
	where the coefficients $C,C_1$ and $C_2$ are independent of $\xi$.
\end{proposition}
\begin{proof}
	The first inequality in \eqref{unbound} follows directly from Corollary \ref{Hs} and Remark \ref{omegareal}.
	Next, we will prove the second inequality in \eqref{unbound}. More specifically, we will prove the second inequality in \eqref{unbound} for $\xi>0$ and the integral $\int_-$ as the other cases could be handled similarly. According to Proposition \ref{FT}, for $\lambda\in\Omega$ and $f,g\in\Phi$ we have
	\begin{align}\lb{prodfor}
	\begin{split}
	(\mathcal{B}_\xi f)(\lambda)\overline{(\mathcal{U}_\xi g)(\lambda)}=&\Big(f(\lambda)-\frac{1}{-i\xi\overline{\omega_-(-1+i\bar\lambda\xi,-\xi)}}({\mathcal{S}}_+(f w^{1/2}))(\lambda)\Big)\\
	\times&\Big(\overline{g(\bar\lambda)}-\frac{1}{-i\xi{\omega_+(-1-i\lambda\xi,\xi)}}({\mathcal{S}}_+(\tilde g w^{1/2}))(\lambda)\Big),
	\end{split}
	\end{align}
	where $\tilde g(z):=\overline{g(\bar z)}$.
	
	Let $\lambda\in\bbR$. According to formula \eqref{detpm}, we have
	\begin{align*}
	\begin{split}
	{-i\xi{\omega_+(-1-i\lambda\xi,\xi)}}&=-i\xi{\det(K_+(-1-i\lambda\xi,\xi))}=-i\xi\big(1-{\int_\mathbb{\gamma_-}\frac{w(z)dz}{zi\xi-i\lambda\xi}}\big)\\
	&=-i\xi\big(1-\frac{1}{i\xi}(\pi iw(\lambda)+P.v.\int_\bbR\frac{w(v)dv}{v-\lambda})\big)\\
	&=P.v.\int_\bbR\frac{w(v)dv}{v-\lambda}-i(\xi-\pi w(\lambda))=-2D(\lambda)-i(\xi-\sqrt\pi e^{-\lambda^2}),
	\end{split}
	\end{align*}
	where $D(\lambda)=e^{-\lambda^2}\int_0^\lambda e^{t^2}dt$ is the Dawson function.
	Similarly, 
	\begin{align*}
	\begin{split}
	-i\xi\overline{\omega_-(-1+i\lambda\xi,-\xi)}&=-i\xi\overline{\big(1-{\int_\mathbb{\gamma_+}\frac{w(z)dz}{-zi\xi+i\lambda\xi}}\big)}=-i\xi{\big(1-\frac{1}{i\xi}{\int_\mathbb{\gamma_+}\frac{w(z)dz}{z-\lambda}}\big)}\\
	&=-i\xi\big(1-\frac{1}{i\xi}(-\pi iw(\lambda)+P.v.\int_\bbR\frac{w(v)dv}{v-\lambda})\big)\\
	&=P.v.\int_\bbR\frac{w(v)dv}{v-\lambda}-i(\xi+\pi w(\lambda))=-2D(\lambda)-i(\xi+\sqrt\pi e^{-\lambda^2}).
	\end{split}
	\end{align*}
	Therefore, the singularities in $(\mathcal{B}_\xi f)(\lambda)\overline{(\mathcal{U}_\xi g)(\lambda)}$ come only from two terms $\frac{1}{{-i\xi{\omega_+(-1-i\lambda\xi,\xi)}}}$ and $\frac{1}{-i\xi\overline{\omega_-(-1+i\lambda\xi,-\xi)}}$. Now, fix $\epsilon\in(0, \sqrt{\pi}/4)$.\\
	Case 1. Let $\xi\in(0,\epsilon]$. Since $\sqrt\pi e^{-\lambda^2}-\sqrt\pi=o(\lambda)$, there exists $\tilde\epsilon>0$ such that $|\sqrt\pi e^{-\lambda^2}-\sqrt\pi|\leq\sqrt{\pi}/4$ for all $\lambda$ such that $|\lambda|\leq\tilde\epsilon$. Therefore,
	\begin{align}\lb{est1}
	\begin{split}
	|{-i\xi{\omega_+(-1-i\lambda\xi,\xi)}}|&=|-2D(\lambda)-i(\xi-\sqrt\pi e^{-\lambda^2})|\\
	&\geq|\sqrt\pi e^{-\lambda^2}-\xi|=|\sqrt{\pi}-\xi+\sqrt\pi e^{-\lambda^2}-\sqrt{\pi}|\geq\sqrt{\pi}/2, \,|\lambda|\leq\tilde\epsilon.
	\end{split}
	\end{align}
	Next, notice that $-2\lim_{|\lambda|\to\infty}\lambda D(\lambda)=-1$. Hence, there exists $\tilde{\tilde{\epsilon}}>0$ such that for all $|\lambda|\geq\tilde{\tilde{\epsilon}}$	$|-2\lambda D(\lambda)|\geq\sqrt{\pi}/2$.  Therefore,
	\begin{align}\lb{est2}
	\begin{split}
	|{-i\lambda\xi{\omega_+(-1-i\lambda\xi,\xi)}}|&\geq|-2\lambda D(\lambda)|\geq\sqrt{\pi}/2, \,|\lambda|\geq\tilde{\tilde{\epsilon}}.
	\end{split}
	\end{align}
	Also, it is clear that for $\lambda$ such that $\tilde{\tilde{\epsilon}}\geq|\lambda|\geq\tilde{\epsilon}$, $|-2D(\lambda)|\geq \tilde C,\,\tilde C=\min\{|-2D(\tilde{\epsilon})|,|-2D(\tilde{\tilde{\epsilon}})|\}$. Therefore,
	\begin{align}\lb{est3}
	\begin{split}
	|{-i\xi{\omega_+(-1-i\lambda\xi,\xi)}}|&\geq|-2 D(\lambda)|\geq \tilde C, \,\tilde{\tilde{\epsilon}}\geq|\lambda|\geq\tilde{\epsilon}.
	\end{split}
	\end{align}
	Notice that $\tilde{\epsilon}, \tilde{\tilde{\epsilon}}$ and $\tilde C$ are all independent of $\xi$. Moreover, the same estimates also hold for $-i\xi\overline{\omega_-(-1+i\lambda\xi,-\xi)}$ with the same constants $\tilde{\epsilon}, \tilde{\tilde{\epsilon}}$ and $\tilde C$. Therefore,
	\begin{align}\lb{2integrals}
	\begin{split}
	\big|\int_{-}(\mathcal{B}_\xi f)(\lambda)\overline{(\mathcal{U}_\xi g)(\lambda)}{w(\lambda)}d\lambda\big|&=\big|\int_{\bbR}(\mathcal{B}_\xi f)(\lambda)\overline{(\mathcal{U}_\xi g)(\lambda)}{w(\lambda)}d\lambda\big|\\
	&=\big|\int_{|\lambda|\leq\tilde{\tilde{\epsilon}}}(\mathcal{B}_\xi f)(\lambda)\overline{(\mathcal{U}_\xi g)(\lambda)}{w(\lambda)}d\lambda+\int_{|\lambda|\geq\tilde{\tilde{\epsilon}}}(\mathcal{B}_\xi f)(\lambda)\overline{(\mathcal{U}_\xi g)(\lambda)}{w(\lambda)}d\lambda\big|
	\end{split}
	\end{align}
	By \eqref{est1} and \eqref{est3}, $\frac{1}{|{-i\lambda\xi{\omega_+(-1-i\lambda\xi,\xi)}}|}\leq \tilde{\tilde C}=\max\{2/\sqrt{\pi},1/\tilde C\}$ for $|\lambda|\leq\tilde{\tilde{\epsilon}}$\\ (similarly, $\frac{1}{|-i\xi\overline{\omega_-(-1+i\lambda\xi,-\xi)}|}\leq \tilde{\tilde C}=\max\{2/\sqrt{\pi},1/\tilde C\}$ for $|\lambda|\leq\tilde{\tilde{\epsilon}}$). Therefore, we have the following inequality: 
	\begin{align*}
	\begin{split}
	\big|\int_{|\lambda|\leq\tilde{\tilde{\epsilon}}}(\mathcal{B}_\xi f)(\lambda)\overline{(\mathcal{U}_\xi g)(\lambda)}{w(\lambda)}d\lambda\big|\leq C_1 \|f\|_{L_w^2}\|g\|_{L_w^2}\leq C_1 \|f\|_{H_w^1}\|g\|_{H_w^1},
	\end{split}
	\end{align*}
	where $C_1$ is independent of $\xi$.
	For the second integral in \eqref{2integrals}, we have 
	\begin{align*}
	\begin{split}
	&\big|\int_{|\lambda|\geq\tilde{\tilde{\epsilon}}}(\mathcal{B}_\xi f)(\lambda)\overline{(\mathcal{U}_\xi g)(\lambda)}{w(\lambda)}d\lambda\big|=|\int_{|\lambda|\geq\tilde{\tilde{\epsilon}}}\Big(f(\lambda)-\frac{1}{-i\lambda\xi\overline{\omega_-(-1+i\bar\lambda\xi,-\xi)}}\lambda({\mathcal{S}}_+(f w^{1/2}))(\lambda)\Big)\\
	\times&\Big(\overline{g(\bar\lambda)}-\frac{1}{-i\lambda\xi{\omega_+(-1-i\lambda\xi,\xi)}}\lambda({\mathcal{S}}_+(\tilde g w^{1/2}))(\lambda)\Big)w(\lambda)d\lambda\big|.
	\end{split}
	\end{align*}
	Then, by \eqref{est2} and the fact that $L_w^2(\bbR)$-norm of  $\lambda({\mathcal{S}}_+(f w^{1/2}))(\lambda)$ is bounded above by $L_w^2(\bbR)$-norm of $f(\lambda)$ (similarly, $L_w^2(\bbR)$-norm of  $\lambda({\mathcal{S}}_+(\tilde g w^{1/2}))(\lambda)$ is bounded above by $L_w^2(\bbR)$-norm of $g(\lambda)$), we have 
	\begin{align*}
	\big|\int_{|\lambda|\geq\tilde{\tilde{\epsilon}}}(\mathcal{B}_\xi f)(\lambda)\overline{(\mathcal{U}_\xi g)(\lambda)}{w(\lambda)}d\lambda\lambda\big|\leq C_2 \|f\|_{L_w^2}\|g\|_{L_w^2}\leq C_2 \|f\|_{H_w^1(\bbR)}\|g\|_{H_w^1(\bbR)},
	\end{align*}
	where $C_2$ is independent of $\xi$.
	
	Case 2. Let $\xi\in[\epsilon, 2\sqrt{\pi}-\epsilon]$. 	Then,
	\begin{align}\lb{unest}
	\begin{split}
		|-i\xi\overline{\omega_-(-1+i\lambda\xi,-\xi)}|=|-2D(\lambda)-i(\xi+\sqrt\pi e^{-\lambda^2})|\geq|\xi|\geq\epsilon.
	\end{split}
\end{align}
	Moreover, since $\lim_{|\lambda|\to\infty}e^{-\lambda^2}=0$, there exists ${\tilde{\epsilon}}>0$ such that for all $|\lambda|\geq{\tilde{\epsilon}}$	$e^{-\lambda^2}\leq\frac{\epsilon}{2\sqrt{\pi}}$.  Therefore,
	\begin{align}\lb{infboun}
	\begin{split}
	|{-i\xi{\omega_+(-1-i\lambda\xi,\xi)}}|&\geq|\xi-\pi w(\lambda)|\geq\epsilon/2, \,|\lambda|\geq{\tilde{\epsilon}},
	\end{split}
	\end{align}
	where $\tilde\epsilon$ is independent of $\xi$. Also, we have the following representation:
	\begin{align}\lb{coeff}
	\begin{split}
	{-i\xi{\omega_+(-1-i\lambda\xi,\xi)}}&=-2D(\lambda)-i(\xi-\sqrt\pi e^{-\lambda^2})=-2D'(0)\lambda+o(\lambda)-i(\xi-\sqrt{\pi}+o(\lambda))\\
	&=-2D'(0)(\lambda-\frac{\sqrt{\pi}-\xi}{2D'(0)}i)+o(\lambda)=-2D'(0)(\lambda-\eta i)+o(\lambda),
	\end{split}
	\end{align}
	where $D(\cdot)$ is the Dawson function and $\eta=\frac{\sqrt{\pi}-\xi}{2D'(0)}\in\bbR$. Therefore, $\frac{(\lambda-\eta i)}{-2D'(0)(\lambda-\eta i)+o(\lambda)}$ is uniformly bounded over $(\lambda, \eta)\in[-\epsilon,\epsilon]\times[-\mu,\mu]$.
	Then, the second inequality in \eqref{unbound} follows from formulas \eqref{prodfor}, \eqref{unest}, \eqref{infboun}, \eqref{coeff} and Proposition \ref{ext}.\\
	Case 3. Let $\xi\in[2\sqrt{\pi}-\epsilon, \infty)$. Then, we have the following uniform estimates:
	\begin{align*}
	\begin{split}
	|-i\xi\overline{\omega_-(-1+i\lambda\xi,-\xi)}|&=|-2D(\lambda)-i(\xi+\sqrt\pi e^{-\lambda^2})|\geq|\xi|\geq2\sqrt{\pi}-\epsilon,\\
	|{-i\xi{\omega_+(-1-i\lambda\xi,\xi)}}|&=|-2D(\lambda)-i(\xi-\sqrt\pi e^{-\lambda^2})|\geq\sqrt{\pi}-\epsilon,
	\end{split}
	\end{align*}
which imply
\begin{align*}
\begin{split}
\big|\int_-(\mathcal{B}_\xi f)(\lambda)\overline{(\mathcal{U}_\xi g)(\lambda)}{w(\lambda)}d\lambda\big|\leq C_1 \|f\|_{L_w^2}\|g\|_{L_w^2}\leq C_1 \|f\|_{H_w^1}\|g\|_{H_w^1},
\end{split}
\end{align*}
where $C_1$ is independent of $\xi$.\\

Finally, the inequalities from the third line of \eqref{unbound} follow from the proof of the second inequality in \eqref{unbound}.
\end{proof}

\begin{proposition}\label{projs} Let $\xi\in\mathbb{R}$. Then
	\begin{enumerate}
			\item If $\xi=\pm\sqrt{\pi}$, then
		\begin{align}\lb{projpm}
		\begin{split}
		P_{-1^\pm}&=\frac{1}{\omega'_+(-1,\pm\sqrt{\pi})}\langle\cdot,\bar e_1\rangle e_1, \hbox{where}\, \, e_1=\pm\frac{1}{\sqrt{\pi}iv}, \bar e_1=\mp\frac{1}{\sqrt{\pi}iv}\in H^{-1}_w(\bbR).
		\end{split}
		\end{align}
		Here $\langle f,\bar e_1\rangle:=\int_{\pm}f(v)\bar e_1(v)w(v)dv$ for $f\in H^1_w(\bbR)$ (see \eqref{H1}, \eqref{defpm}, \eqref{H1ext}).
		\item If $\xi\in(-\sqrt{\pi},\sqrt{\pi})$, then
		\begin{align}
		\begin{split}
		P_{\lambda^*}(\xi)&=\frac{1}{(e_1,\bar e_1)_{L_w^2}}(\cdot,\bar e_1)_{L_w^2}e_1, \\
		&=\frac{1}{\omega'(\lambda^*(\xi),\xi)}(\cdot,\bar e_1)_{L_w^2}e_1, \hbox{where}\, \, e_1=\frac{1}{vi\xi+1+\lambda^*(\xi)},\bar e_1\in L^2_w(\bbR).
		\end{split}
		\end{align}
	
		\item if $\xi\notin[-\sqrt{\pi},\sqrt{\pi}]$, then
		\begin{align}
		P_{\lambda^*}(\xi)&=0.
		\end{align}
	\end{enumerate}
	
\end{proposition}
\begin{proof}

(1) Applying formulas \eqref{perR} and \eqref{inverse}, we arrive at
	\begin{align*}
\begin{split}
P_{-1^\pm}&=-\Res_{\lambda=-1}R^+(\lambda,\pm\sqrt{\pi})= \Res_{\lambda=-1} [R_+^0(\lambda,\pm\sqrt{\pi})VK_+^{-1}(\lambda,\pm\sqrt{\pi})VR_+^0(\lambda,\pm\sqrt{\pi})]\\
&= -\Res_{\lambda=-1} \Big[\frac{1}{\omega_+(\lambda,\pm\sqrt{\pi})} R_+^0(\lambda,\pm\sqrt{\pi})VR_+^0(\lambda,\pm\sqrt{\pi})VR_+^0(\lambda,\pm\sqrt{\pi})\Big]\\
&=-\frac{1}{\omega'_+(-1,\pm\sqrt{\pi})}R_+^0(-1,\pm\sqrt{\pi})VR_+^0(-1,\pm\sqrt{\pi})VR_+^0(-1,\pm\sqrt{\pi})\\
&=-\frac{1}{\omega'_+(-1,\pm\sqrt{\pi})}\langle R_+^0(-1,\pm\sqrt{\pi})\cdot,\mathbb{1}\rangle \langle R_+^0(-1,\pm\sqrt{\pi})\mathbb{1},\mathbb{1}\rangle R_+^0(-1,\pm\sqrt{\pi})\mathbb{1}\\
&=\frac{1}{\omega'_+(-1,\pm\sqrt{\pi})}\langle R_+^0(-1,\pm\sqrt{\pi})\cdot,\mathbb{1}\rangle  R_+^0(-1,\pm\sqrt{\pi})\mathbb{1},
\end{split}
\end{align*} 
 where we used the fact that $\langle R_+^0(-1,\pm\sqrt{\pi})\mathbb{1},\mathbb{1}\rangle=-1$. Formula \eqref{projpm} from the definition fo $R_+^0(-1,\pm\sqrt{\pi})$,  \eqref{H1} and \eqref{H1ext}.

(2)	It follows from Theorem \ref{eigen} that for each $\xi\in(-\sqrt{\pi},\sqrt{\pi})$ the eigenfunction of $L_\xi$ corresponding to the eigenvalue $\lambda^*(\xi)$ is $e_1(v)=\frac{1}{vi\xi+1+\lambda^*(\xi)}$.
We also know that $\ker(P_{\lambda^*}(\xi))=\ran((P_{\lambda^*}(\xi))^*)^\perp$. Next, notice that $(P_{\lambda^*}(\xi))^*$ is the Riesz projection corresponding to the isolated eigenvalue $\lambda^*(\xi)$ of the operator $L^*_\xi$. Hence, $\ran((P_{\lambda^*}(\xi))^*)=span\{\bar e_1\}$, where $\bar e_1$ is the eigenfunction of $L^*_\xi$ corresponding to the eigenvalue $\lambda^*(\xi)$. Therefore, $P_{\lambda^*}(\xi)=\alpha(\cdot,\bar e_1)_{L_w^2}e_1$. And since $P_{\lambda^*}(\xi)$ is the projection onto $span\{ e_1\}$, $\alpha$ must be $\frac{1}{(e_1,\bar e_1)_{L_w^2}}$.

	On the other hand, we can also compute the negative residue of the resolvent for $\xi\in(-\sqrt{\pi},\sqrt{\pi})$ as in (1), that is,
\begin{align}\lb{projall}
\begin{split}
P_{\lambda^*}(\xi)&=\frac{1}{\omega'(\lambda^*(\xi),\xi)}( R^0(\lambda^*(\xi),\xi)\cdot,\mathbb{1})_{L_w^2}  R^0(\lambda^*(\xi),\xi)\mathbb{1}.\\
\end{split}
\end{align}
In this case, the functional $\langle R_+^0(\lambda^*(\xi),\xi)\cdot,\mathbb{1}\rangle$ can be represented in terms of the inner product in ${L_w^2(\bbR)} $, that is, $\langle R_+^0(\lambda^*(\xi),\xi)\cdot,\mathbb{1}\rangle=( R^0(\lambda^*(\xi),\xi)\cdot,\mathbb{1})_{L_w^2} $. Moreover,
\begin{align*}
\begin{split}
(e_1,\bar e_1)_{L_w^2}&=\omega'(\lambda^*(\xi),\xi),\\
-( R^0(\lambda^*(\xi),\xi)\cdot,\mathbb{1})_{L_w^2}&=-( \cdot,(R^0(\lambda^*(\xi),\xi))^*\mathbb{1})_{L_w^2}=(\cdot,\bar e_1)_{L_w^2},\\
-R^0(\lambda^*(\xi),\xi)\mathbb{1}&=e_1.
\end{split}
\end{align*}
\end{proof}

\begin{corollary}\label{P_norm}
	Let $\xi\in\mathbb{R}$. Then the operators $P_{\lambda^*}(\xi)$, $\xi\in(-\sqrt{\pi}, \sqrt{\pi})$, and $P_{-1^\pm}$ are of finite rank into $L^2_w(\bbR)$ and $H^{-1}_w(\bbR)$, respectively, with
	\begin{align}\label{projest}
	\begin{split}
	|(P_{\lambda^*}(\xi)f,g)_{L_w^2}|&\leq \tilde C(\xi)\|f\|_{L_w^2}\|g\|_{L_w^2}\,\,\,\hbox{for}\,\,f,g\in L_w^2(\bbR)\,\hbox{and}\,\,\xi\neq\pm\sqrt{\pi},\\
	|\langle P_{\lambda^*}(\xi)f,g\rangle|&\leq \tilde C \|f\|_{H_w^1(\bbR)}\|g\|_{H_w^1(\bbR)}\,\,\,\hbox{for}\,\,f,g\in H_w^1(\bbR) \,\hbox{and all}\,\,\xi\in\bbR,
	\end{split}
	\end{align}
	where $\tilde C$ is independent of $\xi$.
\end{corollary}

\begin{proof}
	According to Proposition \ref{projs},
		\begin{align*}
	\begin{split}
	\langle P_{\lambda^*}(\xi)f,g\rangle&= \frac{1}{\omega'_+(\lambda^*(\xi),\xi)}\langle f,\bar e_1\rangle\langle e_1, g\rangle,
	\end{split}
	\end{align*}
		where for $\xi\neq\pm\sqrt\pi$ the pairing $\langle \cdot,\cdot\rangle$ can be interpreted as the $L^2_w(\bbR)$-inner product, and in this case, $e_1$ and $\bar e_1$ are the $L^2_w(\bbR)$ functions. Then the first line in \eqref{projest} is straightforward, and the second line in \eqref{projest}  follows from Proposition \ref{ext}.
\end{proof}

We introduce the following families of operators denoted by $P(\xi)\in B(H^1_w(\bbR,dv),H^{-1}_w(\bbR,dv))$ and $\mathcal{U}^*_\xi\mathcal{B}_\xi\in B(H^1_w(\bbR,dv),H^{-1}_w(\bbR,dv))$:
\begin{align}\label{projfor}
\begin{split}
P(\xi):=\left\{\begin{array}{l} 
P_{\lambda^*}(\xi),\,\,\xi\in(-\sqrt{\pi},\sqrt{\pi}),\\
P_{-1^\pm},\,\,\xi=\pm\sqrt{\pi},\\
0,\,\,\xi\notin[-\sqrt{\pi},\sqrt{\pi}],
\end{array} \right.
\end{split}
\end{align}
and,
\begin{align}\label{intfor}
\begin{split}
\langle\mathcal{U}^*_\xi\mathcal{B}_\xi f,g\rangle &:=\int_{-}(\mathcal{B}_\xi f)(\lambda)\overline{(\mathcal{U}_\xi g)(\lambda)}{w(\lambda)}d\lambda,\,\,\,f,g\in H^1_w(\bbR,dv),\,\,\xi>0,\\
\langle\mathcal{U}^*_\xi\mathcal{B}_\xi f,g\rangle &:=\int_{+}(\mathcal{B}_\xi f)(\lambda)\overline{(\mathcal{U}_\xi g)(\lambda)}{w(\lambda)}d\lambda,\,\,\,f,g\in H^1_w(\bbR,dv),\,\,\xi<0,\\
\langle\mathcal{U}^*_\xi\mathcal{B}_\xi f,g\rangle &:=((I-V)f,g)_{L_w^2},\,\,\,f,g\in H^1_w(\bbR,dv),\,\,\xi=0.
\end{split}
\end{align}
	where  $\mathcal{B}_\xi$ and $\mathcal{U}_\xi$ are defined in \eqref{transf}. We give an explicit description of  $\mathcal{B}^*_\xi$ and $\mathcal{U}^*_\xi$ in Appendix \ref{ubadjoint}.
By Proposition \ref{UB_norm} and Corollary \ref{P_norm}, we have 
\begin{align}\label{norm_est}
\begin{split}
&\|P(\xi)\|_{H^1_w\to H^{-1}_w}\leq \tilde C,\,\,\hbox{$\tilde C$ is $\xi$-independent},\\
&\|P(\xi)\|_{L^2_w\to L^2_w}\leq \tilde C(\xi)\,\,\hbox{for $\xi\neq\pm\sqrt{\pi}$},\\
&\|\mathcal{U}^*_\xi\mathcal{B}_\xi\|_{H^1_w\to H^{-1}_w}\leq  C,\,\,\hbox{$ C$ is $\xi$-independent},\\
&\|\mathcal{U}^*_\xi\mathcal{B}_\xi\|_{L^2_w\to L^2_w}\leq  C(\xi)\,\,\hbox{for $\xi\neq\pm\sqrt{\pi}$}.
\end{split}
\end{align}

\begin{theorem}[eigenfunction expansion]\label{main_expension_thm} Let $f\in L^2(\bbR, H^1_w)$ (that is, $f$ is a function of two variables $\xi$ and $v$, and it is an $L^2$-function with respect to $\xi$ and an $H^1_w$ with respect to $v$). Then the following eigenfunction expansion formula is valid:
		\begin{equation*}
		f=\mathcal{U}^*_\xi\mathcal{B}_\xi f+ P(\xi)f,
	\end{equation*}
     where the equality is understood in the weak sense (the $H^{-1}_w$-sense).
     
	Moreover, if $f\in L^2(\bbR, L^2_w)$ and is compactly supported on $(-\sqrt{\pi}, \sqrt{\pi})$ (that is, $f$ is compactly supported with respect to $\xi$), then the following eigenfunction expansion formula is valid:
	\begin{equation*}
	f=\mathcal{U}^*_\xi\mathcal{B}_\xi f+ P(\xi)f,
	\end{equation*}
	where the equality is understood in the strong sense (the $L^2_w$-sense).
\end{theorem}
\begin{proof}
	It directly follows from Corollary \ref{xi=0}, Theorem \ref{MainThm}, Proposition \ref{UB_norm}, Corollary \ref{P_norm} and formula \eqref{norm_est}.
\end{proof}

\section{Time evolution and convergence to Grossly Determined Solutions}\label{s:evol}

\subsection{Solution formula}\label{s:soln}

We first take the Fourier transform of equation \eqref{bgk} in the spatial variable, that is,
\begin{equation}\label{maineq}
	\frac{\partial \hat f}{\partial t}(t,\xi,v)=-vi\xi{\hat f}(t,\xi,v)-\hat f(t,\xi,v)+\int_{\bbR}w(r)\hat f(t,\xi,r)dr.
\end{equation}
Or,
\be\label{feq}
\partial_t\hat f(t,\xi,v)= (L_\xi \hat f)(t,\xi,v).
\ee
By Theorem \ref{main_expension_thm}, for any $\hat f\in L^2(\bbR, H^1_w)$ we have
\begin{equation*}
\hat f=\mathcal{U}^*_\xi\mathcal{B}_\xi \hat f+P(\xi)\hat f.
\end{equation*}
Then the solution of \eqref{feq} can be found in the form:
 \begin{equation}\label{solfor}
 \hat f(t,\xi,v)=\mathcal{U}^*_\xi \big((\mathcal{B}_\xi \hat f)(\xi,\lambda)\big)+P(\xi)\hat f.
 \end{equation}
 
\begin{theorem}\label{main1}
	Let $\hat f_0(\xi,v):=\hat f(0,\xi,v)$ represent the Fourier transform of the initial molecular density of the gas and assume that $\hat f_0\in L^2(\mathbb{R},H^1_{w}(\bbR))$. Then the Cauchy problem associated with \eqref{bgk} has a unique solution and its Fourier transform is
	\ba\label{solnform}
	\hat f(t,\xi,v)&=e^{-t}\mathcal{U}^*_\xi \big(e^{-i\xi\lambda t}(\mathcal{B}_\xi \hat f_0)(\xi,\lambda)\big)+e^{\lambda^*(\xi) t}P(\xi)\hat f_0\,\,\hbox{for}\,\,\xi\in[-\sqrt{\pi},\sqrt{\pi}],\\
	\hat f(t,\xi,v)&=e^{-t}\mathcal{U}^*_\xi \big(e^{-i\xi\lambda t}(\mathcal{B}_\xi \hat f_0)(\xi,\lambda)\big)\,\,\hbox{for}\,\,\xi\notin[-\sqrt{\pi},\sqrt{\pi}],
	\ea
where $\hat f(t,\cdot,\cdot)$ belongs to the space $ L^2(\mathbb{R}, H^{-1}_{w}(\bbR))$. 
	
	Moreover, if $\hat f_0\in L^2(\bbR, L^2_w)$ and is compactly supported on $(-\sqrt{\pi}, \sqrt{\pi})$ (that is, $\hat f_0$ is compactly supported with respect to $\xi$), then the Cauchy problem associated with \eqref{bgk} has a unique solution and its Fourier transform is
		\begin{align*}
	\begin{split}
	\hat f(t,\xi,v)&=e^{-t}\mathcal{U}^*_\xi \big(e^{-i\xi\lambda t}(\mathcal{B}_\xi \hat f_0)(\xi,\lambda)\big)+e^{\lambda^*(\xi) t}P(\xi)\hat f_0\,\,\hbox{for}\,\,\xi\in[-\sqrt{\pi},\sqrt{\pi}],\\
	\hat f(t,\xi,v)&=e^{-t}\mathcal{U}^*_\xi \big(e^{-i\xi\lambda t}(\mathcal{B}_\xi \hat f_0)(\xi,\lambda)\big)\,\,\hbox{for}\,\,\xi\notin[-\sqrt{\pi},\sqrt{\pi}],
	\end{split}
	\end{align*}
	where $\hat f(t,\cdot,\cdot)$ belongs to the space $ L^2(\mathbb{R}, L^{2}_{w}(\bbR))$.
\end{theorem}
\begin{proof}
	By the uniqueness of the eigenfunction expansion from Theorem \ref{main_expension_thm}  if we insert \eqref{solfor} into \eqref{feq} and apply Proposition \ref{L_action}, we can see that the original Cauchy problem is equivalent to two Cauchy problems:
	\begin{align}\label{finaleq}
	\begin{split}
	\partial_t{(\mathcal{B}_\xi \hat f)(t,\xi,\lambda)}&=(-1-i\xi\lambda)(\mathcal{B}_\xi \hat f)(t,\xi,\lambda),\,\,\,(\mathcal{B}_\xi \hat f)(0,\xi,\lambda)=(\mathcal{B}_\xi \hat f_0)(\xi,\lambda),\\
	\partial_t(P(\xi) \hat f)(t,\xi,v)&=\lambda^*(\xi)(P(\xi) \hat f)(t,\xi,v)\,\,\hbox{for}\,\,\xi\in[-\sqrt{\pi},\sqrt{\pi}],\\
	P(\xi) \hat f&=0\,\,\hbox{for}\,\,\xi\notin[-\sqrt{\pi},\sqrt{\pi}],\\
(P(\xi) \hat f)(0,\xi,v)&=(P(\xi) \hat f_0)(\xi,v).
	\end{split}
	\end{align} 
	The results follows from \eqref{finaleq}.
\end{proof}

\subsection{Moments of projectors}\label{s:fullrankmoment}
For use in what follows, we compute, finally, the actions
of $P(\xi)$ and its dual on the unit vector $\mathbb{1}$.

\begin{proposition}\label{rankprop}Let $\xi\in\bbR$. 
	\begin{itemize}
			\item if $\xi=\pm\sqrt{\pi}$. Then
		\begin{align*}
		\begin{split}
		P(\pm\sqrt{\pi})\mathbb{1}&= \frac{1}{\omega'_+(-1,\pm\sqrt{\pi})} e_1\neq0, \hbox{where}\, \, e_1=\pm\frac{1}{\sqrt{\pi}iv} \in H^{-1}_w(\bbR),\\
		P(\pm\sqrt{\pi})^*\mathbb{1}&= \frac{1}{\overline{\omega'_+(-1,\pm\sqrt{\pi})}} \bar e_1\neq0, \hbox{where}\, \, \bar e_1=\mp\frac{1}{\sqrt{\pi}iv} \in H^{-1}_w(\bbR).
		\end{split}
		\end{align*}
		\item 	If $\xi\in(-\sqrt{\pi},\sqrt{\pi})$. Then
		\begin{align*}
		\begin{split}
		P(\xi)\mathbb{1}&=\frac{1}{(e_1,\bar e_1)_{L_w^2}}e_1\neq0,\,\hbox{where}\,\,e_1=\frac{1}{vi\xi+1+\lambda^*(\xi)}  \in L^{2}_w(\bbR).\\
		P(\xi)^*\mathbb{1}&=\frac{1}{(\bar e_1, e_1)_{L_w^2}}\bar e_1\neq0,\,\hbox{where}\,\,\bar e_1=\frac{1}{-vi\xi+1+\lambda^*(\xi)}  \in L^{2}_w(\bbR).
		\end{split}
		\end{align*}
	
		\item if $\xi\notin[-\sqrt{\pi},\sqrt{\pi}]$. Then $P(\xi)\mathbb{1}=0$ and $P(\xi)^*\mathbb{1}=0$.
	\end{itemize}
\end{proposition}

\begin{proof}
	It follows from Proposition \ref{projs} and the fact that
	\begin{equation}\label{momfact}
	\overline{\langle e_1, \mathbb{1}\rangle}=\langle\mathbb{1},\bar e_1\rangle=-\langle R_+^0(\lambda^*(\xi),\xi)\mathbb{1},\mathbb{1}\rangle=1-\omega(\lambda^*(\xi),\xi)=1.
	\end{equation}
\end{proof}

\subsection{Decay to Grossly Determined Solutions}\label{s:decay}

Our goal in this section is to show that the class of general solutions
decay asymptotically to the subclass of grossly determined solutions.
\begin{theorem}\label{main2}
	Let $f$ be the general solution from Theorem \ref{main1} with the initial molecular density $f_0$ such that $\hat f_0\in L^2(\mathbb{R},H^1_{w}(\bbR))$ and let $g(t,x,v):=\mathfrak{F}^{-1}(P(\xi)\hat f)$, where $\mathfrak{F}^{-1}$ represents the inverse Fourier transform map with respect to the variable $\xi$. Then 
	\begin{align}\label{gdsfull}
	\begin{split}
	\hat g(t,\xi,v)&= {\hat \mu(t,\xi)} e_1={e^{\lambda^*(\xi)t} \hat \mu_0(\xi)} e_1\,\,\hbox{for}\,\,\xi\in[-\sqrt{\pi},\sqrt{\pi}],\\
	\hat g(t,\xi,v)&=0={\hat \mu(t,\xi)}\,\,\hbox{for}\,\,\xi\notin[-\sqrt{\pi},\sqrt{\pi}].
	\end{split}
	\end{align}
	where $\hat\mu(t,\xi):=\langle \hat g(t,\xi,\cdot),\mathbb{1}\rangle$, $\hat\mu_0(\xi):=\langle \hat g_0(\xi,\cdot),\mathbb{1}\rangle$ and $\hat g_0(\xi,v):=\hat g(0,\xi,v)$. In particular, $g$ is a grossly determined solution of equation \eqref{bgk}, i.e., its evolution is determined entirely by its moment function
	\begin{equation*}
	\mu(t,x):=\langle g(t,x,v),\mathbb{1}\rangle.
	\end{equation*}

\end{theorem}
\begin{proof}
Evidently, by \eqref{finaleq},  $g$ is the solution of equation \eqref{bgk}, or equivalently, $\hat g$ is the solution of equation \eqref{maineq} satisfying the the evolution equation:
\begin{equation*}
\partial_t \hat g=\lambda^*(\xi)\hat g
\end{equation*}
giving
\be\label{gsoln}
\hat g(\xi,v, t)= e^{\lambda^*(\xi)t} \hat g_0(\xi,v).
\ee
From \eqref{gsoln} it follows immediately that
\begin{equation*}
\hat \mu(\xi, t)= e^{\lambda^*(\xi)t} \hat \mu_0(\xi),
\end{equation*}
giving a self-contained evolution of the moment $\mu$.

It remains only to be seen that $\hat g$ may be recovered aftwerward from $\hat \mu$. Indeed, Proposition \ref{rankprop}  and formula \eqref{projfor},
\begin{align*}
\begin{split}
\hat\mu_0(\xi)&=\langle \hat g_0(\xi,\cdot),\mathbb{1}\rangle=\langle P(\xi)\hat f_0, \mathbb{1}\rangle=\langle \hat f_0, P(\xi)^*\mathbb{1}\rangle=
\frac{1}{\omega_+'(\lambda^*(\xi),\xi)}\langle \hat f_0, \bar e_1\rangle\,\,\hbox{for}\,\,\xi\in[-\sqrt{\pi},\sqrt{\pi}].\\
\hat\mu_0(\xi)&=\langle \hat g_0(\xi,\cdot),\mathbb{1}\rangle=\langle P(\xi)\hat f_0, \mathbb{1}\rangle=0\,\,\hbox{for}\,\,\xi\notin[-\sqrt{\pi},\sqrt{\pi}].
\end{split}
\end{align*}
	
Thus, by Proposition \ref{projs} and formula \eqref{projfor},
	\ba\label{muform}
\hat g_0(\xi,v)&=\hat g(0,\xi,v)= P(\xi)\hat f_0=
\frac{1}{\omega_+'(\lambda^*(\xi),\xi)}\langle \hat f_0,\bar e_1\rangle  e_1=\hat\mu_0(\xi)e_1 \,\,\hbox{for}\,\,\xi\in[-\sqrt{\pi},\sqrt{\pi}],\\
\hat g_0(\xi,v)&=\hat g(0,\xi,v)= P(\xi)\hat f_0=0\,\,\hbox{for}\,\,\xi\notin[-\sqrt{\pi},\sqrt{\pi}].
\ea
\end{proof}

\begin{remark}\label{recoveryrmk}
Evidently, what is needed to recover $\hat g=P(\xi)\hat f$ from 
$\hat \mu= \langle \mathbb{1}, \hat g\rangle$ is that the
	functional $\langle \mathbb{1}, \cdot\rangle$ be of full rank on the range of $P(\xi)$; equivalently, $P(\xi)^*\mathbb{1}\neq 0$, or $\langle \mathbb{1}, e_1\rangle\neq 0$, as shown in \eqref{momfact} and Proposition 8.2.
	This gives the general formula
$$ 
	\hat g= e_1 \hat \mu/ \langle \mathbb{1}, e_1\rangle  ,
$$
	of which the formula $\hat g= e_1 \hat \mu$ of \eqref{muform} is a special case following from
	the fact \eqref{momfact} that $\langle \mathbb{1}, e_1\rangle=1$.
\end{remark}

\begin{remark}
Note that for $\xi\in(-\sqrt{\pi},\sqrt{\pi})$ the Fourier transform
\eqref{gdsfull} of our grossly determined solution $g$ 
agrees with that of the
grossly determined solution of \cite[Theorem 11]{C17}.
\end{remark}

\begin{theorem}\label{main3}
	Let $\hat f_0(\xi,v):=\hat f(0,\xi,v)$ represent the Fourier transform of the initial molecular density of the gas and assume that $\hat f_0\in L^2(\mathbb{R},H^1_{w}(\bbR))$. Then the solution to the Cauchy problem associated with \eqref{maineq} converges to the grossly determined solution $g$ from Theorem \ref{main2} at the exponential rate. More specifically,
	\begin{equation*}
	\|f-g\|_{L^2(\mathbb{R},H^{-1}_{w}(\bbR))} \leq Ce^{- t} \| f_0\|_{L^2(\mathbb{R},H^1_{w}(\bbR))}.
	\end{equation*}
Moreover, if $\hat f_0\in L^2(\bbR, L^2_w)$ and is compactly supported on $(-\sqrt{\pi}, \sqrt{\pi})$ (that is, $\hat f_0$ is compactly supported with respect to $\xi$), then 
\begin{equation*}
\|f-g\|_{L^2(\mathbb{R},L^{2}_{w}(\bbR))} \leq Ce^{- t} \| f_0\|_{L^2(\mathbb{R},L^2_{w}(\bbR))}.
\end{equation*}
\end{theorem}

\begin{proof}
	It directly follows from formula \eqref{norm_est}, Theorem \ref{main1} and Theorem \ref{main2}.
\end{proof}


\subsection{Higher regularity}\label{s:reg}

Applying the results of Theorems \ref{main1}--\ref{main3} to the differentiated equations, we obtain the following
higher-regularity analogs.

\begin{corollary}\label{regcor}
	For $ f_0\in H^{r+s}(\mathbb{R},H^{s+1}_{w}(\bbR))$, $r, s\geq 0$,
the Cauchy problem associated with \eqref{maineq} has a unique solution in $H^{r}(\mathbb{R},H^{s-1}_{w}(\bbR))$,
which, moreover, satisfies
	\begin{equation}\label{dec1}
		\|f-g\|_{H^r(\mathbb{R},H^{s-1}_{w}(\bbR))} \leq Ce^{- t} \| f_0\|_{H^{r+s}(\mathbb{R},H^{s+1}_{w}(\bbR))}.
	\end{equation}
Moreover, if $\hat f_0$ is compactly supported on $(-\sqrt{\pi}, \sqrt{\pi})$ then
	\begin{equation}\label{dec2}
	\|f-g\|_{H^{r}(\mathbb{R},H^{s}_{w}(\bbR))} \leq Ce^{- t} \| f_0\|_{H^{r+s}(\mathbb{R},H^s_{w}(\bbR))}.
\end{equation}
\end{corollary}

\begin{proof}
We first observe, by Parseval's identity, that the same Fourier transform estimates used to prove Theorems
	\ref{main1} and \ref{main3} establish also \eqref{dec1} and \eqref{dec2} for $s=0$ and arbitrary 
	$r\geq 0$.
	The full result then follows by induction on $s$.  Namely, supposing it is true for $s$,
	we consider $f_0 \in H^{r+s+1}(\mathbb{R},H^{s+2}_{w}(\bbR))$.  By the induction hypothesis, we
	thus have a unique solution $f \in H^{r+1}(\mathbb{R},H^{s-1}_{w}(\bbR))$.  
	Defining now $h:=\partial_v f$, we have then $h_0 \in H^{r+s+1}(\mathbb{R},H^{s+1}_{w}(\bbR))$, with
	$h$ satisfying the variational equation
\begin{equation*}
\partial_th- Lh= -\partial_x f
\end{equation*}
obtained by differentiating
$
0= \partial_tf- Lh= \partial_t f + (1- v\partial_x )f - \mathbb{1}\langle \mathbb{1},f\rangle_w$
with respect to $v$.
Noting that $\partial_x f \in H^{r}(\mathbb{R},H^{s-1}_{w}(\bbR))$ by $f \in H^{r+1}(\mathbb{R},H^{s-1}_{w}(\bbR))$,
	and recalling that $h_0 \in H^{r+s+1}(\mathbb{R},H^{s+1}_{w}(\bbR))$, 
	we obtain by Duhamel's principle, together with the induction hypothesis, that 
	$
	h=\partial_v f \in H^{r}(\mathbb{R},H^{s-1}_{w}(\bbR)),
	$
	and thus  $ f \in H^{r}(\mathbb{R},H^{s}_{w}(\bbR))$, yielding the result for $s+1$.
	By induction, we thus obtain the result for all $r, s\geq 0$.
\end{proof}


\section{$L^2\to L^2$ decay by $C_0$ semigroup approach}\label{s:semigp}
Using our detailed spectral expansion of the linearized operator $L$,
we have established $L^2$ decay to grossly determined solutions at the sharp exponential rate
$O(e^{-t})$, at the expense of a loss of two spatial derivatives.
This is somewhat analogous to the case of a first-order system of PDE with real but not semisimple characteristics,
for one may observe the similar phenomenon of boundedness in $L^2$ with loss of one or more derivatives.
For example, consider the system
\begin{equation*}
u_t-v_x=0, \quad v_t=0,
\end{equation*}
or, in vector form $w=(u,v)$, $w_t + A w_x=0$ with $A=\begin{pmatrix} 0 & 1 \\ 0 & 0\end{pmatrix}$
given by a Jordan block, which evidently has a solution that is bounded in time from $H^1\to L^2$, but
unbounded from $L^1\to L^2$.

However, the situation is somewhat less degenerate, in that the solution is bounded (globally in time) 
from $L^2\to L^2$, and in fact {\it decays exponentially in $L^2$} for $L^2$ data 
to the family of grossly determined solutions, at any subcritical rate $O(e^{(\epsilon-1)t})$, $\epsilon>0$.
This is most readily seen by alternative, $C_0$ semigroup estimates, as we now describe.
Indeed, we do not see how to obtain such bounds within the rigged space framework of the rest of the paper;
nor do we see how to obtain the bounds of Theorem \ref{main2} by usual semigroup techniques.

Consider the resolvent equation
\be\label{reseqz}
\lambda f- L f= (\lambda  + v\partial_x + (I-V))f=g,
\ee
where $\Re\lambda>0$ (therefore, $\lambda\in\rho(L)$), $f\in\dom(L)$, $g\in L^2(\R, L^2_w)$ and $V=  (\cdot, \mathbb{1})_{L^2_w} \mathbb{1}$. 
Taking the real part of the $L^2(\R, L^2_w)$ inner product of \eqref{reseqz} with $f$ gives
$$
\Re \lambda \|f\|^2
+ \|f\|^2 - \int_\R |( f,\mathbb{1} )_{L^2_w}|^2dx = \Re \langle f, g\rangle.
$$
By Cauchy-Schwarz, $\int_\R |( f,\mathbb{1} )_{L^2_w}|^2dx\leq\int_\R \| f\|^2_{L^2_w}\|\mathbb{1}\|_{L^2_w}^2dx=\|f\|^2$. Therefore, 
$$
\Re \lambda \|f\|^2\leq \Re \langle f, g\rangle.
$$
Hence, by Cauchy-Schwarz, $ \|(\lambda-L)^{-1}g\| \leq \frac{ \|g\|}{ \Re \lambda } $, verifying that
$e^{Lt}$ is a contraction semigroup in $L^2(\R, L^2_w)$, and thus \eqref{maineq}
is well-posed from $L^2\to L^2$, improving the regularity obtained in Theorem \ref{main1}.

To obtain $L^2\to L^2$ decay to grossly-determined solutions, we consider the Fourier-transformed resolvent equation
\begin{equation*}
\lambda \hat f- L_\xi \hat f= (\lambda  + i\xi v+ (I-V))\hat f=\hat g
\end{equation*}
using a variant of { Pr\"uss' Theorem} (see \cite{Pr}, \cite[Thm. V.1.11]{EN}) established in \cite[Prop. 2.1]{HS}.

\begin{proposition}[Quantitative Gearhardt-Pr\"uss Theorem \cite{HS}]\label{qpruss}
	A $C_0$ semigroup $e^{Lt}$ on a Hilbert space $H$ is exponentially stable,
	$|e^{Lt}|\leq C_1 e^{-\omega_1 t}$ for some $\omega_1>0$, $C_1\geq1$,
	if and only if its generator $L$ 
	(i) has resolvent set containing the right complex half-plane $ \Lambda^+=\{\lambda: \, \Re \lambda > 0\}$, 
	and (ii) satisfies a uniform resolvent estimate $|(\lambda-L)^{-1}|\leq M$ on $\Lambda^+$,
	in which case it satisfies for each $\omega>0$ a uniform exponential growth bound
	\be\label{pexp}
	|e^{Lt}|\leq C(\omega, M)e^{\omega t},\,\,C(\omega, M)\geq1. 
	\ee
\end{proposition}

\br\label{prussrmk}
Here the if only if part of Proposition \ref{qpruss} is the standart Pr\"uss' Theorem (\cite{Pr}, \cite[Thm. V.1.11]{EN}); the quantitative part is expressed in \eqref{pexp} (\cite{HS}). Note that the sharp abstract result of exponential decay is relaxed to exponential {\emph growth} in order to
obtain the quantitative bound \eqref{pexp}.
For $L$ satisfying a uniform resolvent bound $|(\lambda-L)^{-1}|\leq M$ on $\Lambda_{\omega_0}^+:=\{\lambda:
\Re \lambda \geq -\omega_0<0\}$, we obtain from Proposition \ref{qpruss}
a quantitative exponential decay bound $|e^{Lt}|\leq C(\omega, M)e^{-\omega t}$ for any $0<\omega<\omega_0$.
For our applications below, uniformity of estimates with respect to Fourier frequency is convenient, and
so the quantitative nature of this bound will be particularly useful. It is not essential, however; see
Remark \ref{nonessential}.
\er

A useful observation is that when the closure $\bar \Lambda^+$ lies in the resolvent set of $L$,
we can relax in Proposition \ref{qpruss} the assumption of a uniform resolvent bound
for all $\lambda \in \Lambda^+$ to a uniform resolvent bound on the imaginary axis
$\ell_0=\{\lambda: \Re \lambda=0\}$ together with a uniform bound on $\Lambda^+$ for $|\lambda|\geq R$
sufficiently large. For, as the resolvent is analytic on the resolvent set, we obtain from the latter bounds
via the maximum principle a uniform resolvent bound on the entire half-plane $\Lambda^+$, thus giving
the result by Proposition \ref{qpruss}.
Our next result generalizes Proposition \ref{qpruss} still further, giving exponential decay conditions
for finite codimension subspaces, with uniform dependence on parameters.

\begin{corollary}[Finite-codimension Pr\"uss Theorem with parameters]\label{modpruss}
	Let $L(p)$ be a family of generators of $C_0$ semigroups $e^{L(p)t}$, depending on
	a parameter $p\in K\subset \R^n$, with $K$ compact. Suppose further that (i) the vertical line
	$\ell_{\omega_0}=\{\lambda: \Re \lambda=-\omega_0<0\}$ 
	lies in the resolvent set of all $L(p)$, with a uniform resolvent 
	bound $|(\lambda-L(p))^{-1}|\leq M$ on $\ell_{\omega_0}$ for all $p\in K$, (ii) the spectra of $L(p)$ lying to the
	right of $\ell_{\omega_0}$ are of finite total algebraic multiplicity, (iii) the total eigenprojection onto the spectra of $L(p)$ lying in $\Lambda_{\omega_0}^+$ denoted by $Q(p)$ and $L(p)Q(p)$ are continuous with respect to $p\in K$,  and (iv) on the complex
	halfplane $\Lambda_{\omega_0}^+= \{\lambda: \, \Re \lambda > -\omega_0 \}$, there holds
	a uniform resolvent bound $|(\lambda-L(p))^{-1}|\leq M$ for all $|\lambda|\geq R$ sufficiently large, $p\in K$.
	Then, for some $\omega<\omega_0$ and $C>0$,
	\begin{equation*}
	|e^{L(p)t}(I-Q(p))|\leq C e^{-\omega t}, \,\,\hbox{all} \,\,p\in K.
	\end{equation*}

\end{corollary}

\br
The result stated in Corollary \ref{modpruss}  maybe recognized as a variant of \cite[Theorem 1.6]{HS}.  Note in Corollary \ref{modpruss} that it is sufficient for hypothesis (ii)
to check finite multiplicity for a single value of $p$, as the large-$\lambda$ resolvent bound implies that spectra
can neither escape to nor enter from infinity, and so the multiplicity is independent of $p$.
\er

\br\label{nonessential}
In our particular application of operators parametrized by Fourier frequency,
we could work on the space defined by the range of $I-Q(\xi)$ in each Fourier frequency
$\xi$, noting that uniform resolvent estimates for each $\xi$ imply by Parseval's identity an $L^2$ 
resolvent bound on the whole space, giving exponential decay \eqref{pexp} by the usual (nonquantitative) Pr\"uss bound
of \cite{Pr,EN}.
However, Proposition \ref{modpruss} avoids the need for such maneuvers, and seems of independent interest as well.
\er

\begin{proof}[Proof of Corollary \ref{modpruss}]
	By Proposition \ref{qpruss}, it is sufficient to show that $L(p)(I-Q(p))$ restricted to the range of 
	$(I-Q(p))$ satisfy a uniform resolvent bound on $\Lambda^+_{\omega_0}$.
	By assumption, the resolvent set of 
	$L(p)(I-Q(p))$ contains all of $\Omega^+_{\omega_0}$, hence its resolvent is analytic.
	By the maximum principle, it thus suffices to establish a uniform resolvent bound for
	$L(p)(I-Q(p))$ on $\Lambda^+_{\omega_0}$ for $|\lambda|\geq R$ sufficiently large, a bound
	on the compact set $\{|\lambda|\leq R\}\cap L_{\omega_0}$ being available by simple continuity.
	But the latter uniform bound in turn follows readily from assumption (iii) on the full resolvent,
	plus the observation that $L(p)Q(p)$, since finite-dimensional, is a continuous family of bounded operators
	satisfying a uniform resolvent bound $C/|\lambda|$, hence the resolvent of $L(p)(I-Q(p))$, as the
	difference between total and $Q$-projected resolvents, is uniformly bounded as claimed.
\end{proof}

With Corollary \ref{modpruss} in hand, we now readily obtain $L^2\to L^2$ decay to grossly determined solutions.
Modifying \eqref{projfor}, define the truncated $L(\xi)$-invariant projectors
\begin{align*}
\begin{split}
	P_{\xi_0}(\xi):=\left\{\begin{array}{l} 
P_{\lambda^*}(\xi),\,\,\xi\in 
		[-\xi_0, \xi_0]\subset (-\sqrt{\pi},\sqrt{\pi}),\\
0,\,\,\xi\notin[-\xi_0,\xi_0].
\end{array} \right.
\end{split}
\end{align*}
Then, following the proof of Theorem \ref{main2}, we have for any initial molecular density $f_0\in L^2_w(\R^2)$, 
the function 
\be\label{tg}
g^{\xi_0}(t,x,v):=\mathfrak{F}^{-1}(P_{\xi_0}(\xi)\hat f)
\ee
is a grossly determined solution of \eqref{maineq}.

\begin{theorem}\label{main4}
For $\hat f_0(\xi,v) \in L^2(\mathbb{R},L^2_{w}(\bbR))$, 
the solution to the Cauchy problem associated with \eqref{maineq} 
converges to the grossly determined solution $g^{\xi_0}$ of \eqref{tg}
at exponential rate 
	\begin{equation}\label{L2exp}
		\|f-g^{\xi_0}\|_{L^2(\mathbb{R},L^{2}_{w}(\bbR))} \leq Ce^{(\lambda^*(\xi_0)+\eps) t} 
		\| f_0-g^{\xi_0}_0\|_{L^2(\mathbb{R},L^2_{w}(\bbR))}
		\leq
		C_2 e^{(\lambda^*(\xi_0)+\eps) t} \| f_0\|_{L^2(\mathbb{R},L^2_{w}(\bbR))}
	\end{equation}
	for any $\epsilon >0$ and some $C=C(\epsilon)>0$.
\end{theorem}

\begin{proof}
The unperturbed operator $M_\xi:= -iv\xi - 1$ is readily seen to have a uniformly bounded resolvent 
	$R^0(\lambda,\xi)= (- iv\xi - 1-\lambda)^{-1}$ on $\bar \Lambda^+_{\eta} =\{\lambda: \Re \lambda\geq -\eta\}$
for $\eta<1-\epsilon_0$ and all $\xi\in\bbR$, where a positive $\epsilon_0$ is fixed as can be seen from the following estimates:
\begin{align*}
\begin{split}
|- iv\xi - 1-\lambda|\geq|1+\Re\lambda|&>\epsilon_0,\\
\big\|\frac{1}{- iv\xi - 1-\lambda}\big\|_{L^\infty}&<\frac{1}{\epsilon_0}.
\end{split}
\end{align*}   
Moreover, the resolvent of the rank one perturbation $L_\xi=M_\xi + V=
M_\xi +  ( \cdot,\mathbb{1} )_{L^2_w} \mathbb{1}$ has the following representation
(see \eqref{perR} and Proposition \ref{Kinvert})
\begin{align}\label{TBS}
\begin{split}
(L_\xi-\lambda)^{-1}&= 
\big(I-R^0(\lambda,\xi)VK^{-1}(\lambda,\xi)V\big)R^0(\lambda,\xi)\\
&=\Big(I-R^0(\lambda,\xi)V+\frac{1}{\omega(\lambda,\xi)}R^0(\lambda,\xi)VR^0(\lambda,\xi)V\Big)R^0(\lambda,\xi).
\end{split}
\end{align}

	In particular, by Theorem \ref{unifomega}, we have uniform resolvent estimates for all $L_\xi$ on $\Lambda_\eta$ for $|\lambda|\geq R$
	sufficiently large, $\xi\in\bbR$.\\
	For $\xi \in(-\sqrt{\pi},\sqrt{\pi})$ such that $\xi\not \in [-\xi_0, \xi_0]$, we have $ \lambda^*(\xi)< \lambda^*(\xi_0)$, and so
	the resolvent of $L_\xi$ is uniformly bounded on $\ell_{|\lambda^*(\xi_0)+\epsilon|}$ 
	for any fixed $\epsilon>0$, with bound depending on $\epsilon$ (cf. \eqref{TBS} and Theorem \ref{unifomega}).
	Applying Corollary \ref{modpruss} with $Q=0$, we obtain 
	$$
	|e^{L_\xi t}|_{L^2_w}\leq Ce^{\lambda^*(\xi_0)+\epsilon)t}.
	$$
	Similarly, we obtain the bound for $\xi\notin(-\sqrt{\pi}, \sqrt{\pi})$ as there are not isolated eigenvalues of $L_\xi$ to the right of the line $\ell_{-1}$ .\\
	For $\xi \in [-\xi_0, \xi_0]$, on the other hand, we have $ \lambda^*(\xi)\geq \lambda^*(\xi_0)$, and so
	the resolvent of $L_\xi$ is uniformly bounded on $\ell_{|\lambda^*(\xi_0)-\epsilon|}$ 
	for any fixed $\epsilon>0$ such that $\lambda^*(\xi_0)-\epsilon> -1$, 
	with bound depending on $\epsilon$ (cf. \eqref{TBS} and Theorem \ref{unifomega}).
	Applying Corollary \ref{modpruss} with $Q=P_{\xi_0}(\xi)$ ($P_{\xi_0}(\xi)$ and $L_\xi P_{\xi_0}(\xi)$ are continuous due to formula \eqref{projall} and Proposition \ref{continuity16}), we obtain 
	$$
	|e^{L(\xi)t}(I-P_{\xi_0}(\xi)|_{L^2_w}\leq Ce^{\lambda^*(\xi_0)+\epsilon)t}.
	$$
	Combining these estimates yields \eqref{L2exp}, by Parseval's identity together with 
	definition $\hat {g^{\xi_0}}= P_{\xi_0}(\xi)\hat f$ (yielding the first inequality) 
	together with boundedness of $P_{\xi_0}(\xi)$ (yielding the second).
\end{proof}

\br\label{hybridrmk}
A somewhat simpler proof of Theorem \ref{main4} may be obtained by combining the rigged space estimate 
of Theorem \ref{main3} for $\xi\in [-\xi_0, \xi_0]$, with the quantitative Pr\"uss estimate of 
Proposition \ref{qpruss} for $\xi\not \in [-\xi_0, \xi_0]$, avoiding the use of the 
finite-codimension version of Corollary \ref{modpruss}.
\er

\br\label{diffg}
Note that the grossly determined solution $g^{\xi_0}$ of Theorem \ref{main4} is different from the 
grossly determined solution $g$ of Theorem \ref{main3}, the former being a Fourier truncation of the latter.
The difference between the two decays exponentially 
in $L^2(\R, H^{-1}_w(\R))$ for data in $L^2(\R, H^1_w(\R))$,
by comparison of the decay rates toward both solutions.
However, it is in general unbounded in $L^2(\R, L^2_w(\R))$,
since $|P_{\lambda^*}(\xi)|_{ H^1_w\to L^2_w}\sim \|e_1\|_{H^{-1}_w}\|\|e_1\|_{L^2_w}\to \infty $ 
as $\xi\to \pm \sqrt{\pi}$ by Proposition \ref{projs}.
\er 

From Theorem \ref{main4}, we see that the situation of \eqref{maineq}
is more like that of a Jordan block in ODE theory, for which the exponential rate is
degraded from that suggested by the spectral radius, but not destroyed, than that considered above
of a Jordan block occurring in first-order PDE, where even well-posedness is lost.
More precisely, the situation is somewhere between the two, in the sense that, for general 
$L^2(\R, L^2_w(\R))$ or even $L^2(\R, H^1_w(\R))$ data, the solution of \eqref{maineq} does not decay
in $L^2(\R, L^2_w(\R))$ to {\it any} grossly bounded solution at rate $O(e^{-t}r(t))$, with
$r$ growing at subexponential rate.

For, this would imply that slower-decaying modes $e^{\lambda^*(\xi)t} P_{\lambda^*}(\xi)\hat f(\xi)$
be contained in the grossly determined solution, for all $\xi \in (-\sqrt{\pi},\sqrt{\pi})$.
On the other hand, the GDS property- specifically, that the evolution of $f$, and thus $\mu$, be determined
by an autonomous evolution system for $\mu$- together with Remark \ref{recoveryrmk}, requires that
$\hat g(\xi)$ for each Fourier frequency $\xi$ contain only a single eigenmode, so that $\hat g$ must
be exactly $P(\xi)\hat f$.
Thus, the only candidate for such a grossly determined solution is the solution $g=Pf$ of Theorem \ref{main3},
containing the range of all $P(\xi)$.
But, $P$ and thus the complementary projection $(I-P)$ are in general unbounded from 
$L^2(\R, H^1_w(\R)) \to L^2_w(\R))$ by Remark \ref{diffg},
hence $g$ and $f -g$ are both unbounded in $L^2(\R, L^2_w(\R))$ and $r(t)\equiv +\infty$, a contradiction.


\subsection{Grossly determined solutions vs. Chapman--Enskog approximation}\label{s:CE}

It is interesting to compare the description of asymptotic behavior given by the grossly determined
solution \eqref{tg} to that given by the classical Chapman--Enskog expansion \cite{G62,S97},
or ``Navier--Stokes approximation.''
The latter, in the present case comprises solutions $g_{NS}= \mu_{NS}e_1$ satisfying
the second-order scalar conservation law
\be\label{scalarNS}
\mu_t= \frac12 \mu_{xx},
\ee
with the same data $\mu_0$ as for the grossly determined solution \eqref{tg}.
This can most easily be seen by the fact (see, e.g., \cite[Appendix A1]{Z01}) that the dispersion relation
of the second-order Chapman--Enskog equation linearized about a constant state is equal to the second-order
Taylor expansion $\lambda_2(\xi)=-\frac12 \xi^2$ about $\xi=0$ of the 
``neutral'' spectral curve $\lambda=\lambda_*(\xi)$ passing through $(\xi,\lambda)=(0,0)$; see \eqref{l2}.
As the scalar BGK model, hence also its Chapman--Enskog expansion, is linear to begin with, this gives
\be\label{CEevolution}
\hat \mu_{NS}(\xi)=e^{\lambda_2(\xi)t}\hat \mu_0,
\ee
or \eqref{scalarNS}.

Comparing the behavior of the nonlocal evolution
\eqref{gdsfull}(i) to that of the local, diffusion equation \eqref{CEevolution}, we find that they are both
merely bounded in $L^2$ for $L^2$ initial data.  Meanwhile, the difference between the two decays, by Parseval's 
identity and the fact that $\lambda_*$ is even, as
$$
\sup_\xi | e^{\lambda_*(\xi)t- \lambda_2(\xi)t}|
\sim \sup_\xi |e^{-\xi^2t/2}\xi^4 t|
=\sup_\xi |e^{-\xi^2t/2}(\xi^2 t)^2|/t \sim 1/t
$$
as $t\to +\infty$.  Thus, the exact solution converges exponentially to the grossly determined solution
$g$, but only algebraically to the Chapman--Enskog approximation $G_{NS}$, at rate $(1+t)^{-1}$ in $L^2$
for $L^2$ initial data.

\section{Discussion and open problems}\label{s:disc}

In summary, we have shown that the spectral program initiated by Carty in \cite{C16krm,C17}
can be rigorously completed using the rigged space framework developed by Ljance and others \cite{L70} for small
nonselfadjoint perturbation of (selfadjoint) multiplication operators, while at the same time 
demonstrating the potential of the latter for practical applications.
As noted in the introduction, it appears likely that our approach should extend, if perhaps less
explicitly, to the case of finite rank perturbations, including the full BGK model linearized about a Maxwellian
state.

However, the analysis also highlights an important limitations of the rigged space approach for nonselfadjoint problems,
at least when used as here solely via the generalized Parseval inequality/spectral expansion.
Namely, different from the selfadjoint case, there can arise considerable cancellation in the solution 
formula analogous to \eqref{solnform} via spectral expansion for the associated linear evolution problem.
Thus, the bounds on the solution obtained here from \eqref{solnform} by crudely integrating the norm of the solution 
over $\lambda$ involve a loss of two derivatives, whereas the semigroup estimates of Section \ref{s:semigp}
show that no such derivative loss in fact occurs.

This seems somewhat analogous to the situation of analytic semigroup theory and estimation through the inverse Laplace transform formula $e^{LT}=\frac{1}{2\pi i} \oint_\Gamma e^{\lambda t} (\lambda -L)^{-1} \, d\lambda$, where $\Gamma$ is
any sectorial contour enclosing (in appropriate sense) the spectra of $L$.  Taking $\Gamma$ distance zero from the spectra yields the spectral expansion formula, \cite{Kato}, which in general may involve Jordan blocks and other delicate
cancellation. Typically one does not estimate the semigroup in this way, but rather uses the power of analytic
extension to obtain bounds through resolvent estimates at a finite distance from the spectra.
This raises the question whether cancellation may be detected (i) (at least in some cases) {\it directly} from a very explicit description of the spectral expansion, thus combining
the useful aspects of detailed eigenexpansion and control of conditioning, or (ii) 
indirectly, by using the analytic extension inherent in the construction of the rigged space to obtain an estimate at
finite distance from the spectrum of the original unperturbed operator.


As regards question (ii), the only route we see is to start with the formal resolution of the identity
\be\label{idres}
 \Id =\int_{\R}[R] (\lambda) d\lambda  +  \sum_j P_{\lambda_j}
\ee
coming from the eigenfunction expansion analogous to Theorem \ref{main_expension_thm},
where $[R]$ denotes the jump in resolvent $R=(\lambda-L)^{-1}$ across the real line, and $P_{\lambda_j}$
the spectral projectors as $\lambda_j$ runs over the point spectrum of the perturbed operator $L$,
then analytically continue the integral $\int_\R$ into the complex plane by continuation of $[R]$:
that is, to continue $\mathcal{B}$ and $\mathcal{U}$ in \eqref{intfor} while holding fixed the function $f$.
For, otherwise, we see no useful way to estimate the trace of $f$ on a perturbed contour $i\eta + \R$ from
its trace on $\R$.
This leads to a solution formula
\be\label{modilt}
 e^{Lt} =\int_{\R+i\eta } e^{\lambda t} [R] (\lambda) d\lambda  +  \sum_j e^{Lt} P_{\lambda_j}
\ee
modifying the standard inverse Laplace transform formula, valid on the class of functions for which
\eqref{idres} holds, which can be usefully estimated by varying $\eta$.
For similar estimates in a different (and sectorial) context, see, e.g., \cite{OZ03}.

We note, in the scalar BGK case considered here, that the class of functions on which \eqref{idres}
holds is $H^2$, the same function class
$D(L^2)$ on which the inverse Laplace transform formula is guaranteed to hold by $C_0$ semigroup theory \cite{Pa11}.
Moreover, denoting $[R]=R^+-R^-$, and noting that
 $\int_{\R-i\eta } e^{\lambda t} R^- (\lambda) d\lambda  $ vanishes by causality
 for $ \eta>0$, we see that \eqref{modilt} reduces in this case to the
 standard inverse Laplace transform formula
 $$
 e^{Lt} =\int_{\R+i\eta } e^{\lambda t} R^+ (\lambda) d\lambda  +  \sum_j e^{Lt} P_{\lambda_j}.
 $$
 Thus, the main advantage of the rigged-space formalism for this type of calculation seems to us to be to give
 a useful functional calculus by which to compute the integral \eqref{modilt}.
 As far as analytic continuation of $[R]$, it seems that this must be determined afterward to hold in strong sense
 and not only the weak sense guaranteed by rigged space theory: that is, analyticity of $[R]$ is not guaranteed 
 by the rigged space formalism, but is a separate issue.
 Whether one could conclude analyticity (in strong sense) from the rigged space point of view is an interesting question
 for further investigation.

A related issue, and one of our original motivations for pursuing the present work, is whether the explicit
spectral representation formula/generalized Fourier transform afforded by the rigged space approach, can yield
also bounds in other norms than the original $H^s_w$ of the rigged space construction, for example in the
Banach norms $W^{k,p}_w$.
In particular, as noted in \cite{PZ16,Z17}, it is a very interesting question related to invariant manifolds for
a stationary kinetic problem (and thereby existence/structure of shock and boundary layers) whether or not there is 
an $L^\infty\to L^\infty$ bound on the resolvent $L^{-1}$ of the linearized problem $Lu=f$
restricted to the complementary subspace to the kernel of the collision operator- in the present case, the complementary
subspace to $\mathbb{1}$.
The answer for Boltzmann's equation is not known; the study for BGK models, whose linearized operators, 
as finite-rank perturbations of multiplication operators, fit the analytical framework used here, could perhaps be a useful step toward that ultimate goal.  
Recall that the linearized Boltzmann equation, as a compact perturbation of a multiplication
operator, is the limit of finite-rank perturbations.

Finally, we return to the physical question with which we opened the paper, of Truesdell and Muncaster's conjecture
\cite{TM} of decay to grossly determined solutions for Boltzmann's equation, and presumably for related kinematic
and relaxation systems as well.
For the full BGK model, a slight modification of the methods used here should verify decay to grossly determined
solutions at the linear level, where the grossly determined solutions are appropriate Fourier truncations 
of the family of discrete eigenmodes as the Fourier frequency $\xi$ is varied.
However, so far as we know, such a result has not been carried out in any nonlinear setting. 
Thus, a very interesting open problem is to verify decay to grossly determined solutions for {\it any}
example of a nonlinear kinetic or relaxation system.
An equally interesting question, assuming that such a result were carried out, would be to identify the
resulting asymptotic dynamics as a Taylor expansion in the Fourier frequency $\xi$.
This should presumably agree to lowest order with the local (i.e., differential) model given by
formal Chapman--Enskog expansion (CE); however, being nonlocal, the GDS dynamics should differ at higher orders,
for which (CE) is known to become ill-posed.
This could 
perhaps
shed interesting new light on 
Slemrod's investigations in \cite{S97} of nonlocal closures of (CE)
designed to restore well-posedness while preserving higher-order agreement with (CE).

\appendix
\section{Computation of discrete spectra}\label{s:discrete}
Finally, we show that, remarkably, both the spectral determinant 
(or ``Evans function'' \cite{GLM,GLMZ,GLZ}) $\omega(\lambda,\xi)$
and the associated spectral curve $\lambda^*(\xi)$, $\omega(\lambda^*(\xi),\xi)=0$,
may be explicitly determined for the scalar BGK model, along with the full Taylor expansion of $\lambda^*$ 
about $\xi=0$.

\begin{proposition}\label{discreteprop}
	\begin{enumerate}
		\item 	The function $\lambda^*$ from Proposition \ref{disspec} satisfies the following singular differential equation
		\begin{align*}
		\frac{d\lambda^*}{d\xi}=\frac{\xi}{2\lambda^*}+\frac{\lambda^*}{\xi}+\frac{1}{\xi}
		\end{align*}
		whose implicit solution is 
		\begin{equation*}
		e^{-z^2}\xi=-2\int e^{-z^2}dz, \,\,z=\frac{\lambda^*}{\xi}+\frac{1}{\xi}.
		\end{equation*}
		If we impose the initial condition $\lambda^*(0)=0$, then the solution $\lambda^*(\cdot)$ is real and its implicit formula is given by
		\begin{align}\label{xipos}
		\begin{split}
		e^{-(\frac{\lambda^*}{\xi}+\frac{1}{\xi})^2}&=\frac{2}{\xi}\int^{\infty}_{(\frac{\lambda^*}{\xi}+\frac{1}{\xi})}e^{-t^2}dt,\,\,\,\xi>0,\\
		e^{-(\frac{\lambda^*}{\xi}+\frac{1}{\xi})^2}&=-\frac{2}{\xi}\int_{-\infty}^{(\frac{\lambda^*}{\xi}+\frac{1}{\xi})}e^{-t^2}dt,\,\,\,\xi<0,\\
		\lambda^*(0)&=0,
		\end{split}
		\end{align}
		or,
		\begin{align*}
		e^{-(\frac{\lambda^*}{\xi}+\frac{1}{\xi})^2}&=\frac{\sqrt{\pi}}{\xi}[{\sign{\xi}}-\erf\big(\frac{\lambda^*}{\xi}+\frac{1}{\xi}\big)], \,\,\xi\neq0,\\
		\lambda^*(0)&=0.
		\end{align*}
		Also, $\lambda^*(\cdot)$ has the following serier representation near $\xi=0$:
		\begin{equation*}
		\lambda^*(\xi)=\sum_{j=1}^\infty a_{2j}\xi^{2j},\,\,\,a_2=-\frac{1}{2},\,\,\,a_{2j}=\sum_{r=1}^{j-1}(2r-1) a_{2r}a_{2(j-r)},\,\,j\geq2.
		\end{equation*} 
		\item Moreover, if $\lambda$ is real and $\xi=\pm\sqrt{\tau}$, then $\omega$ satisfies the heat equation with the imposed initial condition
		\begin{align*}
		\partial_\tau\omega&=	-\frac{1}{4}\partial_{\lambda\lambda}\omega,\\
		\omega(\lambda,0)&=1-\frac{1}{1+\lambda}, \,\,\lambda>-1.
		\end{align*}
		\item Let $\lambda$ be a real fixed value, that is, $\omega(\lambda,\xi)=1-
		\int_\mathbb{R}\frac{w(v)dv}{vi\xi+1+\lambda}=1-\int_\mathbb{R}\frac{(1+\lambda)w(v)dv}{(1+\lambda)^2+(v\xi)^2}.$ Then, $\omega$  satisfies the following differential equation:
		\begin{align*}
		\begin{split}
		\frac{d\omega}{d\xi}+\Big(\frac{1}{\xi}+\frac{2(1+\lambda)^2}{\xi^3}\Big)\omega=\frac{1}{\xi}+\frac{2\lambda(1+\lambda)}{\xi^3}
		\end{split}
		\end{align*}
		whose solution is
		\begin{align*}
		\begin{split}
		\omega=\frac{e^{\frac{(1+\lambda)^2}{\xi^2}}}{\xi}\int e^{-\frac{(1+\lambda)^2}{\tilde\xi^2}}\Big(1+\frac{2\lambda(1+\lambda)}{\tilde\xi^2}\Big)d\tilde\xi.
		\end{split}
		\end{align*}
		If we impose the initial condition $\omega(\lambda,0)=1-\frac{1}{1+\lambda}, \,\,\lambda>-1$, then the solution is
		
		\begin{align*}
		\begin{split}
		\omega=\frac{e^{\frac{(1+\lambda)^2}{\xi^2}}}{\xi}\int_0^\xi e^{-\frac{(1+\lambda)^2}{\tilde\xi^2}}\Big(1+\frac{2\lambda(1+\lambda)}{\tilde\xi^2}\Big)d\tilde\xi,
		\end{split}
		\end{align*}
		which also implies \eqref{xipos}.
	\end{enumerate}

\end{proposition}
\begin{proof}
	(1) We have $\int w(v)dv=1$ and $w'(v)=-2vw(v)$. Therefore, if $\Im\lambda\neq-1$, then
	\begin{align*}
	\begin{split}
	1&=\int w(v)dv=\int_\mathbb{R}\frac{vi\xi w(v)dv}{vi\xi +1+\lambda}+(1+\lambda)\int_\mathbb{R}\frac{w(v)dv}{vi\xi +1+\lambda}\\
	&\frac{-i\xi}{2}\int_\mathbb{R}\frac{ w'(v)dv}{vi\xi +1+\lambda}+(1+\lambda)(1-\omega(\lambda,\xi))\overset{by\, parts}{=}\frac{-(i\xi)^2}{2}\int_\mathbb{R}\frac{ w(v)dv}{vi\xi +1+\lambda}+(1+\lambda)(1-\omega(\lambda,\xi)).
	\end{split}
	\end{align*}
	Hence, by applying a corollary of the dominated convergence theorem on differentiation under integral sign, we arrive at $1=\frac{\xi^2}{2}\omega'_{\lambda}(\lambda,\xi)+(1+\lambda)(1-\omega(\lambda,\xi)).$
	Therefore,
	\begin{align*}
	\begin{split}
	\omega'_{\lambda}(\lambda,\xi)=2\frac{1-(1+\lambda)(1-\omega(\lambda,\xi))}{\xi^2}.
	\end{split}
	\end{align*}
	After substituting $\lambda^*(\xi)$ for $\lambda$, we arrive at
	\begin{align}\label{derlam}
	\begin{split}
	\omega'_{\lambda}(\lambda^*(\xi),\xi)=\frac{-2\lambda^*(\xi)}{\xi^2}.
	\end{split}
	\end{align}
	
	Next, by applying a corollary of the dominated convergence theorem on differentiation under integral sign, we differentiate $\omega$ with respect to $\xi$.
	\begin{align*}
	\begin{split}
	\omega'_{\xi}(\lambda,\xi)&=\int_\mathbb{R}\frac{iv w(v)dv}{(vi\xi +1+\lambda)^2}=\frac{1}{\xi}\big(\int_\mathbb{R}\frac{ w(v)dv}{vi\xi +1+\lambda}-(1+\lambda)\int_\mathbb{R}\frac{ w(v)dv}{(vi\xi +1+\lambda)^2}\big)\\
	&=\frac{1}{\xi}(1-\omega(\lambda,\xi)-(1+\lambda)\omega'_{\lambda}(\lambda,\xi)).
	\end{split}
	\end{align*}
	Then, by using \eqref{derlam}, we obtain $	\omega'_{\xi}(\lambda^*(\xi),\xi)=\frac{1}{\xi}(1+(1+\lambda^*(\xi))\frac{2\lambda^*(\xi)}{\xi^2}).$
	Therefore,
	\begin{align*}
	\begin{split}
	(\lambda^{*})'(\xi)&=-\frac{\omega'_{\xi}(\lambda^{*}(\xi),\xi)}{\omega'_{\lambda}(\lambda^{*}(\xi),\xi)}=-\frac{\frac{1}{\xi}(1+(1+\lambda^*(\xi))\frac{2\lambda^*(\xi)}{\xi^2})}{\frac{-2\lambda^*(\xi)}{\xi^2}}=\frac{\xi}{2\lambda^*}+\frac{\lambda^*}{\xi}+\frac{1}{\xi}.
	\end{split}
	\end{align*}
	Our next goal is to solve the obtained ODE for $\lambda^*$. We rewrite the given ODE as
	\begin{align*}
	\begin{split}
	\frac{(\lambda^{*})'}{\xi}-\frac{\lambda^{*}}{\xi^2}&=\frac{1}{2\lambda^*}+\frac{1}{\xi^2}.
	\end{split}
	\end{align*}
	We introduce $u=\frac{\lambda^{*}}{\xi}$, therefore, $u'=\xi^{-1}(\lambda^{*})'-\xi^{-2}\lambda^{*}$, and $	u'=\frac{1}{2\xi u}+\frac{1}{\xi^2}$.
	Now, let $z=u+\frac{1}{\xi}$. Then $z'=\frac{1}{2\xi(z-\frac{1}{\xi}) }$.
Next, we treat $\xi$ as a function of $z$. Then
	\begin{align*}
	\begin{split}
	\xi'&=2(\xi z-1),
	\end{split}
	\end{align*}
	which is a linear equation. In particular, we can rewrite it as $\xi'-2z\xi=-2$.	Therefore, the solution is 
	\begin{equation*}
	e^{-z^2}\xi=\xi(z_0)-2\int^z_{z_0} e^{-t^2}dt.
	\end{equation*}
	Notice that $z=\frac{\lambda^{*}}{\xi}+\frac{1}{\xi}$.  The function $\frac{\lambda^{*}}{\xi}$ goes to $0$ as $\xi\to0^+$. Therefore, $z\to\infty$ as $\xi\to0^+$, and 
	\begin{equation*}
	e^{-z^2}\xi=2\int^{\infty}_z e^{-t^2}dt,
	\end{equation*}
	or,
	\begin{align*}
	e^{-(\frac{\lambda^*}{\xi}+\frac{1}{\xi})^2}=\frac{2}{\xi}\int^{\infty}_{(\frac{\lambda^*}{\xi}+\frac{1}{\xi})}e^{-t^2}dt,\,\,\,\xi>0.
	\end{align*}
	Similarly, $z\to-\infty$ as $\xi\to0^-$, and $e^{-z^2}\xi=-2\int_{-\infty}^z e^{-t^2}dt$,
	or,
	\begin{align*}
	e^{-(\frac{\lambda^*}{\xi}+\frac{1}{\xi})^2}=-\frac{2}{\xi}\int_{-\infty}^{(\frac{\lambda^*}{\xi}+\frac{1}{\xi})}e^{-t^2}dt,\,\,\,\xi<0.
	\end{align*}
	Next, our goal is to find the expansion of $\lambda^*(\cdot)$ near $0$. From the formula for $\omega$ it is clear that $\lambda^*(\cdot)$ is an even function and $\lambda^*(0)=0$, therefore, the expansion series has only even powers greater than $0$, that is, $\lambda^*(\xi)=\sum_{j=1}^\infty a_{2j}\xi^{2j}$. We also know that $\lambda^*(\cdot)$ solves the differential equation $\lambda^*\frac{d\lambda^*}{d\xi}=\frac{\xi}{2}+\frac{\lambda^*(\lambda^*+1)}{\xi}$. Therefore,
	\begin{equation*}
	\sum_{j=1}^\infty a_{2j}\xi^{2j}\sum_{j=1}^\infty 2ja_{2j}\xi^{2j-1}=\frac{\xi}{2}+\sum_{j=1}^\infty a_{2j}\xi^{2j-1}\sum_{j=1}^\infty a_{2j}\xi^{2j}+\sum_{j=1}^\infty a_{2j}\xi^{2j-1}
	\end{equation*} 
	If we compare the coefficients in front of different powers of $\xi$, we arrive at
	\begin{align*}
	\begin{split}
	0&=\frac{1}{2}+a_2,\\
	2a_2^2&=a_4+a_2^2,\\
	&\vdots\\
	a_{2j}&=\sum_{r=1}^{j-1}(2r-1) a_{2r}a_{2(j-r)}\,\,\,\hbox{(induction)}.
	\end{split}
	\end{align*}

	(2) Let $\xi=\pm\sqrt{\tau}$. Now, we compute the second partial derivative of $\omega$ with respect to $\lambda$.
	\begin{align*}
	\begin{split}
	\partial_{\lambda\lambda}\omega&=-2\int_\mathbb{R}\frac{ w(v)dv}{(\pm vi\sqrt{\tau} +1+\lambda)^3}.
	\end{split}
	\end{align*}
	On the other hand, 
	\begin{align*}
	\begin{split}
	\partial_{\tau}\omega&=\frac{i}{\pm2\sqrt{\tau}}\int_\mathbb{R}\frac{v w(v)dv}{(vi\sqrt{\tau} +1+\lambda)^2}\overset{by\, parts}{=}\frac{-i}{\pm4\sqrt{\tau}}(\pm2i\sqrt{\tau})\int_\mathbb{R}\frac{ w(v)dv}{(vi\sqrt{\tau} +1+\lambda)^3}=-\frac{1}{4}\partial_{\lambda\lambda}\omega.
	\end{split}
	\end{align*}
	(3) We compute the derivative of $\omega$ with respect to $\xi$.
	\begin{align*}
	\begin{split}
	\frac{d\omega}{d\xi}&=\int_\mathbb{R}\frac{2\xi v^2(1+\lambda)w(v)dv}{((1+\lambda)^2+(v\xi)^2)^2}=(1+\lambda)\int_\mathbb{R}\frac{2\xi v^2w(v)dv}{((1+\lambda)^2+(v\xi)^2)^2}.
	\end{split}
	\end{align*}
	Introduce $z=\frac{1}{(1+\lambda)^2+(v\xi)^2}$. Then $dz=-\frac{2v\xi^2}{((1+\lambda)^2+(v\xi)^2)^2}dv$ and
	\begin{align*}
	\begin{split}
	\frac{d\omega}{d\xi}&=-\frac{(1+\lambda)}{\xi}\int_\mathbb{R}{ {v}w(v)dz}\overset{by\, parts}{=}\frac{(1+\lambda)}{\xi}\Big(\int_\mathbb{R} zd(vw(v))\Big)=\frac{(1+\lambda)}{\xi}\int_\mathbb{R} \frac{w(v)-2v^2w(v)}{(1+\lambda)^2+(v\xi)^2}dv\\
	&=\frac{1}{\xi}(1-\omega)-\frac{2(1+\lambda)}{\xi^3}\int_\mathbb{R} \frac{(v\xi)^2w(v)}{(1+\lambda)^2+(v\xi)^2}dv=\frac{1}{\xi}(1-\omega)-\frac{2(1+\lambda)}{\xi^3}\Big(\int_\mathbb{R} w(v)dv\\
	&-\int_\mathbb{R} \frac{(1+\lambda)^2w(v)}{(1+\lambda)^2+(v\xi)^2}dv\Big)=\frac{1}{\xi}(1-\omega)-\frac{2(1+\lambda)}{\xi^3}+\frac{2(1+\lambda)^2}{\xi^3}(1-\omega)=-\Big(\frac{1}{\xi}+\frac{2(1+\lambda)^2}{\xi^3}\Big)\omega\\
	&+\frac{1}{\xi}+\frac{2\lambda(1+\lambda)}{\xi^3}.
	\end{split}
	\end{align*}
	Or,
	\begin{align*}
	\begin{split}
	\frac{d\omega}{d\xi}+\Big(\frac{1}{\xi}+\frac{2(1+\lambda)^2}{\xi^3}\Big)\omega=\frac{1}{\xi}+\frac{2\lambda(1+\lambda)}{\xi^3}.
	\end{split}
	\end{align*}
	Therefore, the solution is
	\begin{align*}
	\begin{split}
	{\xi e^{-\frac{(1+\lambda)^2}{\xi^2}}}\omega=\int e^{-\frac{(1+\lambda)^2}{\tilde\xi^2}}\Big(1+\frac{2\lambda(1+\lambda)}{\tilde\xi^2}\Big)d\tilde\xi.
	\end{split}
	\end{align*}
	Recall that $\omega(\lambda,0)=1-\frac{1}{1+\lambda}$. Since $\lim_{\xi\to0}\xi e^{-\frac{(1+\lambda)^2}{\xi^2}}\omega=0$, we have ${\xi e^{-\frac{(1+\lambda)^2}{\xi^2}}}\omega=\int_0^\xi e^{-\frac{(1+\lambda)^2}{\tilde\xi^2}}\Big(1+\frac{2\lambda(1+\lambda)}{\tilde\xi^2}\Big)d\tilde\xi$. Or,
	\begin{equation}\label{omegaexp}
	\omega=\frac{e^{\frac{(1+\lambda)^2}{\xi^2}}}{\xi}\int_0^\xi e^{-\frac{(1+\lambda)^2}{\tilde\xi^2}}\Big(1+\frac{2\lambda(1+\lambda)}{\tilde\xi^2}\Big)d\tilde\xi.
	\end{equation}
	Finally, we would like to show that the implicit solution for $\lambda^*>-1$ from item (1) can be derived from the solution for $\omega$. Indeed, let $\xi>0$. Then, since $\lambda^*$ satisfies the equation $\omega(\lambda^*(\xi),\xi)=0$. Therefore, by \eqref{omegaexp}, we have $\int_0^\xi e^{-\frac{(1+\lambda^*(\xi))^2}{\tilde\xi^2}}\Big(1+\frac{2\lambda^*(\xi)(1+\lambda^*(\xi))}{\tilde\xi^2}\Big)d\tilde\xi=0$. Let $t=\frac{1+\lambda^*(\xi)}{\tilde\xi}$. Then $d\tilde\xi=-\frac{1+\lambda^*(\xi)}{t^2}dt$ and $-\int_{\infty}^{\frac{1+\lambda^*(\xi)}{\xi}} e^{-t^2}\Big(1+\frac{2t^2\lambda^*(\xi)}{(1+\lambda^*(\xi))}\Big)\frac{1+\lambda^*(\xi)}{t^2}dt=0$. Or, $\int^{\infty}_{\frac{1+\lambda^*(\xi)}{\xi}} e^{-t^2}\Big(\frac{1+\lambda^*(\xi)}{t^2}+2\lambda^*(\xi)\Big)dt=0$. Next, we apply the integration by parts formula for the integral
	\begin{align*}
	\begin{split}
	\int^{\infty}_{\frac{1+\lambda^*(\xi)}{\xi}} e^{-t^2}\frac{1+\lambda^*(\xi)}{t^2}dt=\xi e^{-\frac{(1+\lambda^*(\xi))^2}{\xi^2}}-2\int^{\infty}_{\frac{1+\lambda^*(\xi)}{\xi}} (1+\lambda^*(\xi))e^{-t^2}dt.
	\end{split}
	\end{align*}
	Hence, the implicit equation for $\lambda^*$ can be written as
	\begin{equation*}
	\xi e^{-\frac{(1+\lambda^*(\xi))^2}{\xi^2}}-2\int^{\infty}_{\frac{1+\lambda^*(\xi)}{\xi}} (1+\lambda^*(\xi))e^{-t^2}dt+\int^{\infty}_{\frac{1+\lambda^*(\xi)}{\xi}} 2\lambda^*(\xi)e^{-t^2}dt=0.
	\end{equation*}
	Or, $e^{-\frac{(1+\lambda^*(\xi))^2}{\xi^2}}=\frac{2}{\xi}\int^{\infty}_{\frac{1+\lambda^*(\xi)}{\xi}} e^{-t^2}dt$ which is identical to the first line in \eqref{xipos}. The case $\xi<0$ could be handled in the similar fashion.
\end{proof}

From Proposition \ref{discreteprop} item (1) evidently, we have that the second order Taylor expansion $\lambda_2(\xi)$ of $\lambda^*$ about $0$ is
\be\label{l2}
\lambda_2(\xi)= - \frac12 \xi^2.
\ee

\section{The adjoint generalized Fourier transforms}\lb{ubadjoint}

For general interest, we give here an explicit description of $\mathcal{B}^*_\xi$ and $\mathcal{U}^*_\xi$ defined in \eqref{intfor}. We introduce the following rigging of the spaces:
\begin{equation*}
\Phi'\subset H':=L^2_w(\bbR; d\lambda)\subset(\Phi')^*,
\end{equation*} 
which is identical to the previous rigging \eqref{rigging}, but this time it is defined with respect of $\lambda$.
\begin{proposition}\lb{adjointfor}
	Let $\hat{\mathcal{S}}_\pm:L^2(\bbR; dv)\to L^2(\bbR; d\lambda)$, $\mathcal{U}_\xi$ and $\mathcal{B}_\xi$  be as in \eqref{hilbert}, \eqref{transf} and \eqref{transf1}, respectively.  Then,
	\begin{itemize}
		\item $\hat{\mathcal{S}}^*_\pm:L^2(\bbR; d\lambda)\to L^2(\bbR; dv)$ and $(\hat{\mathcal{S}}^*_\pm f)(v)=-\lim_{\eta\to0^\pm}\int_{\mathbb{R}}\frac{f(\lambda)}{\lambda-(v+i\eta)}d\lambda,\,\,f\in L^2(\bbR; d\lambda)$. From now on, we will treat $\hat{\mathcal{S}}_\pm$ and ${\mathcal{S}}_\pm$ as operators from $L^2(\bbR; d\lambda)$ to $L^2(\bbR; dv)$ and from $L^2_w(\bbR; d\lambda)$ to $L^2_w(\bbR; dv)$, respectively. Hence, $\hat{\mathcal{S}}^*_\pm=-\hat{\mathcal{S}}_\pm$ and ${\mathcal{S}}^*_\pm=-{\mathcal{S}}_\pm$. 
		
		\item For $\xi\neq0,\pm\sqrt{\pi}$\begin{align*}
		\begin{split}
		\mathcal{U}^*_\xi&:\, H^s_w(\bbR; d\lambda) \to H^s_w(\bbR; dv),\\
		(\mathcal{U}^*_\xi f)(v)&=f(v)-\frac{1}{i\xi} (\tilde S_-f)(v),\,\,\,\hbox{where}\,\,\,(\tilde S_-f)(v):=(S_-(\frac{fw^{1/2}}{\omega_+}))(v), \,\,\\
		&\omega_+(\lambda):=\omega_+(-1-i\lambda\xi,\xi)\,\, \hbox{and} \,\,f\in H^s_w(\bbR; d\lambda),\\
		\mathcal{B}^*_\xi&:\, H^s_w(\bbR; d\lambda) \to H^s_w(\bbR; dv),\\
		(\mathcal{B}^*_\xi f)(v)&=f(v)+\frac{1}{i\xi} (\tilde S_+f)(v),\,\,\,\hbox{where}\,\,\,(\tilde S_+f)(v):=(S_+(\frac{fw^{1/2}}{\omega_-}))(v), \,\,\\
		&\omega_-(\lambda):=\omega_-(-1+i\lambda\xi,-\xi)\,\, \hbox{and} \,\,f\in H^s_w(\bbR; d\lambda).
		\end{split}
		\end{align*}
		\item For $\xi=\pm\sqrt{\pi}$\begin{align*}
		\begin{split}
		(\lambda\mathcal{U}_\xi)^*&:\, H^s_w(\bbR; d\lambda) \to H^s_w(\bbR; dv),\\
		((\lambda\mathcal{U}_\xi)^* f)(v)&=vf(v)-\frac{1}{i\xi} (\tilde S_-(\lambda f))(v),\,\,\,\hbox{where}\,\,\,(\tilde S_-(\lambda f))(v):=(S_-(\frac{\lambda fw^{1/2}}{\omega_+}))(v), \,\,\\
		&\omega_+(\lambda):=\omega_+(-1-i\lambda\xi,\xi)\,\, \hbox{and} \,\,f\in H^s_w(\bbR; d\lambda),\\
		\mathcal{B}^*_\xi&:\, H^s_w(\bbR; d\lambda) \to H^s_w(\bbR; dv),\\
		(\mathcal{B}^*_\xi f)(v)&=f(v)+\frac{1}{i\xi} (\tilde S_+f)(v),\,\,\,\hbox{where}\,\,\,(\tilde S_+f)(v):=(S_+(\frac{fw^{1/2}}{\omega_-}))(v), \,\,\\
		&\omega_-(\lambda):=\omega_-(-1+i\lambda\xi,-\xi)\,\, \hbox{and} \,\,f\in H^s_w(\bbR; d\lambda).
		\end{split}
		\end{align*}
		
	\end{itemize}
	\begin{proof}
		The first item in the statement of Proposition \ref{adjointfor} is obvious. Now, we derive a formula for $U_\xi^*$.
		Let $d\in L^2_w(\bbR; dv)$ and $f\in L^2_w(\bbR; d\lambda)$. Then
		\begin{align*}
		\begin{split}
		(U_\xi d,f)_{L^2_w(\bbR; d\lambda)}&=\int_\bbR d(\lambda)\overline{f(\lambda)}w(\lambda)d\lambda-\int_\bbR\frac{1}{i\xi\overline{\omega_+(-1-i\lambda\xi,\xi)}}({\mathcal{S}}_-(dw^{1/2}))(\lambda)\overline{f(\lambda)}w(\lambda)d\lambda\\
		&=\int_\bbR d(v)\overline{f(v)}w(v)dv-\frac{1}{i\xi}\int_\bbR\frac{1}{\overline{\omega_+(-1-i\lambda\xi,\xi)}}({\hat{\mathcal{S}}}_-(dw))(\lambda)\overline{f(\lambda)}w(\lambda)d\lambda\\
		&=\int_\bbR d(v)\overline{f(v)}w(v)dv-\frac{1}{i\xi}\int_\bbR({\hat{\mathcal{S}}}_-(dw))(\lambda)\overline{\Big[\frac{{f(\lambda)}w(\lambda)}{{\omega_+(-1-i\lambda\xi,\xi)}}\Big]}d\lambda\\
			\end{split}
	\end{align*}

\begin{align*}
\begin{split}		
		&=\int_\bbR d(v)\overline{f(v)}w(v)dv+\frac{1}{i\xi}\int_\bbR(dw)(v)\overline{ (\tilde S_-f)(v)}dv\\
		&=\int_\bbR d(v)\overline{[f(v)-\frac{1}{i\xi} (\tilde S_-f)(v)]}w(v)dv,
		\end{split}
		\end{align*}
		where $(\tilde S_-f)(v):=(S_-(\frac{fw^{1/2}}{\omega_+}))(v)$ and $\omega_+(\lambda):=\omega_+(-1-i\lambda\xi,\xi)$. Similarly, one can derive the remaining formulas.
	\end{proof}
\end{proposition}

\begin{theorem}
	Let $\xi\neq0$, $v\in\bbR$ and $\phi\in\Phi'$. Then
	\begin{align*}
	\begin{split}
	{(\mathcal{U}^*_\xi \phi)(v)}&=\frac{1}{w^{1/2}(v)}\langle\phi,\delta^+_{v}\rangle,\\
	(\mathcal{B}^*_\xi \phi)(v)&=\frac{1}{w^{1/2}(v)}\langle\phi,\delta^{a-}_{v}\rangle,
	\end{split}
	\end{align*}
	where
	\begin{align*}
	\begin{split}
	\delta^+_{v}(\lambda)&=-\frac{1}{2\pi i}\frac{1}{ v-\lambda}w^{-1/2}(\lambda)+\frac{w^{1/2}(v)w^{-1/2}(\lambda)}{2\pi\xi{\omega_+(-1-i\lambda\xi,\xi)}}\frac{1}{ v-\lambda}[{({\mathcal{S}}_+\mathbb{1})(v)}-{({\mathcal{S}}\mathbb{1})(\lambda)}],\\
	\delta^{a-}_{v}(\lambda)&=-\frac{1}{2\pi i}\frac{1}{ v-\lambda}w^{-1/2}(\lambda)-\frac{w^{1/2}(x)w^{-1/2}(\lambda)}{2\pi\xi{\omega_-(-1+i\lambda\xi,-\xi)}}\frac{1}{ v-\lambda}[{({\mathcal{S}}_-\mathbb{1})(v)}-{({\mathcal{S}}\mathbb{1})(\lambda)}],
	\end{split}
	\end{align*}
	 which can be treated as holomorphic functions of $\lambda$ outside of the certain strip. 
\end{theorem}
\begin{proof}
	It is similar to the proof of Proposition \ref{FT}.
\end{proof}


\end{document}